\documentclass[12pt]{article}

\usepackage{setspace}
\usepackage{amsmath,amssymb,amsthm,amsfonts}

\usepackage[pdftex,xdvi]{graphicx}
\usepackage{epstopdf}
\usepackage{rotating}
\usepackage[colorlinks, citecolor=blue, urlcolor=blue, pdfborder={0 0 1}, citebordercolor={1 1 1}]{hyperref}
\usepackage{booktabs}
\usepackage{longtable}
\usepackage{array}
\usepackage{multirow}
\usepackage{wrapfig}
\usepackage{float}
\usepackage{colortbl}
\usepackage{pdflscape}
\usepackage{tabu}
\usepackage{threeparttable}
\usepackage{threeparttablex}
\usepackage[normalem]{ulem}
\usepackage{makecell}
\usepackage{xcolor}
\usepackage{subcaption}
\usepackage{morefloats}

\usepackage{natbib}

\usepackage{enumerate}

\numberwithin{equation}{section}

\theoremstyle{definition}

\newtheorem{remark}{Remark}

\theoremstyle{plain}
\newtheorem{theorem}{Theorem}[section]
\numberwithin{theorem}{section} 

  \newtheorem{proposition}[theorem]{Proposition}
\newtheorem{lemma}[theorem]{Lemma}
\newtheorem{definition}[theorem]{Definition}
\newtheorem{assumption}{Assumption}[section]
\newtheorem{corollary}{Corollary}[section]

\DeclareMathOperator{\rank}{rank}

\DeclareMathOperator{\tr}{tr}

\newcommand{\R}{\mathbb{R}}

\newcommand{\reals}{\mathcal{R}}

\renewcommand{\Pr}{{\mathrm{P}}}

\newcommand{\vhat}{\widehat{v}}

\newcommand{\E}{\mathbb{E}}

\newcommand{\lambdahat}{\widehat{\lambda}}
\newcommand{\Xihat}{\widehat{\Xi}}

\newcommand{\Vhat}{\widehat{V}}
\newcommand{\Lhat}{\widehat{L}}

\newcommand{\xihat}{\widehat{\xi}}

\usepackage[a4paper,margin=1in,textheight=9.5in]{geometry}

\usepackage[small,compact]{titlesec}

\def\spacingset#1{\renewcommand{\baselinestretch}%
{#1}\small\normalsize} \spacingset{1}
\allowdisplaybreaks

\begin{document}
\onehalfspacing
\title{\large\bf{Fixed-order PCA: Theory for Overestimated Factor Models
}}
\author{
Yuan Liao\\  \footnotesize Department of Economics \\ \footnotesize University  of Iowa\and
Xin Tong\\  \footnotesize Department of Mathematics \\ \footnotesize National  University  of Singapore
\and  Wanjie Wang \\  \footnotesize Department of Statistics\\ \footnotesize  National  University  of Singapore
\and  Dacheng Xiu \\  \footnotesize Booth School of Business\\ \footnotesize University of Chicago and NBER}

\date{ \vspace{.3in}  \today}
\maketitle
\begin{abstract} 
We develop asymptotic theory for principal component analysis (PCA) of a high-dimensional factor model in which the working dimension $R$ is fixed and only required to satisfy $R \ge r$, where $r$ is the true number of factors. Building on anisotropic local laws from random matrix theory, we show that the ``extra'' empirical eigencomponents beyond the $r$-th are asymptotically noise-governed, incoherent, and nearly orthogonal to the factor loadings. We introduce two rotations, an expanded $r\times R$ map $H'$ and a compressed $R\times r$ map $H^{+}$, and establish consistency of the estimated factors under both. As an application, we analyze a factor-augmented regression for treatment-effect inference and prove $\sqrt{T}$-asymptotic normality for every fixed $R \ge r$. These results provide a theoretical underpinning for the common empirical practice of adopting a conservative upper bound on the number of factors, and shift the analytical burden from consistent dimension selection to the milder requirement of bounding $r$ from above.

\vspace{.1in}

\end{abstract}

\newpage

%
 \onehalfspacing

\section{Introduction} \label{sec:intro}

Factor models provide a parsimonious description of high-dimensional data by decomposing an observation matrix $X \in \R^{N\times T}$ as
\begin{equation}\label{eq:intro-model}
X = B F' + U,
\end{equation}
where $B \in \R^{N\times r}$ is the loading matrix, $F \in \R^{T\times r}$ collects the latent factors, and $U \in \R^{N\times T}$ collect the idiosyncratic errors. Such models have become central in economics, finance, genomics, and signal processing, and principal component analysis (PCA) is by far the most widely used estimator. The asymptotic theory developed by  {\cite{connor1986performance,bai03}} and others typically rests on a critical prerequisite: the investigator must first consistently estimate the true number of factors $r$. Under strong-factor assumptions and clear eigenvalue gaps, a variety of methods, such as the information-criteria-based procedure \citep{BN02} and the eigenvalue-ratio test \citep{AH}, can consistently recover $r$. In practice, however, the signal is rarely so clean. In such settings, these methods can be unstable and may fail to detect weak factors while also selecting spurious ones. As a result, downstream inference may be adversely affected, particularly when the number of factors is underestimated.

Motivated by this fragility, 
empirical researchers often proceed by selecting a working number of factors 
$R$ that is intended to exceed the true 
$r$, and then base estimation and inference on this potentially over-specified model. Variants of this approach appear widely in the applied literature, with $R$ typically chosen by rule of thumb or guided by domain knowledge. Despite its prevalence, the statistical properties of this practice remain largely uncharacterized. This paper provides a theoretical foundation for such over-specification.

\subsection*{Contributions}
We develop a unified spectral theory for \emph{fixed-order} PCA, wherein the working dimension $R$ is an arbitrary, user-specified integer satisfying $R \ge r$. Our contributions are threefold.

\emph{(1) Spectrum and eigenvectors beyond the true factor dimension.} Building on anisotropic local laws from random matrix theory \citep{knowles2017anisotropic}, we characterize the empirical singular values $\lambda_{r+1}(X), \ldots, \lambda_{R}(X)$ and their associated singular vectors. We show that the ``extra'' singular values track those of the noise matrix $U$ up to a sharp remainder, and that the overestimated singular vectors are incoherent (also known as delocalized) with respect to any direction independent of $U$ and are near-orthogonal to the factor space. {Our random-matrix-based analysis achieves a strictly sharper alignment rate between the overestimated principal-component directions and the factor loadings than the standard Davis--Kahan / $\sin\Theta$ benchmark.} These facts provide the technical basis for the rest of the paper.

\emph{(2) Estimation of the full factor space.} Because the fitted factor matrix $\widehat F$ has $R$ columns rather than $r$, the usual notion of a rotation matrix must be extended. We introduce two rotations, an \emph{expanded} $r \times R$ rotation $H$ and a \emph{compressed} $R \times r$ generalized inverse $H^{+}$, and establish convergence rates for $\widehat F - FH'$ and $\widehat F H^{+} - F$. The former is faster because it only requires $F$ to lie in the column space of $\widehat F$, whereas the latter must further extract each of the $r$ factor directions from the $R$-dimensional fitted space.

\emph{(3) Robust factor-augmented inference.} As an illustrative application, we study treatment-effect estimation in a factor-augmented regression, where the number of latent confounders is unknown and consistent selection of $r$ is not invoked. We show that $\sqrt{T}$-asymptotic normality of the resulting estimator persists for every fixed $R \ge r$ when $r\ge 1$. Adding extra principal components may increase residual variation but does not affect the signal at first order, because the extra components are asymptotically orthogonal to the factor space and, since $R - r$ is fixed, their inclusion inflates the variance by {an asymptotically negligible factor of $1+O(1/T)$}.

These results imply that, within the asymptotic regime considered here, overestimating $r$ by a bounded amount is {asymptotically harmless}, whereas underestimation can have first-order consequences. Fixed-order PCA therefore provides a conservative and transparent procedure.

\subsection*{Related Literature}

\emph{Robustness to overspecification of the factor number.} The consequences of overspecifying factor dimension have been examined in several settings. \cite{MW11} showed that inference in interactive-effects models remains valid when the number of factors is overspecified, using perturbation theory for linear operators \citep{Kato95}. They established conditions under which the inclusion of overestimated eigenvectors does not affect the first order behavior of the parameter of interest.   Their operator-perturbation route is more flexible across estimator families, while the random-matrix-based analysis that we adopt in this paper specializes to PCA but yields sharper rates when applicable (e.g., Theorem~\ref{thm:main}(iii)).
In addition, we allow the factors to be much weaker in the sense that the top singular values of $X$ can grow much slower than $NT$, which is a   strictly broader regime than the strong-factor scaling employed in \cite{MW11}. 
 \cite{barigozzi2020consistent} established consistency of the low-rank component under overestimation through a trimmed PCA estimator that enforces incoherence by post-processing. \cite{fan2022learning} and \cite{choi2023inference} propose diversified-projection estimators that are robust to overspecification; their approach requires prespecified weights that are independent of the idiosyncratic noise, typically obtained from external characteristics or a held-out sub-sample. Our analysis shows that, under the setting of~\eqref{eq:intro-model} together with the incoherence conditions of \cite{knowles2017anisotropic}, vanilla PCA is itself robust to overspecification, with no trimming, weighting, or external information required; the incoherence that those papers manufacture by post-processing or by external weights is here a generic property of noise-side singular vectors, supplied by the local law.

\emph{Entrywise eigenvector perturbation.}  A parallel methodological line \citep{fan2018ell, cape2019two, abbe2020entrywise,fan2021recent} develops entrywise (sup-norm) perturbation bounds for the leading singular subspace of low-rank, incoherent signal matrices, sharpening the Davis--Kahan $\sin\Theta$ rate by exploiting incoherence of the signal.   Importantly, their results only apply to the  first $r$ eigenvectors. In contrast,  the setting $R>r$ requires a different apparatus: the empirical eigencomponents that govern overspecification lie inside the noise bulk, where no spectral gap is available, so the deterministic perturbation arguments developed in this literature for the leading subspace do not directly apply. We draw instead on the anisotropic local law of \cite{knowles2017anisotropic}, which supplies the analogous entrywise control of noise singular vectors in this no-gap regime; this random-matrix-based control is the natural counterpart of the entrywise perturbation bounds derived in the cited works, and is what enables our finer analysis of the extra components.

\emph{Spectral methods with weak or mixed signals.} More broadly, our results complement a growing literature on the spectrum of factor and spiked-covariance models with weak or mixed signal strengths \citep{onatski2010determining, onatski2012asymptotics, wang2017asymptotics, freyaldenhoven2022factor, uematsu2022estimation, bai2023approximate, giglio2023prediction}. 
Recently, \cite{bykhovskaya2025weak}   developed a unified random-matrix framework for inference on factor strength based on the top $r$ sample eigenvalues associated with the factors, uniformly across strong, weak, and critical regime. Our theoretical results are developed under a   regime  in which the signal strength is required only to diverge with the sample size and is not constrained to the strong-factor scaling of \cite{bai03}; this places our analysis squarely within the asymptotic setting considered by these works and makes the resulting spectral characterizations directly relevant to the inferential and prediction problems they study. Relative to that literature, we provide a single framework that simultaneously handles signal and noise eigencomponents in an overestimated spectrum, with explicit rates that readily plug into other factor-model inference problems. 
Our approach takes a different route than the factor-number selection
 \cite{BN02,onatski2010determining, ahn2013panel}: rather than asking when these procedures recover $r$ exactly, we ask what inference can be performed given an arbitrary upper bound, which is the practically relevant object once the analyst recognizes that a sharp $\widehat r$ is unavailable. 

\emph{Connection to double machine learning.} The factor-augmented inferential application we develop in Section~\ref{secapp} sits within the partialling-out /  {double machine learning (DML)} tradition \citep{belloni2014inference, chernozhukov2016double, hansen2018fac}, in which a low-dimensional treatment effect is estimated from a Neyman-orthogonal moment after residualizing both the outcome and the treatment with respect to high-dimensional nuisance components. Two features distinguish our setting. First, the nuisance is the \emph{latent} factor space spanned by $F$, recovered by PCA from the panel $X$ rather than  {constructed by selecting or fitting a function of observed controls; the two nuisances of canonical DML --- the conditional expectations of the outcome and the treatment given the controls --- share a common latent factor space in our setting, so a single PCA step suffices to residualize both equations, and the ``double'' structure of DML survives in the residualization rather than in the ML estimation step}; the standard  {DML} ``product of nuisance rates'' condition is replaced by  {a factor-strength condition stated explicitly in Assumption~\ref{ass4.1}(ii)}. Second,  {in the factor-augmented regression of Section~\ref{secapp}, a structural independence assumption between the regression errors and the factor--noise pair $(F,U)$} supplies the conditional independence that cross-fitting is engineered to deliver in the i.i.d.\ setting; consequently, neither sample splitting nor data-driven tuning of the working dimension $R$ is required, and any fixed $R\ge r$ with $R-r$ bounded yields valid $\sqrt{T}$ inference.

\subsection*{Organization and Notation}

The remainder of the paper is organized as follows. Section~\ref{secpca} introduces the model, fixed-order PCA, and our main spectral results. Section~\ref{secapp} applies these results to inference on a treatment coefficient in a factor-augmented regression. Section~\ref{secsim} reports simulation evidence. Section~\ref{secemp} presents an empirical application. Section~\ref{secconcl} concludes. All proofs are collected in the appendices.

Throughout the paper, $\lambda_k(A)$ denotes the $k$-th largest singular value of a matrix $A$, $\|A\|$ its spectral norm, and $\|A\|_{\mathrm{F}}$ its Frobenius norm. For a full-column-rank matrix $A$, $P_A = A(A'A)^{-1}A'$; for an arbitrary matrix, $A^{+}$ denotes its Moore--Penrose pseudoinverse and $P_A = A A^{+}$. We write $a_n \asymp b_n$ if $c\,b_n \le a_n \le C\,b_n$ for some $0 < c \le C < \infty$, and $a_n \ll b_n$ (equivalently, $a_n = o(b_n)$) if $a_n/b_n \to 0$. 
We write $\Phi(T) \prec \phi(T)$ (\emph{stochastic domination up to subpolynomial factors}) if $\Pr\{\Phi(T) > T^{\varepsilon}\, \phi(T)\} \to 0$ for every fixed $\varepsilon>0$; equivalently, $\Phi(T)=O_P(T^{\varepsilon}\phi(T))$ for every fixed $\varepsilon>0$. Thus $\prec$ agrees with $O_P$ up to a factor of size $T^{o(1)}$, and in particular absorbs powers of $\log T$. By itself, $\Phi\prec\phi$ with $\phi\to0$ need not imply $\Phi=o_P(1)$. Accordingly, we infer $\Phi=o_P(1)$ from a bound $\Phi\prec\phi$ only when an explicit rate condition guarantees $T^\varepsilon\phi(T)\to0$ for some $\varepsilon>0$. We use $c$ and $C$ for generic positive constants whose values may change across displays, and $\varepsilon$ for an arbitrarily small positive constant.

\section{Fixed-Order PCA: Model and Theory}\label{secpca}

\subsection{Factor Model Setup}\label{secmodel}

We consider the static approximate factor model in which, for each time period $t = 1, \ldots, T$, we observe an $N \times 1$ vector  {$x_t$} satisfying
\[
 {x_t} = B f_t + u_t.
\]
Here $B$ is an $N \times r$ matrix of factor loadings, $f_t$ is an $r \times 1$ vector of latent common factors, and $u_t$ is an $N \times 1$ vector of idiosyncratic components. We assume throughout that the factor process $\{f_t\}_{t\le T}$ and the idiosyncratic process $\{u_t\}_{t\le T}$ are mutually independent and that both have mean zero. 

 Let  {$X = (x_1, \ldots, x_T)$}, $F = (f_1, \ldots, f_T)'$, and $U = (u_1, \ldots, u_T)$. Then $X$ admits the matrix representation
\[
X = M + U, \qquad M = B F',
\]
where $M \in \R^{N \times T}$ is the low-rank signal matrix and $U \in \R^{N \times T}$ collects idiosyncratic errors.

We model idiosyncratic components as $u_t = \Sigma_e^{1/2} e_t$, where $e_t$ is an $N \times 1$ vector of standardized shocks and $\Sigma_e$ is an $N \times N$ positive-definite matrix. In matrix form,
\[
U = \Sigma_e^{1/2} E, \qquad E = (e_1, \ldots, e_T),
\]
and we denote by $\sigma_1, \ldots, \sigma_N$ the eigenvalues of $\Sigma_e$. {The matrix $\Sigma_e$ is allowed to be a general positive definite matrix subject to the spectral regularity conditions in Assumption~\ref{aspt:D0}(ii) below: 
 {its eigenvalues are bounded  and the associated deformed Marchenko--Pastur law is regular.}
Together these conditions accommodate cross-sectional heteroskedasticity and weak cross-sectional dependence in the idiosyncratic components, but rule out approximate factor-like structures within $u_t$ that would create outlying spikes in the spectrum of $\Sigma_e$; the latter would correspond to additional, undetected factors and is best modelled by enlarging $r$.}

 {Two features of this setup deserve emphasis. First,   $B$ is treated as deterministic and $F$ as random throughout; statements such as ``$\max_{i \le N}\|b_i\|=O(1)$'' below should be read as deterministic bounds uniform in $N$. The factors, in contrast, are random and possibly unbounded: Assumption~\ref{aspt:X}(i) below imposes $\max_{t\le T}\|f_t\|=O_P(\log T)$, which accommodates any factor process whose coordinates have uniformly subexponential tails (see the discussion following Assumption~\ref{aspt:X}).
 Second, all dependence in $U$ is channeled through the deterministic operator $\Sigma_e^{1/2}$ acting on entries $e_{i,t}$ that are independent across both $i$ and $t$ (Assumption~\ref{aspt:D0}(i)). This is the standard random-matrix setting in which the anisotropic local law of \cite{knowles2017anisotropic} applies, but it is more restrictive than the classical approximate-factor-model framework of \cite{bai03}, which permits weak \emph{serial} as well as cross-sectional dependence in $u_t$.}  {We emphasize this trade-off explicitly:  on one hand, both cross-sectional and serial independence are required by the local-law we use to characterize the spectral theory of the extra eigenvectors. On the other hand,  a more general HAC-type extension to weakly serially-dependent $u_t$ is the most empirically pressing direction for future work, which we list accordingly in Section~\ref{secconcl}. The applications for which our setup is best suited are therefore those in which independence across $(i,t)$ is plausible by sampling design --- repeated cross-sections, randomized rollouts, large-panel snapshots in survey data. Meanwhile, in this paper  the factor process $f_t$ is indeed allowed to be serially dependent. }  

Because $B$ and $F$ are identified only up to a rotation, we fix a canonical normalization tied to the singular vectors of $M$. Let $M = \Xi_r L_r V_r'$ denote the SVD of $BF'$, write $S_B = N^{-1}B'B$ and $S_f = T^{-1}F'F$, and let $H_B,H_F\in\R^{r\times r}$ be the rotation matrices that simultaneously diagonalize $S_B^{1/2}S_f S_B^{1/2}$ and $S_f^{1/2} S_B S_f^{1/2}$, normalized so that
\begin{equation}\label{identBF}
\frac{1}{\sqrt{N}}\,B \;=\; \Xi_r\,H_B^{-1},\qquad \frac{1}{\sqrt{T}}\,F \;=\; V_r\,H_F^{-1},\qquad (H_F'H_B)^{-1} \;=\; (NT)^{-1/2}\,L_r.
\end{equation}
The construction of $H_B$ and $H_F$ from $(S_B,S_f)$ and the verification of~\eqref{identBF} are entirely algebraic; the explicit formulas and computation are deferred to Section~\ref{secrotation}. The identities in~\eqref{identBF} are the only consequences used in the body of the paper.

\subsection{Fixed-Order PCA}\label{secfpca}

Throughout, the true number of latent factors $r = \dim(f_t) \ge 0$ is assumed to be fixed and bounded. When factors are not strong (the singular values of $M$ grow slower than $\sqrt{NT}$), consistently estimating $r$ typically requires stringent technical conditions and performs poorly in finite samples; information criteria and eigenvalue-ratio methods are well documented to under- or over-estimate $r$ in finite sample, depending on signal strength.

To avoid this instability, we adopt \emph{fixed-order PCA}. Rather than attempting consistent recovery of $r$, the statistician specifies an integer $R$ (the \emph{working number of factors}) and imposes only
\[
R \ge r.
\]
The estimator is then constructed from the first $R$ empirical singular components of $X$, regardless of whether $R$ matches the true factor dimension.

{The requirement $R\ge r$ is still substantive: one needs a credible upper bound on the factor dimension, supplied for example by a conservative screen from information criteria, or by scientific constraints on the number of latent channels. When no such upper bound is available, fixed-order PCA should be viewed as a sensitivity analysis over a small grid of plausible $R$'s rather than as a replacement for factor-number learning.}

{From a machine-learning perspective, the working dimension $R$ is the tuning parameter of fixed-order PCA. The main results of this section can therefore be read as a robustness-to-tuning-parameter statement: valid inference holds for any fixed $R$ within a range bounded below by the unknown $r$ and bounded above by a constant, without data-driven optimization of $R$.} 

Write the singular value decomposition of $X$ as
\[
X = \widehat \Xi_R \widehat L_R \widehat V_R' + \widehat \Xi_{-R} \widehat L_{-R} \widehat V_{-R}',
\]
where $\widehat \Xi_R \in \R^{N \times R}$, $\widehat V_R \in \R^{T \times R}$, and $\widehat L_R \in \R^{R \times R}$ collect the top-$R$ left singular vectors, right singular vectors, and singular values, and $(\widehat \Xi_{-R}, \widehat L_{-R}, \widehat V_{-R})$ collect the remaining components. The PCA estimators of loadings and factors are
\[
\widehat B = \sqrt{N}\,\widehat \Xi_R, \qquad \widehat F = \frac{1}{\sqrt{N}}\, X' \widehat \Xi_R = \frac{1}{\sqrt{N}}\,\widehat V_R \widehat L_R,
\]
and the low-rank component is estimated by singular-value thresholding,
\[
\widehat M = \widehat B \widehat F' = \widehat \Xi_R \widehat L_R \widehat V_R' = X \widehat V_R \widehat V_R'.
\]
In particular, $\widehat S_f := T^{-1} \widehat F'\widehat F = (TN)^{-1} \widehat L_R^{2}$; since the top $R$ singular values of $X$ are positive almost surely under Assumption~\ref{aspt:D0},   $\widehat S_f$ is invertible (almost surely).

A central contribution of this paper is to show that fixed-order PCA still consistently recovers the factor space spanned by $B$ even when $R$ strictly exceeds $r$. The additional empirical eigencomponents (those beyond the true factor dimension) do not contaminate the recovery of the true factor space; instead, they converge to well-characterized noise-governed directions. The leading $r$ components of fixed-order PCA asymptotically reconstruct the true factor space, while the remaining $R-r$ components behave in a controlled and predictable manner.

\subsection{Asymptotic Behavior of the Extra Eigencomponents}\label{secextra}

The main objective of this section is to establish the consistency of PCA when $R > r$. We partition the empirical eigencomponents as
\[
\widehat \Xi_R = (\widehat\Xi_r, \widehat\Xi_{-r}), \qquad \widehat V_R = (\widehat V_r, \widehat V_{-r}),
\qquad \widehat L_R = \mathrm{diag}\bigl(\widehat L_r, \widehat L_{-r}\bigr),
\]
where $\widehat\Xi_r$ denotes the leading $r$ left singular vectors, the usual \emph{spiked} eigenvectors, and $\widehat\Xi_{-r}$ collects the remaining $R-r$ eigenvectors, which we refer to as the \emph{extra} (or overestimated) eigenvectors. When $R = r$, we set $\widehat\Xi_{-r} = \emptyset$. The analogous decomposition applies to $\widehat V_R$ and to the singular values: writing $\widehat\lambda_l:=\lambda_l(X)$, the block $\widehat L_{-r}=\mathrm{diag}(\widehat\lambda_{r+1},\ldots,\widehat\lambda_R)$ collects the extra singular values.

The behavior of the extra eigenvectors is the key obstacle: $\widehat\Xi_{-r}$ carries no factor signal, and its asymptotic effect depends delicately on the noise $U$.  For PCA-based inference to remain valid under overspecification, $\widehat\Xi_{-r}$ must be incoherent, in the sense that its entries spread approximately uniformly across coordinates rather than concentrating on a few positions; coherent eigenvectors (also known as localized eigenvectors), in contrast, would interact spuriously with deterministic directions of interest. We establish this incoherency for $\widehat\Xi_{-r}$ as Theorem~\ref{thm:main}(ii) below.

\begin{assumption}\label{aspt:D0}
The idiosyncratic shocks and their cross-sectional covariance satisfy:
\begin{enumerate}
\item[(i)] \text The entries $e_{i,t}$ of $e_t$ are independent random variables with $\E[e_{i,t}] = 0$ and $\E[e_{i,t}^2] = 1$. In addition, they are uniformly subexponential: for a constant $C$ independent of $(i,t,N,T)$,
\[
\sup_{i,t}\|e_{i,t}\|_{\psi_1}\le C,\qquad
\|e_{i,t}\|_{\psi_1}:=\inf\{K>0: \mathbb E\exp(|e_{i,t}|/K)\leq 2\}.
\]


\item[(ii)] 
The eigenvalues of $\Sigma_e$ lie in $[c,C]$ for constants $0<c\le C<\infty$ that do not depend on $N$, and their associate deformed Marchenko--Pastur (MP) has \textit{regular edges and bulks}, as defined in Definition \ref{def:regularity} in Appendix \ref{sec:mp}.
 \end{enumerate}
\end{assumption}

We require the noise be subexponential-tailed, as defined in \cite{vershynin2020high}. Condition (ii)  allows us to develop a local-law of the spectrum of $X$ for  factor models. 
The MP law   of $U$ can be  determined using its Stieltjes transform. Assuming it has regular edges and bulks,  we show that  the analysis of $X$ can proceed  by leveraging the
eigenvalue rigidity, edge spacing, and anisotropic incoherence properties of $U$ established
in \cite{knowles2017anisotropic}. In part (ii) of this assumption,  the required regularity on  edges and bulks for the   eigenvalues  of $\Sigma_e$ are standard,  and   we defer the detailed definition to Definition \ref{def:regularity}.  
Here are two examples to satisfy this condition. Let the ordered eigenvalues 
 of  $\Sigma_e$ be $\sigma_1\geq \sigma_2\geq...\geq\sigma_N$, and their 
empirical spectral distribution  be 
$\pi = \frac{1}{N}\sum_{i=1}^N \delta_{\sigma_i}$, where $\delta_{\sigma_i}$ denotes the   Dirac measure at $\sigma_i$.
 
\begin{enumerate}
    \item \textbf{Discrete limit.} The measure $\pi$ is supported on $n$ fixed values
    $\{s_k\}_{k=1}^{n}$, and $\pi(s_k)$ converges to a limit as $T \to \infty$.
    \item \textbf{Continuous limit.} There exists a measure $\pi_\infty$ whose support is an interval $[a, b]$ with $0<a\le b<\infty$, and whose density $p_\infty$ is bounded above and below on it: $\kappa \le p_\infty(x) \le \kappa^{-1}$ for all $x \in [a,b]$ and some constant $\kappa > 0$; and $\pi$ converges
    weakly to $\pi_\infty$ as $T \to \infty$. 
\end{enumerate}
Therefore, Assumption~\ref{aspt:D0}(ii)  is satisfied in standard examples: the homoskedastic case $\Sigma_e=I\sigma$, sparse perturbations of the identity, and $\Sigma_e$ with bounded condition number whose limiting spectral measure has positive density.


\begin{assumption}\label{aspt:X}
The loadings, factors, and signal strength satisfy:
\begin{enumerate}
\item[(i)]  {The loadings $B$ are deterministic with row norms uniformly bounded in $N$, and the factors $F$ are random with row norms bounded in probability up to a logarithmic factor:}
\[
\max_{i \le N} \|b_i\|  {\le C}, \qquad \max_{t \le T} \|f_t\| = O_P(\log T).
\]
\item[(ii)] $N/T \to \phi$ for some $\phi \in (0, \infty)$ with $\phi \ne 1$.
\item[(iii)] There exists a sequence $\nu_M \to \infty$ with $\nu_M = O(\sqrt{N})$ such that, with probability approaching one as $T\to\infty$,
\[
c \nu_M \sqrt{T} \le \lambda_r(M) < \cdots < \lambda_1(M) \le C \nu_M \sqrt{T},
\]
and
\[
c \le \lambda_r(S_f) \le \cdots \le \lambda_1(S_f) \le C, \qquad S_f := \frac{1}{T} F'F.
\]
\end{enumerate}
\end{assumption}

The $O_P(\log T)$-bound in Part~(i) holds whenever the coordinates of $f_t$ have uniformly bounded subexponential norms, since the maximum of $T$ subexponential random variables is $O_P(\log T)$; Gaussian factors satisfy the stronger bound $O_P(\sqrt{\log T})$. Because $F$ is independent of the noise $E$, the random-matrix estimates in the proofs are applied conditionally on $F$, on the event that the bounds in parts~(i) and~(iii) hold (see the preliminaries in Appendix~\ref{sec:prelim}); the logarithmic factor from part~(i) is absorbed by every bound stated in terms of stochastic domination $\prec$, since $\log T = o(T^{\varepsilon})$ for every $\varepsilon>0$, and surfaces explicitly only in the entrywise bounds underlying the inference theory of Section~\ref{secapp} (see Assumption~\ref{ass4.1}(ii) and Step~4 of the proof of Theorem~\ref{thinf}).

The restriction $\phi \ne 1$ is inherited from \cite{knowles2017anisotropic}: the anisotropic local law guarantees incoherence of the singular vectors of $U$ only when the aspect ratio is bounded away from the MP edge.  {We do not claim that the conclusions fail at or near $\phi=1$; rather, the available local-law input used here does not provide the required singular-vector incoherence there.} To our knowledge, the extent to which this is a proof artifact versus a genuine phase transition for incoherence remains open. Recent work on the MP edge in the regime $\phi\to 1$ (see, e.g., the line of research developed by L.\ Erd\H{o}s and collaborators on rigidity at the soft and hard edges) studied the rate at which $\phi$ converges to one, and suggested  that the natural fluctuation scale changes from the bulk $T^{-1/2}$ to an edge scale of order $T^{-2/3}$; whether this scale is compatible with the entrywise bounds we require for Theorem~\ref{thm:main} is an interesting question for future work, and would be the natural route to closing the $\phi\ne 1$ gap.

Part~(iii) is the factor-strength condition: equivalently, $$\lambda_k(B'B) \asymp \nu_M^{2}, \quad  k = 1, \ldots, r.$$  The sequence $\nu_M$ indexes the signal strength and governs the sharpness of all subsequent results. At one extreme, $\nu_M \asymp \sqrt{N}$ recovers the strong-factor setting of \cite{bai03}; at the other, the results remain informative as long as $\nu_M \to \infty$, which is considerably weaker.

 {The role of $\nu_M$ across the main results is as follows. Separation of the first $r$ singular values from the noise requires only $\nu_M\to\infty$. To turn the subpolynomial-loss bounds $\prec\nu_M^{-1}$ and $\prec\nu_M^{-2}$ into genuine incoherence and near-orthogonality, Theorem~\ref{thm:main} imposes the mild additional requirement that $T^{\tau}=o(\nu_M)$ for some constant $\tau>0$, i.e., that $\nu_M$ diverge at least at some polynomial rate; the same condition supports the factor-space recovery of Theorem~\ref{thm:factor}. Furthermore, the inference theorem strengthens this to the product-of-rates condition $\sqrt{T}\log T=o(\nu_M^{2})$ in Assumption~\ref{ass4.1}(ii), the analogue of the DML rate condition.}

The factor-strength part of Assumption~\ref{aspt:X} is imposed only when $r\ge 1$. Meanwhile we allow a special case  $r=0$, that is,  there is no factor present   so   $X=U$.  In this case, 
 statements about extra components reduce to the corresponding noise-only local-law statements, which we state in Corollary~\ref{coro:rzero} below. This case is interesting in   applications  where statisticians are concerned about the  impact of confounding factors but are not sure whether they are present.  Therefore, allowing $r=0$ as a special case makes the PCA-based inference be also robust to whether confounding factors are present.

The following theorem is the analytical foundation of the paper. It shows that the extra spectrum closely resembles the spectrum of the noise matrix, and that the extra eigenvectors are incoherent and nearly orthogonal to the factor directions, with rates expressed through the stochastic-domination notation $\prec$ defined at the end of Section~1. The theorem is stated for $r\ge 1$, the setting of direct interest for factor-model inference, whereas the  case $r=0$ is recorded as Corollary~\ref{coro:rzero} below.

\begin{theorem}[Extra spectrum, $R > r$]\label{thm:main}
Under Assumptions~\ref{aspt:D0} and~\ref{aspt:X} with $r\ge 1$, suppose additionally that $T^\tau=o(\nu_M)$ for some constant $\tau>0$. Then the bounds in (i)--(iii) below hold uniformly in $k \in \{1, \ldots, R-r\}$.
\begin{enumerate}
\item[(i)] \emph{Singular values.} For any $R > r \ge 1$ and $k = 1, \ldots, R-r$,
\[
\lambda_k(U)^2 \ge \lambda_{r+k}(X)^2, \qquad \lambda_k(U)^2 - \lambda_{r+k}(X)^2 \prec \nu_M^{-2}\, T.
\]
\item[(ii)] \emph{Incoherence of extra eigenvectors.} For any sequence of unit vectors $\{\eta_i, \zeta_t: i\leq N, t\leq T\}$, which are independent of $U$,  $\|\eta_i\|=1$ and $\|\zeta_t\|=1$, $i=1...N$, $t=1,...,T$,  
\[
\max_{i\leq N}\|\eta_i' \widehat\Xi_{-r}\| \prec \nu_M^{-1}, \qquad \max_{t\leq T}\|\zeta_t' \widehat V_{-r}\| \prec \nu_M^{-1}.
\]
\item[(iii)] \emph{Near-orthogonality with factors.}  
\[
\|B\|_{\mathrm{F}}^{-1}\,\| B' \widehat\Xi_{-r}\| \prec \nu_M^{-2}, \qquad \|F\|_{\mathrm{F}}^{-1}\,\|F' \widehat V_{-r}\| \prec \nu_M^{-2}.
\]
\end{enumerate}
\end{theorem}

The theorem admits the following geometric interpretation. The signal matrix $M$ contributes $r$ singular values of order $\nu_M\sqrt{T}$, while the noise matrix $U$ has Marchenko--Pastur spectrum on the order of $\sqrt{T}$. When $\nu_M \gg 1$, the empirical singular spectrum of $X$ exhibits the so-called ``canonical outlier-plus-bulk pattern": $r$ spike outliers at scale $\nu_M\sqrt{T}$, well separated from a Marchenko--Pastur bulk concentrated at scale $\sqrt{T}$. They contain essentially all of the factor information in $X$. In addition, the next $R-r$ empirical singular vectors $\widehat\Xi_{-r}$ of $X$ lie in the orthogonal complement of this signal subspace and are therefore noise-driven; their geometry is governed by the local law for $U$. The contribution of the signal to these extra subspaces is a second-order effect, captured by the $\nu_M^{-2}$ rate, up to a subpolynomial factor, in Theorem~\ref{thm:main}(iii).

Part~(ii) of Theorem~\ref{thm:main} implies $\|\eta'\widehat\Xi_{-r}\| \overset{P}{\to} 0$ for every unit vector $\eta$ independent of $U$, at rate $\nu_M^{-1}$ up to a subpolynomial factor. Indeed, choose $\varepsilon\in(0,\tau)$ in the definition of $\prec$, so that $T^\varepsilon/\nu_M\to0$. This is precisely the incoherence condition: were $\widehat\Xi_{-r}$ sparsely concentrated, say at the first coordinate, then taking $\eta = (1, 0, \ldots, 0)'$ would produce $\eta' \widehat\Xi_{-r} = 1$, contradicting the conclusion. In other words, even though the extra eigenvectors carry no signal, they are diffuse rather than coherent.

This theorem leads to three statistical insights.  First, result (ii) imply the $\ell_{\infty}$ perturbation bounds for $R>r$, by setting $\eta$ and $\zeta$ as canonical basis vectors:
$$
\|\widehat\Xi_{-r}\|_{\infty}= \max_{k}\|\eta_k' \widehat\Xi_{-r}\| \prec \nu_M^{-1}
$$
where $\eta_k'=(0,..,0,1,0...)$, taking value one on its $k$ th element (similarly we have bounds for $\|\widehat V_{-r}\|_{\infty}$). 
 The related results on deterministic $\ell_\infty$ perturbation arguments of \cite{fan2018ell} and \cite{abbe2020entrywise}   only apply to the first $r$ eigenvectors, but not for $\widehat\Xi_{-r}$ or $\widehat V_{-r}$. The technical challenge in extending $R=r$ to $R>r$ is that the extra singular vectors correspond to singular values  without  clear separation from neighboring singular values, so the standard eigen-gap arguments do not apply to them.  Our Theorem~\ref{thm:main}(ii) leverages the    random-matrix local law to analyze the
  entrywise structure of the extra singular vectors.

Secondly,   the improvement from $\nu_M^{-1}$ in (ii)  to $\nu_M^{-2}$ in (iii)  is genuine: the extra eigenvectors are driven by $U$ and thus nearly orthogonal to the factor space. The gap between the signal and noise singular values is what drives the faster $\nu_M^{-2}$ rate up to a subpolynomial factor.
The rate is  also strictly sharper than what a Davis--Kahan / $\sin\Theta$ argument would deliver. The standard $\sin\Theta$ bound for a perturbed singular subspace yields only $O_P(\nu_M^{-1})$, with no improvement for the eigenvectors orthogonal to factors.  By contrast, an arbitrary $\eta$ in (ii) carries no a priori alignment with the signal block, so only the leading-order incoherence is available, which explains why the rate in (ii) is slower.

 Lastly,   each extra empirical singular vector, denoted by $\widehat\xi_k$ as the $k$th column of $\widehat \Xi_R$ for $k>r$,  admits the orthogonal decomposition
$$\widehat\xi_k = \Xi_r x_k + \Xi_c y_k,$$ where $\Xi_c\in\R^{N\times(N-r)}$ is an orthonormal completion of $\Xi_r$, so that $[\,\Xi_r,\ \Xi_c\,]$ is an $N\times N$ orthogonal matrix, and the coefficients are the projections $x_k=\Xi_r'\widehat\xi_k\in\R^r$ and $y_k=\Xi_c'\widehat\xi_k\in\R^{N-r}$. The signal-aligned coefficient $x_k$ shrinks at the strictly faster rate $\|x_k\|\prec\nu_M^{-2}$ (and hence is $o_P(1)$ under the polynomial-rate condition), while the noise-aligned coefficient $y_k$ inherits the entrywise incoherency of the underlying noise singular vectors; this is the geometric reason why over-estimation in PCA is asymptotically benign for inference.

\medskip 

As a statistical application, Theorem~\ref{thm:main} is especially useful for inference with over-estimated factors, as formalized below.

\begin{corollary}[$r \ge 1$]\label{coro2.1}
Suppose the assumptions of Theorem~\ref{thm:main} hold and $T^{\tau}=o(\nu_M^{2})$ for some constant $\tau>0$. Let $G_N$ and $G_T$ be any $N \times K$ and $T \times K$ matrices with $K = O(1)$, $\|G_N\|_{\mathrm{F}} + \|G_T\|_{\mathrm{F}} = O_P(\sqrt{T})$, and $G_N, G_T$ independent of $U$. Then
\[
\frac{1}{NT}\,\| \widehat B' U G_T\| + \frac{1}{NT}\,\| \widehat F' U' G_N\| \;\prec\; T^{-1/2}\, \nu_M^{-1}.
\]
\end{corollary}

To see that the rate in Corollary~\ref{coro2.1} is sharp, consider the strong-factor case $\nu_M = \sqrt{T}$. Then $(NT)^{-1}\|\widehat F' U' G_N\|  \prec T^{-1}$, which matches, up to a $T^{o(1)}$ factor, the rate as if the true factors were used: $(NT)^{-1}\|F' U' G_N\|= O_P(T^{-1})$.
So the price of replacing the population factor by its PCA estimator does not introduce extra first order  variance. Such a statistical corollary  plays a central role in Section~\ref{secapp} for factor-augmented inference.

\medskip 

Lastly, the corollary below specifies the special case that $r=0$:

\begin{corollary}[Boundary case $r=0$]\label{coro:rzero}
Suppose Assumption~\ref{aspt:D0} holds, $N/T\to\phi\ne 1$, and $r=0$, so that $X=U=\Sigma_e^{1/2}E$. Then for every bounded integer $R$ and any unit-norm vectors $\eta\in\R^N,\zeta\in\R^T$ independent of $U$,
\[
\|\eta'\widehat\Xi_R\|\;\prec\; T^{-1/2},\qquad \|\zeta'\widehat V_R\|\;\prec\;T^{-1/2}.
\]
\end{corollary}

Each of the leading $R$ empirical singular vectors of $X$ is therefore incoherent with respect to any direction independent of $U$ at the standard noise-only rate. The bound follows from a more general result, which we state as   Proposition~\ref{prop:loallaw} in Appendix~\ref{sec:genthm}: when $r=0$, the empirical singular vectors of $X$ coincide with those of $U$, and no factor-strength condition is needed.

\subsection{Asymptotic Behavior of the Overestimated Factor Space}\label{secspace}

We next turn to estimation of the low-rank signal and the factor space.

\begin{theorem}[Low-rank recovery, $r \ge 1$]\label{thm:common}
Under the assumptions of Theorem~\ref{thm:main}, for each fixed  $R \ge r$ with $R-r$ bounded,
\[
\frac{1}{\sqrt{NT}}\,\| \widehat M - M\|_{\mathrm{F}} = O_P\bigl(T^{-1/2}\bigr).
\]
{The bound is  uniform over $R\in\{r,r+1,\ldots,\bar R\}$ for any fixed $\bar R\ge r$.}
\end{theorem}

Theorem~\ref{thm:common} shows that the low-rank component is recovered at the standard parametric rate, uniformly over any bounded working dimension $R \ge r$. Two features of this statement are worth noting. The rate $T^{-1/2}$ matches the optimal rate achievable by a hypothetical oracle estimator that knows $r$: overestimation by a bounded amount does not slow down the rate of convergence for low-rank recovery. The extra components $\widehat\Xi_{-r}\widehat L_{-r}\widehat V_{-r}'$ included in $\widehat M$ have Frobenius norm of order $\sqrt{T}$, but Theorem~\ref{thm:main}(iii) ensures that this contribution is aligned with the noise direction rather than the signal direction.  While alternative estimators based on hard-thresholding, soft-thresholding, or nuclear-norm minimization can also achieve parametric recovery, fixed-order PCA does so without any tuning parameter beyond the integer $R$ itself; this is useful in settings where cross-validation  is unstable or computationally expensive.

We now address estimation of the factor space itself. In the classical case $R = r$, PCA consistently recovers $F$ up to an invertible $r \times r$ rotation. When $R > r$, the notion of rotation must be generalized, and we introduce two natural extensions.

\textbf{Expanded rotation.} Define the $r \times R$ matrix
\[
H' = \frac{1}{N} B' \widehat B.
\]
The model $X = B F' + U$ immediately yields
\begin{equation}\label{eq2.1}
\widehat F = F H' + \frac{1}{N}\, U' \widehat B.
\end{equation}
Up to the statistical error $N^{-1} U' \widehat B$, the true factors $F$ are \emph{expanded} by $H'$ into the larger space spanned by $\widehat F$.

\textbf{Compressed rotation.} Let $H^{+} = (HH')^{+} H$ denote the $R \times r$ Moore--Penrose inverse of $H$. Post-multiplying~\eqref{eq2.1} by $H^{+}$ and using $H' H^{+} = I_r$, valid on the event $\{\lambda_r(H) > 0\}$ (which has probability tending to one by Proposition~\ref{prop:H-invert} below), gives
\begin{equation}\label{eq2.2}
\widehat F H^{+} = F + \frac{1}{N}\, U' \widehat B\, H^{+}.
\end{equation}
Thus $\widehat F$ is \emph{compressed} by $H^{+}$ back to an object of the true dimension.

 {The two rotations serve different purposes. The expanded rotation $H$ is appropriate when the goal is to align the columns of $\widehat F$ with $F$ \emph{without} forcing a dimension match; this formulation applies when $\widehat F$ enters downstream as a linear regressor, since regression projects onto a span and is invariant to the choice of basis within that span.
 Let $\mathrm{col}(F)$ denote the linear space spanned by the columns of $F$.
 The identity~\eqref{eq2.1} exhibits the geometric content of the expansion: the $r$-dimensional factor span $\mathrm{col}(F)$ is approximately contained in the $R$-dimensional fitted span $\mathrm{col}(\widehat F)$, with $H'$ mapping a basis of the former into the latter.

 The compressed rotation $H^{+}$ is appropriate when the goal is to identify the $r$ true factor directions \emph{within} the larger $R$-dimensional fitted space; this formulation applies when the downstream object is an $r$-column object explicitly referencing the true factor dimension rather than the working dimension $R$. The identity~\eqref{eq2.2} performs the inverse extraction of $\mathrm{col}(F)$ from inside $\mathrm{col}(\widehat F)$; 
 The two notions coincide when $R = r$, in which case both reduce to the standard $r\times r$ rotation matrix of the classical PCA literature.}

The next proposition records the non-degeneracy scale of $H$ uniformly over $R \ge r$.   We write $\nu_{\min}:=\lambda_r(S_B)$ when $r\ge 1$, so that $\nu_{\min}\asymp \nu_M^2/N$.

\begin{proposition}[Non-degeneracy of the rotation]\label{prop:H-invert}
Under the assumptions of Theorem~\ref{thm:main} with $r \ge 1$, there exists a constant $c_0 > 0$ depending only on $(c, C, \phi)$ such that, for every bounded $R \ge r$,
\[
\Pr\{\lambda_r(H) \ge c_0\nu_{\min}^{1/2}\} \to 1, \qquad \|H^{+}\| = O_P(\nu_{\min}^{-1/2}).
\]
\end{proposition}

Under strong factors, $\nu_{\min}\asymp 1$, so Proposition~\ref{prop:H-invert} recovers the familiar bounded-inverse behavior $\|H^+\|=O_P(1)$. Under weak factors, the inverse rotation carries the additional scale $\nu_{\min}^{-1/2}$.

\begin{theorem}[Factor space, $r \ge 1$]\label{thm:factor}
Suppose the assumptions of Theorem~\ref{thm:main} hold and there exists a constant $\tau > 0$ with $T^{\tau} = o(\nu_M)$. Then for every bounded $R \ge r$:
\begin{enumerate}
\item[(i)] \emph{Expanded rotation.}
\[
\frac{1}{\sqrt{T}}\,\| \widehat F - F H'\| = O_P(T^{-1/2}), \qquad \frac{1}{T} \left\|  F' (\widehat F - F H') \right\| \;\prec\; T^{-1/2}\, \nu_M^{-1}.
\]
\item[(ii)] \emph{Compressed rotation.}
\[
\frac{1}{\sqrt{T}}\,\| \widehat F H^{+} - F\| = O_P(\nu_M^{-1}), \qquad \frac{1}{T}\,\| F'(\widehat F H^{+} - F)\| \;\prec\; \nu_M^{-2}.
\]
\item[(iii)] \emph{Inverse factor covariance.}
\[
H' \left(\frac{1}{T} \widehat F' \widehat F\right)^{-1} H = \left(\frac{1}{T} F' F\right)^{-1} + o_P(1).
\]
\end{enumerate}
\end{theorem}

We use the $\nu_M$ form here for direct comparison with (ii) and  \cite{bai03}. When $\nu_M=\sqrt{T}$ (strong factors) both  $\frac{1}{\sqrt{T}}\,\| \widehat F - F H'\|$ and  $\frac{1}{\sqrt{T}}\,\| \widehat F H^{+} - F\| $ converge at  $T^{-1/2}$. When $\nu_M$ is slower than $\sqrt{T}$ (weaker factors), the compressed rate in (ii) is slower than the expanded rate in (i), illustrating that recovering the true factor space is a more difficult problem than aligning the true factor space in the over-estimated space. 


In addition, {Part~(iii) of Theorem~\ref{thm:factor} shows that the inverse estimated factor covariance, properly sandwiched by $H$, is consistent for the true inverse factor covariance regardless of the working dimension $R \ge r$. These inverse covariance matrices are often used for factor-augmented inference when estimated  factors are being used as regressors.
Hence,  result (iii) is the property on which the robust factor-augmented inference of the next section relies.}

Our normalization differs from the one adopted in \cite{bai2023approximate}, and the two conventions place the scale of the problem on opposite sides of the factor decomposition. We fix the loadings through $\widehat B = \sqrt{N}\,\widehat\Xi_R$, so that $N^{-1}\widehat B'\widehat B = I_R$ and the scale of $\widehat B$ is independent of the factor strength; the scale is instead carried by $\widehat F$, whose empirical second moment $T^{-1}\widehat F'\widehat F = (NT)^{-1}\widehat L_R^2$ has leading $r$ eigenvalues of order $\nu_M^2/N$ and $R-r$ extra eigenvalues of order $1/N$, smaller by the factor $\nu_M^2$. By contrast, \cite{bai2023approximate} fix the factors through $\widetilde F = \sqrt{T}\,\widetilde V$, the leading eigenvectors of $X'X$, so that $T^{-1}\widetilde F'\widetilde F = I$ and it is the loading estimate that absorbs the scale; their Proposition~1 establishes consistency of $\widetilde F$ for the rotated factor space under weaker loadings, with rates that deteriorate as the loadings weaken --- the same weak-factor margin that $\nu_M$ indexes here. The two conventions estimate the same column space and are related by an invertible rotation, so the distinction is only where the scale resides.

\section{Application to Treatment-Effect Inference}\label{secapp}

 {As an illustrative application of the spectral results in Section~\ref{secpca}, we consider inference on a treatment coefficient $\beta$ in the factor-augmented system}
\begin{eqnarray}\label{eq4.1}
y_t &=& \mu_y + \beta\,g_t + \rho'f_t + \eta_t,\nonumber\\
g_t &=& \mu_g + \alpha_g'f_t + \varepsilon_{g,t}, \\
z_t &=& \mu_z + \alpha_z'f_t + \varepsilon_{z,t},\nonumber
\end{eqnarray}
where $g_t$ is the treatment variable of interest. We consider the application where $g_t$ is correlated with $\eta_t$ (so that it is endogenous). Then we include an observed instrumental variable (IV) $z_t$, and $\mathbb E( \eta_t|\varepsilon_{z,t}, f_t)=0$.  The latent factors $f_t$ are not directly observed; instead we observe the high-dimensional panel of controls
\begin{equation}\label{eq4.2}
x_t \;=\; B f_t + u_t,
\end{equation}
exactly as in Section~\ref{secmodel}, so that $f_t$ can be recovered by PCA from $X$.  The intercepts $(\mu_y,\mu_g,\mu_z)$ are unrestricted nuisance parameters that absorb marginal means of $(y_t,g_t,z_t)$ at no asymptotic cost. The  innovations $(\eta_t,\varepsilon_{g,t})$ are allowed to be mutually correlated within a period.

 This setup places the application within the partialling-out / Frisch--Waugh--Lovell / Neyman-orthogonal residualized-regression tradition of DML \citep{chernozhukov2016double}, with the distinguishing feature that the high-dimensional nuisance is the latent factor space spanned by $F$, recovered by PCA from the panel $X$, rather than a function of observed controls. The new content of this section is to show that fixed-order PCA is a valid nuisance estimator, delivering $\sqrt{T}$-asymptotic normality with the unadjusted Eicker--White sandwich variance, without sample splitting or  data-driven selection of $R$.

 First, we apply PCA to $X$ to extract $R$ factors, whose $T\times R$ matrix is denoted by $\widehat F$,  as in Section~\ref{secfpca}.
Then  we use the  factor-augmented IV estimator of $\beta$, which is based on the residualized just-identified IV regression:
\begin{equation}\label{eq4.3}
\widehat\beta \;=\; \bigl(\widehat\varepsilon_z'\,\widehat\varepsilon_g\bigr)^{-1}\,\widehat\varepsilon_z'\,\widehat\varepsilon_y,
\end{equation}
with $\widehat\varepsilon_y = (I - P_{[1_T,\widehat F]}) Y$, $\widehat\varepsilon_g = (I - P_{[1_T,\widehat F]}) G$, and $\widehat\varepsilon_z = (I - P_{[1_T,\widehat F]}) Z$ denoting residuals from regressing the outcome  {$Y=(y_1,\ldots,y_T)'$}, the treatment  {$G=(g_1,\ldots,g_T)'$}, and the instrument  {$Z=(z_1,\ldots,z_T)'$} on a constant and the PCA factor estimator. The projection on $[1_T,\widehat F]$ absorbs the unobserved intercepts together with the latent-factor nuisance $\rho'f_t$, and $\widehat\beta$ is the second-stage IV slope on the residualized variables.

The key innovation of our estimator is that  our analysis does not require consistent estimation of the true factor number; it suffices that $R\geq r$ so that the span of $\widehat F$ cover the span of $F$. The key structural assumption is that $(\eta_t,\varepsilon_{g,t},\varepsilon_{z,t})$ is i.i.d.\ across $t$ and independent of $(F,U)$, formalized in Assumption~\ref{ass4.1}(i) below. Under i.i.d.\ errors, the autocovariances of the score $\varepsilon_{z,t}\eta_t$ vanish, so the long-run variance reduces to the contemporaneous expectation and the heteroskedasticity-only Eicker--White (HC$_0$) sandwich is the natural variance estimator; the analysis allows arbitrary contemporaneous heteroskedasticity in the score, but not serial dependence or factor-driven conditional variance dynamics.

\begin{remark}[OLS as a special case]\label{rem:OLS}In the special case that the treatment variable $g_t$ is uncorrelated with $\eta_t$, our method collapses to the residual based on OLS estimator by setting $z_t= g_t$. Then $\widehat\beta$ is similar to the OLS estimator analyzed in    \cite{belloni2014inference}, but with the high-dimensional Lasso step replaced by the PCA step.
\end{remark}

\begin{assumption}\label{ass4.1}
The regression and instrument errors and signal strength satisfy:
\begin{enumerate}
\item[(i)] $(\eta_t,\varepsilon_{g,t},\varepsilon_{z,t})$ is independent and identically distributed across $t$, independent of $(F,U)$, with $\E[\eta_t] = \E[\varepsilon_{g,t}] = \E[\varepsilon_{z,t}] = 0$, the exclusion restriction $\E[\eta_t\,\varepsilon_{z,t}] = 0$, and the relevance condition $\gamma := \E[\varepsilon_{g,t}\,\varepsilon_{z,t}] \ne 0$. 
\item[(ii)] $\sqrt{T}\,\log T = o(\nu_M^{2})$.
\item[(iii)] $\E|\eta_t|^{8} + \E|\varepsilon_{g,t}|^{8} + \E|\varepsilon_{z,t}|^8 \le C$, with $\E[\varepsilon_{z,t}^{2}\eta_t^{2}]>0$.
\item[(iv)] When $r\ge1$, the factor sample mean obeys $T^{-1/2}\sum_{t=1}^T f_t=O_P(1)$.
\item[(v)]
When $r\ge2$, with probability converging to $1$, the within-signal singular values are uniformly separated at the same scale, so for some $c$
\[
\min_{1\le k<r}\{\lambda_k(M)-\lambda_{k+1}(M)\}\ge c\,\nu_M\sqrt T.
\]
\end{enumerate}
\end{assumption}

Condition~(ii) is the factor-strength requirement relevant for inference; it is equivalent to $\lambda_r(B'B) \gg T^{1/2}\log T$ and is weaker than the standard strong-factor assumption. {Condition~(iii) is a uniform moment bound that supports both Lyapunov's central limit theorem for $\sqrt{T}(\widehat\beta-\beta)$ and consistency of the heteroskedasticity-robust sandwich variance estimator.} Condition~(iv) is the only time-series restriction imposed on the factor sample mean; it holds, for example, under standard weak-dependence conditions that give a central limit theorem, and ensures that removing the intercept perturbs the factor matrices by only $O_P(1)$ in operator norm.
Condition~(v) can be commonly found in existing econometric analysis works, e.g. \cite{fan2021recent}.
When $z_t = g_t$, the relevance condition $\gamma = \E[\varepsilon_{g,t}^2] \ne 0$ holds automatically and the exclusion $\E[\eta_t\,\varepsilon_{z,t}]=0$ becomes the OLS exogeneity $\E[\eta_t\,\varepsilon_{g,t}]=0$.

\begin{theorem}\label{thinf}
Suppose Assumptions~\ref{aspt:D0},~\ref{aspt:X}, and~\ref{ass4.1} hold,  with the convention that Assumption~\ref{ass4.1}(ii) imposes no condition when $r=0$ since $\nu_M$ is undefined in that case. Define the residual $\widehat\eta_t = \widehat\varepsilon_{y,t} - \widehat\beta\,\widehat\varepsilon_{g,t}$ and the Eicker--White (HC$_0$) variance estimator
\[
\widehat\sigma^{2} \;=\; \bigl(\widehat\varepsilon_z'\widehat\varepsilon_g/T\bigr)^{-1}\,\Bigl(\tfrac{1}{T}\sum_{t=1}^{T}\widehat\varepsilon_{z,t}^{\,2}\,\widehat\eta_{t}^{\,2}\Bigr)\,\bigl(\widehat\varepsilon_z'\widehat\varepsilon_g/T\bigr)^{-1}.
\]
Then for every fixed (bounded) integer $R$ with $0\leq r \leq R$,
\[
\sqrt{T}\,\widehat\sigma^{-1}(\widehat\beta - \beta) \;\overset{d}{\longrightarrow}\; \mathcal N(0, 1).
\]
{The estimator $\widehat\sigma^2$ is consistent for the just-identified IV sandwich variance
\[
\sigma^2 \;=\; \gamma^{-2}\,\E[\varepsilon_{z,t}^{2}\eta_t^{2}],\qquad \gamma\;=\;\E[\varepsilon_{g,t}\,\varepsilon_{z,t}],\ \text{as in Assumption~\ref{ass4.1}(i)}.
\]
In the OLS specialization $z_t=g_t$, this reduces to $\sigma^2=\E[\varepsilon_{g,t}^2]^{-2}\,\E[\varepsilon_{g,t}^2\eta_t^2]$, the partialled-out OLS sandwich.}
\end{theorem}

We briefly discuss {the variance-vs-bias tradeoff behind Theorem~\ref{thinf}. Including $R-r$ extra principal components removes $R-r$ extra degrees of freedom from the residualized regression, creating a finite-sample degrees-of-freedom cost of the familiar order $T/(T-R-1)= {1+O(R/T)}$ {; this is asymptotically negligible for fixed $R$ but should be read as a finite-sample variance cost when $R$ is moderate relative to $T$, as documented in Section~\ref{secsim}}. Overestimation does not induce first-order bias: by Theorem~\ref{thm:main}(iii), the overestimated principal directions are asymptotically orthogonal to the factor space, so they do not remove confounding signal. }

Moving on to the variance,  note that  if $F$ were observed, the score would be $T^{-1/2}\sum_t\varepsilon_{z,t}\eta_t$. Replacing $F$ by $\widehat F$ introduces additional terms involving the residual errors and the estimated projection $P_{\widehat F}$. Because $\widehat F$ is a function of $(F,U)$ and the regression errors are independent of $(F,U)$, these terms can be controlled conditionally on the factor-estimation sample (see also Remark~\ref{rem:DML}). The residual bounds established in the proof of Theorem~\ref{thinf} (Appendix~\ref{sec:proofsecapp}) show that the effect of estimating the factor space is $o_P(1)$ after $\sqrt{T}$ normalization whenever $\sqrt{T}\log T=o(\nu_M^2)$. Thus the feasible residualized score has the same first-order limit as the infeasible score based on the true factor space. Therefore, the net effect is no first-order bias and only a mild finite-sample efficiency cost relative to the consequences of underestimation.

 {Theorem~\ref{thinf} does not prescribe how large $R$ should be in finite samples; we defer concrete guidance to Section~\ref{secsim}, where variance inflation as a function of $R/T$ is examined empirically. A working dimension chosen slightly above the value returned by an information criterion or an eigenvalue-ratio test is a natural conservative default. In addition, Theorem~\ref{thinf} covers the boundary case $r=0$ (cf.\ Corollary~\ref{coro:rzero}): estimating $R>0$ ``factors'' when none are present does not invalidate inference. }

\begin{remark}[Connection to related work]\label{rem:DML}
 {In canonical DML, cross-fitting is used to weaken the dependence between the estimated nuisance and the score evaluated on the same observations. Here that dependence is structurally absent: $\widehat F$ is a measurable function of $(F,U)$. By  Assumption~\ref{ass4.1}, $(F,U)$, and therefore $\widehat F$, are 
 independent of   the regression and instrument errors,  so no sample splitting is needed.}

{The closest non-DML antecedent is \cite{MW11}, who establish robustness-to-overspecification for interactive-fixed-effects panel models via Kato perturbation of a profile-likelihood objective. Two further contrasts are worth noting beyond the rate comparison given in Section~\ref{secextra}. First, as noted in the Introduction, the operator-perturbation route is more flexible across estimator families (e.g.\ quasi-MLE, GMM), while the local-law route specializes to PCA but yields sharper rates when applicable. Second, the variance-sandwich consistency in Theorem~\ref{thm:factor}(iii) that underlies our unadjusted-sandwich CLT does not appear to be available from operator-perturbation alone.}
\end{remark}

\section{Simulations}\label{secsim}

We report Monte Carlo evidence on the finite-sample behavior of the fixed-order PCA inference of Section~\ref{secapp}, {focusing on the OLS specialization $z_t=g_t$ (Remark~\ref{rem:OLS})}. The data-generating process has $i = 1, \ldots, N$ and $t = 1, \ldots, T$, intercepts $\mu_g = 2$ and $\mu_y = 3$, and $\beta = 0$. Loadings are generated sparsely:
\[
b_{i,k} \sim \mathcal N(0, 1) \text{ with probability } p_N = N^{-\alpha}, \qquad b_{i,k} = 0 \text{ otherwise},
\]
so a non-zero $\alpha\in[0,1)$ controls the strength of the latent factors: at $\alpha = 0$ the loadings are dense and $\nu_M\asymp\sqrt N$ (the classical strong-factor case of \cite{bai03}); as $\alpha$ grows the non-zero proportion $N^{-\alpha}$ shrinks and $\nu_M\asymp N^{(1-\alpha)/2}$ approaches the boundary $\nu_M\to\infty$ admitted by Assumption~\ref{aspt:X}(iii). Idiosyncratic errors are $u_t = \Sigma_e^{1/2}e_t$ with $e_{i,t}$ i.i.d.\ $\mathcal N(0,1)$ across both $i$ and $t$, exactly as required by Assumption~\ref{aspt:D0}(i), and $\Sigma_e = \mathrm{diag}(D)$ with $D_{ii}\sim\mathrm{Uniform}(0.5,1.5)$. The regression errors $\varepsilon_{g,t}$, $\eta_t$, the factors $f_{k,t}$, the loadings $\rho_k$, and $\alpha_{g,k}$ are all i.i.d.\ standard normal and mutually independent.

For each replication, $\widehat\beta$ is computed by partialling out $[1_T, \widehat F]$ from $y_t$ and $g_t$ and running OLS on the resulting residuals, with $\widehat F$ from the SVD of $X$ as defined in Section~\ref{secapp}; standard errors are Eicker--White HC$_0$. We report the standardized statistic
\[
t_\beta = \sqrt{T}\,\frac{\widehat\beta - \beta}{\widehat\sigma},
\]
and assess the closeness of its empirical distribution under the null $\beta = 0$ to the standard-normal benchmark implied by Theorem~\ref{thinf}. The tables below report, across $1{,}000$ Monte Carlo replications per cell, the sample mean and sample standard deviation (sd) of $t_\beta$, its $0.025$ and $0.975$ quantiles, and the Kolmogorov--Smirnov (KS) $p$-value against $\mathcal N(0,1)$. The Monte Carlo standard error of the sample sd at this resolution is roughly $1/\sqrt{2{,}000}\approx 0.022$.

\paragraph{Choice of $(N, T, \alpha)$} The theoretical conditions of Section~\ref{secapp} impose three joint requirements: (a) Assumption~\ref{aspt:X}(ii) requires $N/T \to \phi$ with $\phi \in (0, \infty)$ and $\phi \neq 1$, so $N$ and $T$ should grow proportionally and the aspect ratio should be bounded away from one; (b) Assumption~\ref{ass4.1}(ii) imposes the rate condition $\sqrt{T}\log T = o(\nu_M^{2})$, which under the sparse-loading design becomes $T\log^2 T = o(N^{2 - 2\alpha})$ and binds non-trivially when $\alpha$ is close to $1/2$; (c) the working dimension $R$ should be a fixed bounded overestimate of the true $r$. We therefore organize the evidence around three axes that map to (a)--(c), with $N = 200$ throughout:
\begin{itemize}
\item Experiment~1 fixes $\alpha = 0$ (strong factors, so the rate condition is slack) and varies $T \in \{100, 400, 800\}$, giving $N/T \in \{2, 0.5, 0.25\}$, all bounded away from one. This isolates the role of the aspect ratio.
\item Experiment~2 fixes $T = 400$ (so $N/T = 0.5$, well inside the theory) and varies $\alpha \in \{0.0, 0.1, 0.2, 0.3, 0.4\}$, sweeping factor strength from the classical regime toward the rate boundary.
\item Experiment~3 sets $r = 0$ (no factors), so the inference reduces to the partialled-out boundary case, and varies $T \in \{100, 400, 800\}$.
\end{itemize}

\paragraph{Experiment 1: varying the aspect ratio $N/T$} Table~\ref{table:aspect} reports the diagnostics for $R \in \{1, 2, 3, 6, 12, 30\}$, spanning under-specification ($R < r$), correct specification ($R = r$), and over-specification. The contrast across $R$ is sharp. When $R < r$, the sd of $t_\beta$ explodes to between $4$ and $13$ and KS rejects at every $T$. Once $R \ge r$, the sd drops to between $1.01$ and $1.13$ and KS $p$-values typically exceed $0.4$; at $T \in \{400, 800\}$ the standard-normal approximation is quite accurate across the whole $R \ge r$ range, and at $T = 100$ it degrades only once $R$ approaches $T/3$, consistent with the intuition of  variance-inflation. The aspect-ratio behaviour is symmetric: the $T = 100$ ($N/T = 2$) and $T = 400$ ($N/T = 0.5$) columns behave equivalently, in line with the symmetry of Assumption~\ref{aspt:X}(ii) about $\phi = 1$. Figure~\ref{fig:hist} displays the corresponding histograms for $R \in \{2, 3, 12\}$, with the $R = 2$ column ($R < r$) wildly diffuse and most mass off the plotting window.

\begin{figure}[h!]
\centering
\includegraphics[width=\textwidth]{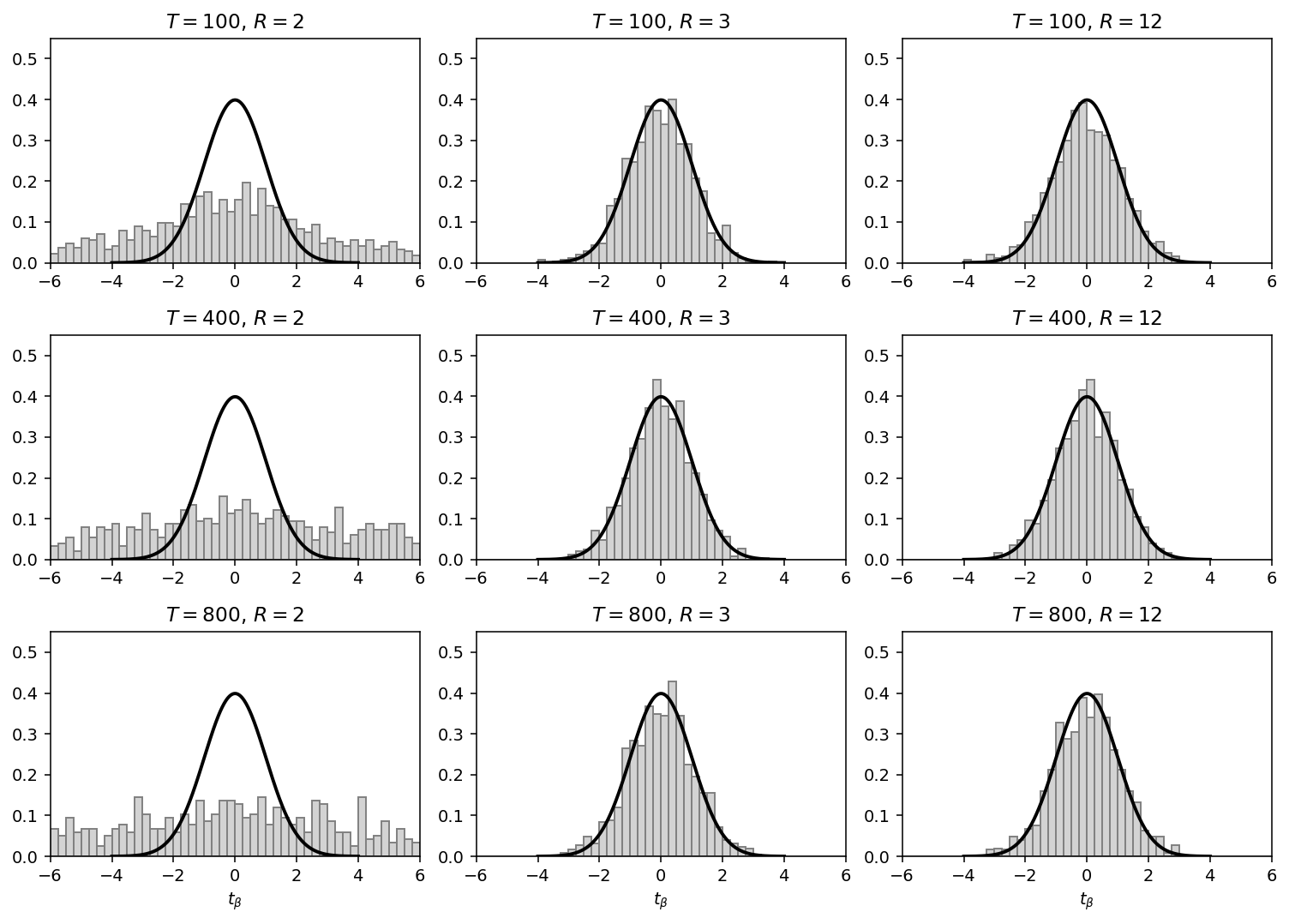}
\caption{Empirical density of $t_\beta$ across $1{,}000$ Monte Carlo replications under the design of Experiment~1 ($r = 3$, $\alpha = 0$, $N=200$). Rows: $T \in \{100, 400, 800\}$. Columns: $R \in \{2, 3, 12\}$. Black curves are the standard-normal density. The $R = 2$ column ($R < r$) is wildly diffuse and most of its mass lies outside the plotting window.}
\label{fig:hist}
\end{figure}

\begin{table}[htbp]
\centering
\caption{Experiment~1, $r = 3$, $\alpha = 0$. Each entry is the sample mean and sample standard deviation (sd) of $t_\beta$ across $1{,}000$ Monte Carlo replications, the $0.025$ and $0.975$ sample quantiles ($\alpha_{.025}$, $\alpha_{.975}$), and the $p$-value of the Kolmogorov--Smirnov test of the empirical distribution of $t_\beta$ against $\mathcal N(0, 1)$. The standard-normal benchmarks for the quantiles are $\pm 1.96$.}
\label{table:aspect}
\begin{tabular}{c|cccccc}
\toprule
$T$ & $R$ & mean & sd & $\alpha_{.025}$ & $\alpha_{.975}$ & KS $p$ \\
\midrule
\multirow{6}{*}{$100$}
 & $1$  & $-0.11$ & $5.19$  & $-10.09$ & $10.25$ & $<10^{-4}$ \\
 & $2$  & $-0.21$ & $4.25$  & $-9.50$  & $9.17$  & $<10^{-4}$ \\
 & $3$  & $-0.02$ & $1.06$  & $-2.09$  & $2.12$  & $0.475$ \\
 & $6$  & $-0.02$ & $1.09$  & $-2.18$  & $2.21$  & $0.335$ \\
 & $12$ & $-0.02$ & $1.12$  & $-2.26$  & $2.16$  & $0.153$ \\
 & $30$ & $-0.01$ & $1.28$  & $-2.68$  & $2.52$  & $0.001$ \\
\midrule
\multirow{6}{*}{$400$}
 & $1$  & $+0.07$ & $9.72$  & $-20.06$ & $18.15$ & $<10^{-4}$ \\
 & $2$  & $-0.05$ & $8.62$  & $-19.42$ & $18.14$ & $<10^{-4}$ \\
 & $3$  & $-0.00$ & $1.01$  & $-2.06$  & $1.96$  & $0.947$ \\
 & $6$  & $+0.00$ & $1.01$  & $-2.08$  & $1.95$  & $0.943$ \\
 & $12$ & $+0.00$ & $1.02$  & $-2.06$  & $1.97$  & $0.675$ \\
 & $30$ & $-0.00$ & $1.05$  & $-2.02$  & $2.02$  & $0.600$ \\
\midrule
\multirow{6}{*}{$800$}
 & $1$  & $+0.29$ & $13.07$ & $-26.75$ & $24.60$ & $<10^{-4}$ \\
 & $2$  & $+0.47$ & $12.05$ & $-24.92$ & $25.19$ & $<10^{-4}$ \\
 & $3$  & $+0.01$ & $1.07$  & $-2.24$  & $2.09$  & $0.521$ \\
 & $6$  & $+0.01$ & $1.07$  & $-2.20$  & $2.08$  & $0.500$ \\
 & $12$ & $+0.01$ & $1.07$  & $-2.25$  & $2.11$  & $0.603$ \\
 & $30$ & $+0.01$ & $1.08$  & $-2.21$  & $2.11$  & $0.387$ \\
\bottomrule
\end{tabular}
\end{table}

\paragraph{Experiment 2: varying the factor strength $\alpha$} Table~\ref{table:strength} reports the diagnostics across $\alpha$ and $R \in \{1, 2, 3, 6, 12\}$. As $\alpha$ grows, $\nu_M\asymp N^{(1-\alpha)/2}$ shrinks, so the rate condition $\sqrt T\log T = o(\nu_M^2)$ becomes $T\log^2 T = o(N^{2-2\alpha})$; the rate-condition margin $N^{2-2\alpha}/T$ decays as $\alpha$ increases. The simulations track this scale sharply. Under-specified rows ($R \in \{1, 2\}$) behave as in Experiment~1, with sd of $t_\beta$ between $8$ and $10$ across all $\alpha$. Once $R \ge r$, the sd becomes a clean monotone function of $\alpha$: $1.07$--$1.08$ at $\alpha \in \{0.0, 0.1\}$ (rate condition slack, KS $p \gtrsim 0.25$), $1.20$ at $\alpha = 0.2$, $1.31$ at $\alpha = 0.3$, and $1.83$ at $\alpha = 0.4$, with $0.025$/$0.975$ quantiles widening to $\pm 3.7$. The role of the rate condition $\sqrt T\log T = o(\nu_M^2)$ is therefore not a proof artifact: it sharply governs the finite-sample quality of the standard-normal approximation.

\begin{table}[htbp]
\centering
\caption{Experiment~2, $r = 3$, $T = 400$. Each entry is the sample mean and sample standard deviation (sd) of $t_\beta$ across $1{,}000$ Monte Carlo replications, the $0.025$ and $0.975$ sample quantiles, and the KS $p$-value against $\mathcal N(0, 1)$.}
\label{table:strength}
\begin{tabular}{c|cccccc}
\toprule
$\alpha$ & $R$ & mean & sd & $\alpha_{.025}$ & $\alpha_{.975}$ & KS $p$ \\
\midrule
\multirow{5}{*}{$0.0$}
 & $1$  & $-0.04$ & $9.66$ & $-18.24$ & $19.42$ & $<10^{-4}$ \\
 & $2$  & $-0.06$ & $8.49$ & $-18.47$ & $17.09$ & $<10^{-4}$ \\
 & $3$  & $-0.00$ & $1.07$ & $-2.10$  & $2.10$  & $0.315$ \\
 & $6$  & $+0.00$ & $1.07$ & $-2.03$  & $2.14$  & $0.393$ \\
 & $12$ & $+0.01$ & $1.08$ & $-2.09$  & $2.13$  & $0.364$ \\
\midrule
\multirow{5}{*}{$0.1$}
 & $1$  & $+0.10$ & $9.66$ & $-20.06$ & $18.72$ & $<10^{-4}$ \\
 & $2$  & $+0.35$ & $8.55$ & $-16.75$ & $18.48$ & $<10^{-4}$ \\
 & $3$  & $+0.04$ & $1.07$ & $-2.00$  & $2.17$  & $0.253$ \\
 & $6$  & $+0.04$ & $1.07$ & $-2.01$  & $2.20$  & $0.076$ \\
 & $12$ & $+0.04$ & $1.07$ & $-2.04$  & $2.13$  & $0.077$ \\
\midrule
\multirow{5}{*}{$0.2$}
 & $1$  & $+0.21$ & $9.69$ & $-18.65$ & $19.64$ & $<10^{-4}$ \\
 & $2$  & $+0.25$ & $8.33$ & $-15.62$ & $18.41$ & $<10^{-4}$ \\
 & $3$  & $-0.02$ & $1.19$ & $-2.42$  & $2.27$  & $0.027$ \\
 & $6$  & $-0.02$ & $1.20$ & $-2.40$  & $2.30$  & $0.018$ \\
 & $12$ & $-0.02$ & $1.20$ & $-2.47$  & $2.23$  & $0.016$ \\
\midrule
\multirow{5}{*}{$0.3$}
 & $1$  & $-0.04$ & $9.31$ & $-18.85$ & $17.78$ & $<10^{-4}$ \\
 & $2$  & $+0.17$ & $8.02$ & $-17.29$ & $15.20$ & $<10^{-4}$ \\
 & $3$  & $+0.04$ & $1.31$ & $-2.66$  & $2.45$  & $<10^{-4}$ \\
 & $6$  & $+0.03$ & $1.31$ & $-2.71$  & $2.51$  & $<10^{-4}$ \\
 & $12$ & $+0.03$ & $1.31$ & $-2.76$  & $2.54$  & $<10^{-4}$ \\
\midrule
\multirow{5}{*}{$0.4$}
 & $1$  & $-0.54$ & $9.52$ & $-19.34$ & $19.33$ & $<10^{-4}$ \\
 & $2$  & $-0.16$ & $8.31$ & $-17.13$ & $17.06$ & $<10^{-4}$ \\
 & $3$  & $-0.06$ & $1.83$ & $-3.69$  & $3.70$  & $<10^{-4}$ \\
 & $6$  & $-0.06$ & $1.83$ & $-3.66$  & $3.71$  & $<10^{-4}$ \\
 & $12$ & $-0.06$ & $1.82$ & $-3.65$  & $3.70$  & $<10^{-4}$ \\
\bottomrule
\end{tabular}
\end{table}

\paragraph{Experiment 3: boundary case ($r = 0$)} Theorem~\ref{thinf} includes the no-factor boundary; Table~\ref{table:boundary} illustrates its Gaussian conclusion in this case. At $T \in \{400, 800\}$ the standard-normal approximation is essentially exact for every $R \in \{0, 3, 6, 12, 30\}$ (sd within $1.02$--$1.10$, KS $p \gtrsim 0.4$); at $T = 100$, the approximation deteriorates only when $R$ approaches $T/3$.

\begin{table}[htbp]
\centering
\caption{Experiment~3, $r = 0$. Each entry is the sample mean and sample standard deviation (sd) of $t_\beta$ across $1{,}000$ Monte Carlo replications, the $0.025$ and $0.975$ sample quantiles, and the KS $p$-value against $\mathcal N(0, 1)$.}
\label{table:boundary}
\begin{tabular}{c|cccccc}
\toprule
$T$ & $R$ & mean & sd & $\alpha_{.025}$ & $\alpha_{.975}$ & KS $p$ \\
\midrule
\multirow{5}{*}{$100$}
 & $0$  & $+0.02$ & $1.06$ & $-1.90$ & $2.14$ & $0.795$ \\
 & $3$  & $+0.02$ & $1.09$ & $-1.97$ & $2.17$ & $0.664$ \\
 & $6$  & $+0.01$ & $1.11$ & $-1.94$ & $2.23$ & $0.464$ \\
 & $12$ & $+0.01$ & $1.13$ & $-2.02$ & $2.34$ & $0.088$ \\
 & $30$ & $-0.02$ & $1.28$ & $-2.47$ & $2.73$ & $<10^{-3}$ \\
\midrule
\multirow{5}{*}{$400$}
 & $0$  & $-0.02$ & $1.05$ & $-2.14$ & $2.06$ & $0.711$ \\
 & $3$  & $-0.02$ & $1.05$ & $-2.16$ & $2.10$ & $0.363$ \\
 & $6$  & $-0.02$ & $1.06$ & $-2.17$ & $2.10$ & $0.698$ \\
 & $12$ & $-0.02$ & $1.07$ & $-2.18$ & $2.09$ & $0.545$ \\
 & $30$ & $-0.02$ & $1.10$ & $-2.22$ & $2.19$ & $0.459$ \\
\midrule
\multirow{5}{*}{$800$}
 & $0$  & $-0.01$ & $1.02$ & $-2.08$ & $1.98$ & $0.935$ \\
 & $3$  & $-0.01$ & $1.02$ & $-2.07$ & $2.00$ & $0.984$ \\
 & $6$  & $-0.01$ & $1.03$ & $-2.09$ & $2.01$ & $0.980$ \\
 & $12$ & $-0.01$ & $1.04$ & $-2.12$ & $2.02$ & $0.958$ \\
 & $30$ & $-0.01$ & $1.05$ & $-2.16$ & $2.08$ & $0.962$ \\
\bottomrule
\end{tabular}
\end{table}

Two operational takeaways follow. Since $\nu_M$ is unobserved in practice, the leading singular values of $X$ are a natural diagnostic for the rate condition $\sqrt T\log T = o(\nu_M^2)$: if they do not visibly separate from the bulk, the inflation in Table~\ref{table:strength} should be expected. And since under-specification is catastrophic while modest over-specification is essentially free, a natural conservative default is to choose $R$ slightly above the value returned by an information criterion or eigenvalue-ratio test, and to interpret unstable estimates across $R \ge \widehat r$ as evidence that the factor-strength regime is unfavorable rather than that $R$ is too small.


\section{Empirical Application}\label{secemp}

We illustrate fixed-order PCA on a single cross-section from the Health and Retirement Study \citep{Sonnega2014HRS}, the panel survey at the empirical core of the structural retirement literature including \cite{French2011HRS}, with respondents treated as i.i.d.\ draws. The substantive question is whether \emph{labor supply at the intensive-and-extensive margin}---the choice variable $N_t$ in the canonical life-cycle model of \cite{French2011HRS}---affects depressive symptoms after controlling for a low-dimensional latent state of vitality, socioeconomic resources, and underlying mental-health propensity that drives both labor-supply choices and contemporaneous well-being. We apply the partialled-out residualized regression of Section~\ref{secapp}.

{The key statistical insight from this application (as shown by results presented in Table~\ref{table:hrs-real})  is the \emph{robustness profile} of the fixed-order PCA estimator across $R$. The IV estimator stays positive at every $R\ge 1$, and over the upper-$R$ tail $R\in\{7,\ldots,28\}$ it stays within $\pm 8\%$ of the saturating $R=N$ value of $+0.660$ --- exactly the robustness-to-tuning-parameter property the fixed-order PCA framework is designed to deliver. The larger readings at $R\in\{2,3,5\}$, which peak at $+1.005$ (some $52\%$ above the saturating value), are confined to the small-$R$ region where $R$ may still fail to dominate $r$, and their pattern reproduces on real data the asymmetry documented in Section~\ref{secsim}: the estimator is sensitive to $R$ below the true dimension and flat above it. In contrast, a procedure that demands consistent recovery of $r$ would have to commit to a single $\widehat r$, whereas our framework certifies inference uniformly over any $R$ in the upper-$R$ tail.}

We now provide detailed implementation of this application. 
The data are from the RAND HRS Longitudinal File 1992--2022 v1.0, restricted to the wave-14 (2018) interview. After complete-case filtering, the sample has $T=14{,}672$ respondents. Treatment is annual hours worked,
\[
g_t \;=\; (\text{hours/week})_t \times (\text{weeks/year})_t \;+\; (\text{2nd-job hours/week})_t \times (\text{2nd-job weeks/year})_t,
\]
clipped to $[0, 5{,}000]$ and set to zero for non-workers. The marginal distribution is bimodal: a $62.3\%$ atom at $g_t=0$ (non-workers, including fully retired respondents and other labor-market non-participants) and a continuous mass concentrated near full-time, with $\Pr(g_t \in [1{,}000, 2{,}200)) = 20.5\%$ and $\Pr(g_t \ge 2{,}200) = 10.5\%$. The continuous specification strictly nests the binary retirement indicator: it preserves the extensive margin (the $g_t=0$ atom) while resolving the intensive-margin variation among bridge-job holders, partial retirees, and full-time workers that the structural retirement literature treats as economically distinct states. The outcome $y_t = \mathrm{cesd}_t$ is the eight-item Center for Epidemiologic Studies Depression score, integer in $[0,8]$ with higher values indicating more depressive symptoms (sample mean $1.53$, sample standard deviation $2.02$). The control panel $x_t \in \mathbb R^N$ collects {$N=28$} standardized fundamentals, including sex, ethnicity, education, marital and veteran status,   chronic-condition and smoking indicators, cognitive scores and household variables.   Each measures a distinct dimension of the underlying biopsychosocial state.

A preliminary spectral  decomposition of the standardized panel returns shows that there are no single dominant factor: the first   three principal components capture roughly $27\%$ of the total variation.



 {To bring the IV machinery of Section~\ref{secapp} to bear, we use the institutional cutoff at age 62 as the instrument:
\[
z_t \;=\; \mathbf{1}\{\text{respondent $t$ is at least 62 years old}\}.
\]
Age 62 is the early Social Security claiming age and the empirically dominant discontinuity in U.S.\ retirement hazards. We prefer it to the alternative cutoff at age 65 because it isolates the labor-supply / Social Security channel from the Medicare insurance channel that turns on at 65. 
Conditional on the latent state $f_t$, $z_t$ shifts labor supply primarily through this institutional channel rather than through individual mental-health propensity, supporting the exclusion restriction in Assumption~\ref{ass4.1}(i). The first-stage relationship is unambiguously strong: at $R=0$, regressing $g_t$ on $z_t$ gives a slope of $-0.98$ thousand hours/year ($t=-57.7$), and  {the residualized first-stage Wald statistic remains in the range $|t|\in[40,55]$ across all working dimensions $R\in\{1,\ldots,28\}$ reported below}.\footnote{ {Estimating with the alternative cutoff $z_t=\mathbf{1}\{\text{age}_t\ge 65\}$ delivers qualitatively identical conclusions across all $R$: the same sign pattern, IV point estimates within roughly $5\%$ of the age-62 values, and the same significance ordering. We report the age-62 specification as primary because the absence of the Medicare channel makes the just-identified exclusion restriction more defensible.}}}

Table~\ref{table:hrs-real} reports both the OLS estimator $\widehat\beta_{\mathrm{OLS}}$ ($z_t=g_t$) and the IV estimator $\widehat\beta_{\mathrm{IV}}$ ($z_t=\mathbf{1}\{\text{age}_t\ge 62\}$) of Theorem~\ref{thinf} across {$R\in\{0,1,2,3,5,7,10,15,20,28\}$}, both with HC$_0$ standard errors and both reported per a $1{,}000$-hours-per-year increment.

 {One scope qualification should be flagged before reading the numbers. The i.i.d.-across-respondents assumption used in our theory ignores household clustering present in HRS, where spouses appear as separate observations sharing many of the controls. We report the unweighted HC$_0$ inference for transparency with our theory but note that a household-clustered standard error would slightly inflate the reported standard errors. With this caveat, the IV estimator under the exclusion restriction in Assumption~\ref{ass4.1}(i) does have a causal local-average-treatment-effect (LATE) interpretation in the sense of \cite{ImbensAngrist1994}: $\widehat\beta_{\mathrm{IV}}$ identifies the effect of labor supply on depressive symptoms among the subpopulation of compliers, namely respondents whose retirement timing responds to crossing the  {age-62 institutional threshold}. Off-panel confounders that are absorbed neither by $\widehat F$ nor by $z_t$ (for example, idiosyncratic preferences and unanticipated household shocks) cannot bias $\widehat\beta_{\mathrm{IV}}$ to first order so long as they are uncorrelated with the  {age-62 cutoff} after factor adjustment.}

{Three patterns are visible. (i) Omitted-factor bias at $R=0$ is severe: OLS returns $\widehat\beta = -0.276$ ($t=-19.7$), the strong negative cross-sectional association between hours worked and depressive symptoms documented in the HRS-based literature \citep{DaveRashadSpasojevic2008,MandalRoe2008}; once even one principal component is partialled out, the OLS estimate collapses to essentially zero (magnitude below $0.04$ across all $R\ge 1$). The unconditioned cross-sectional correlation is almost entirely accounted for by a single factor direction in $x_t$, and OLS without an instrument provides no informative signal once that factor is removed. (ii) The IV LATE, by contrast, is uniformly \emph{positive} across all $R\ge 1$, ranging between $+0.64$ and $+1.01$, and is flat over the upper-$R$ tail, taking the value $+0.660$ ($t=12.95$) at the saturating $R=N=28$. The sign flip relative to OLS at $R=0$ is the textbook signature of selection bias: respondents who endogenously work more hours are mentally healthier on average for unobserved reasons that the controls do not span. Once the age-62 instrument purges this selection, the IV LATE points the other way: among respondents whose labor supply is shifted by the institutional cutoff, an additional $1{,}000$ hours per year of work raises CES-D by roughly $0.66$ points (about a third of the outcome's standard deviation), consistent in sign with the IV-based retirement literature \citep{Bonsang2012,Coe2011,MazzonnaPeracchi2012,Insler2014} that finds retirement to be protective for mental health among compliers. (iii) HC$_0$ standard errors and the first-stage Wald statistic are essentially flat across $R\ge 1$ for both columns.}

\begin{table}[htbp]
\centering
\caption{ {Real-data estimates of the effect of annual labor supply (hours per year, in thousands) on the eight-item CES-D depression score in HRS wave 14 (2018), partialling out the top-$R$ principal components of the standardized $N=28$-dimensional control panel. Both columns instantiate Theorem~\ref{thinf}: the OLS column corresponds to the $z_t=g_t$ specialization, and the IV column to $z_t=\mathbf{1}\{\text{age}_t\ge 62\}$. HC$_0$ standard errors. $T=14{,}672$.}}
\label{table:hrs-real}
\begin{tabular}{c|cc|cc}
\toprule
\multirow{2}{*}{$R$} & \multicolumn{2}{c|}{OLS} & \multicolumn{2}{c}{IV} \\
\cmidrule(lr){2-3}\cmidrule(lr){4-5}
 & $\widehat\beta_{\mathrm{OLS}}$ & $t$ & $\widehat\beta_{\mathrm{IV}}$ & $t$ \\
\midrule
0 (no controls) & $-0.276$ & $-19.72$ & $+0.331$ & $+8.89$  \\
1               & $-0.019$ & $-1.38$  & $+0.640$ & $+16.27$ \\
2               & $-0.002$ & $-0.15$  & $+0.910$ & $+18.12$ \\
3               & $+0.018$ & $+1.23$  & $+1.005$ & $+19.15$ \\
5               & $+0.013$ & $+0.84$  & $+0.986$ & $+18.87$ \\
7               & $-0.003$ & $-0.23$  & $+0.711$ & $+14.90$ \\
10              & $-0.033$ & $-2.21$  & $+0.646$ & $+13.19$ \\
15              & $-0.032$ & $-2.17$  & $+0.647$ & $+13.28$ \\
20              & $-0.027$ & $-1.81$  & $+0.684$ & $+13.33$ \\
28 (saturating) & $-0.030$ & $-1.98$  & $+0.660$ & $+12.95$ \\
\bottomrule
\end{tabular}
\end{table}

\section{Conclusion}\label{secconcl}

This paper develops spectral theory for principal component analysis in  factor models when the working number of factors $R$ is fixed and weakly dominates the true, unknown factor dimension $r$. Leveraging anisotropic local laws from random matrix theory, we show that the overestimated empirical eigencomponents are noise-governed, incoherent, and near-orthogonal to the factor space; that the low-rank signal and factor space are recovered at the usual parametric rates under suitably generalized (expanded and compressed) rotations; and that factor-augmented inference on a treatment coefficient remains asymptotically valid for any bounded $R \ge r$ when $r\ge 1$. These results formally justify a common empirical practice (deliberately overestimating or adopting a conservative upper bound on the number of factors) and shift the analytical burden from consistent factor-number selection to the structurally milder requirement of bounding $r$ from above.

Beyond the technical contributions, our results have a methodological implication for empirical practice. The dominant approach in factor-model inference treats consistent dimension selection as a logically prior step: estimate $r$, condition on it, and proceed. This paper shows that this step can be replaced by the weaker requirement of specifying an upper bound $R\ge r$. The benefit is robustness: an inferential procedure that depends only on $R\ge r$ is insulated from the finite-sample volatility of $\widehat r$, which is well documented to be substantial in the signal regimes where applied researchers most need reliable inference. The cost is a variance inflation that scales with $R/T$. For empirically relevant settings such as factor-augmented treatment-effect inference, factor-based forecasting, and large-panel principal-component regression, this trade-off can be favorable.

 Several directions merit further investigation, listed in approximate order of substantive importance. First, the factor-augmented regression we consider assumes serially-uncorrelated residuals. A parallel extension to weakly serially-dependent residuals appears tractable so long as cross-sectional independence in $u_t$ is maintained.
 The genuinely difficult extension is to allow serial \emph{and} cross-sectional dependence simultaneously in $u_t$, since that is precisely the regime in which the anisotropic local law of \cite{knowles2017anisotropic} ceases to apply, and a fundamentally different random-matrix input would be required.  Second, our analysis imposes $N/T \to \phi \ne 1$ to inherit incoherence from \cite{knowles2017anisotropic}; removing this gap may require new random-matrix tools, perhaps through a fluctuation-scale argument at the Marchenko--Pastur edge. Finally, our framework restricts attention to bounded $R-r$; the regime in which $R$ grows slowly with $T$ would require sharper control of the cumulative variance contribution of the overestimated components and may be relevant for sieve-type applications. 

\newpage

\appendix 

\section{Regularity Conditions of the MP law}\label{sec:mp}

We summarize the regularity conditions on the population covariance spectrum, following \cite{knowles2017anisotropic}, using our notation system. Given $\Sigma_e$ with eigenvalues ordered $\sigma_1\geq \ldots \geq \sigma_N>0$,  recall from Section~\ref{secextra} the empirical spectral measure
\[
  \pi \mathrel{\mathop:}= \frac{1}{N} \sum_{i=1}^{N} \delta_{\sigma_i}.
\]
Following \cite{knowles2017anisotropic}, in the appendices $\phi \mathrel{\mathop:}= N/T$ denotes the finite-sample dimensional ratio; it converges to the limit in Assumption~\ref{aspt:X}(ii), and the distinction is immaterial for the statements below.

The MP law $\varrho$ describes the typical spectral distribution of the companion sample covariance matrix $T^{-1}U'U=T^{-1}E'\Sigma_e E$ when $T\to \infty$ and $E$ follows Assumption~\ref{aspt:D0}(i). Its nonzero eigenvalues coincide with those of $T^{-1}UU'$, so equivalently $\varrho$ describes the squared singular values of $T^{-1/2}U$ (together with the appropriate zero mass in the rectangular case).
The Stieltjes transform of $\varrho$ is the unique solution in $\mathbb{C}^+$
of the self-consistent equation
\begin{equation}\label{eq:selfconsistent}
  \frac{1}{m} = -z + \phi \int \frac{x}{1 + mx}\,\pi(\mathrm{d}x), \qquad z \in \mathbb{C}^+.
\end{equation}
Let $n \mathrel{\mathop:}= |\operatorname{supp}\pi \setminus \{0\}|$ be the number of distinct
nonzero eigenvalues of $\Sigma$, and write $\operatorname{supp}\pi \setminus \{0\} =
\{s_1 > s_2 > \cdots > s_n\}$. Setting $r_i \mathrel{\mathop:}= \phi\,\pi(\{s_i\})$,
the equation \eqref{eq:selfconsistent} is equivalently written as
\begin{equation}\label{eq:fdef}
  z = f(m), \qquad
  f(x) \mathrel{\mathop:}= -\frac{1}{x} + \sum_{i=1}^{n} \frac{r_i}{x + s_i^{-1}}.
\end{equation}
The function $f$ is smooth on each of the $n+1$ open intervals
\[
  I_1 \mathrel{\mathop:}= (-s_1^{-1},\, 0), \quad
  I_i \mathrel{\mathop:}= (-s_i^{-1},\, -s_{i-1}^{-1}) \;\; (i = 2,\ldots,n), \quad
  I_0 \mathrel{\mathop:}= \mathbb{R} \setminus \bigcup_{i=1}^{n} I_i,
\]
of the real projective line $\overline{\mathbb{R}} = \mathbb{R} \cup \{\infty\}$.

Let $\mathcal{C} \subset \overline{\mathbb{R}}$ denote the multiset of critical points of $f$,
where a nondegenerate critical point is counted once and a degenerate one twice, and where
$\infty$ is counted as a nondegenerate critical point when $\phi = 1$. Some of the known properties of $\mathcal{C}$ include:
\begin{itemize}
  \item $|\mathcal{C} \cap I_0| = |\mathcal{C} \cap I_1| = 1$ and
        $|\mathcal{C} \cap I_i| \in \{0, 2\}$ for $i = 2, \ldots, n$.
  \item $|\mathcal{C}| = 2p$ is even, for some integer $p \geq 1$.
\end{itemize}

We label the $2p - 1$ critical points in $I_1 \cup \cdots \cup I_n$ as
$x_1 \geq x_2 \geq \cdots \geq x_{2p-1}$, and denote by $x_{2p}$ the unique critical point
in $I_0$. The associated \emph{critical values} are
\[
  a_k \mathrel{\mathop:}= f(x_k), \qquad k = 1, \ldots, 2p,
\]
and satisfy $a_1 \geq a_2 \geq \cdots \geq a_{2p}$, with $a_k \in [0, C]$ for all $k$.
Moreover, $x_k = m(a_k)$, where $m$ is extended to $\mathbb{R}$ by continuity from
$\mathbb{C}^+$.

The support of the asymptotic density $\varrho$ in $(0, \infty)$ consists of exactly $p$
connected components:
\begin{equation}\label{eq:support}
  \operatorname{supp}\varrho \cap (0,\infty)
  = \bigcup_{k=1}^{p} [a_{2k},\, a_{2k-1}] \cap (0,\infty).
\end{equation}
Thus the $k$-th \emph{bulk component} of $\varrho$ is the interval $[a_{2k}, a_{2k-1}]$,
with left edge $a_{2k}$ and right edge $a_{2k-1}$. The spectral edges are collectively indexed
by $k = 1, \ldots, 2p$, where odd indices $k = 1, 3, \ldots, 2p-1$ correspond to right edges
and even indices $k = 2, 4, \ldots, 2p$ correspond to left edges of the respective bulk
components.

We now state the regularity conditions, which govern whether the density $\varrho$ is
well-behaved near a given spectral edge or in the interior of a given bulk component. Fix
a small constant $\delta>0$ throughout.

\begin{definition}[Regularity {\cite[Definition~2.7]{knowles2017anisotropic}}]
\label{def:regularity}
\hfill
\begin{enumerate}[\normalfont(i)]
  \item \textbf{Regular edge.} The edge $k \in \{1, \ldots, 2p\}$ is called \emph{regular}
    if the following three conditions hold simultaneously:
    \begin{equation}\label{eq:reg_edge}
      a_k \geq \delta, \qquad
      \min_{l \neq k} |a_k - a_l| \geq \delta, \qquad
      \min_{1 \leq i \leq n} |x_k + s_i^{-1}| \geq \delta.
    \end{equation}
  \item \textbf{Regular bulk component.} The $k$-th bulk component $k \in \{1, \ldots, p\}$
    is called \emph{regular} if, for every fixed $\delta' > 0$, there exists a constant
    $c \equiv c_{\delta, \delta'} > 0$ such that the density of $\varrho$ is bounded below by $c$
    on the compactly interior interval $[a_{2k} + \delta',\, a_{2k-1} - \delta']$.
\end{enumerate}
\end{definition}

\section{{Canonical rotation construction}}\label{secrotation}

This appendix records the explicit construction of the rotations $H_B$ and $H_F$ used in identity~\eqref{identBF}, deferred from Section~\ref{secmodel} for compactness. Recall from Section~\ref{secmodel} the SVD $M=\Xi_r L_r V_r'$ of $BF'$ and the matrices $S_B=N^{-1}B'B$ and $S_f=T^{-1}F'F$. Let $J\in\R^{r\times r}$ denote the diagonal matrix of common eigenvalues of $S_B^{1/2}S_f S_B^{1/2}$ and $S_f^{1/2}S_B S_f^{1/2}$. Let $G_B$ be the $r\times r$ matrix whose columns are eigenvectors of $S_B^{1/2}S_f S_B^{1/2}$ corresponding to eigenvalues $J$, and similarly let $G_F$ collect the eigenvectors of $S_f^{1/2}S_B S_f^{1/2}$. Define
\begin{equation}\label{eb.1}
H_B \;:=\; S_B^{-1/2}\,G_B,\qquad H_F \;:=\; S_f^{-1/2}\,G_F.
\end{equation}
A direct computation gives
\[
\frac{1}{NT}\,MM'\Bigl(\frac{1}{\sqrt N}BH_B\Bigr) \;=\; \Bigl(\frac{1}{\sqrt N}BH_B\Bigr)\,J,\qquad \Bigl(\frac{1}{\sqrt N}BH_B\Bigr)'\Bigl(\frac{1}{\sqrt N}BH_B\Bigr) \;=\; I_r,
\]
so $N^{-1/2}BH_B = \Xi_r$ and equivalently $N^{-1/2}B = \Xi_r H_B^{-1}$. The parallel computation on the factor side, using $(NT)^{-1}M'M = T^{-1}F S_B F'$, gives
\[
\frac{1}{NT}\,M'M\Bigl(\frac{1}{\sqrt T}FH_F\Bigr) \;=\; \Bigl(\frac{1}{\sqrt T}FH_F\Bigr)\,J,\qquad \Bigl(\frac{1}{\sqrt T}FH_F\Bigr)'\Bigl(\frac{1}{\sqrt T}FH_F\Bigr) \;=\; I_r,
\]
so $T^{-1/2}FH_F = V_r$ and $T^{-1/2}F = V_r H_F^{-1}$. Substituting both identities into $M=BF'$ yields $M = \sqrt{NT}\,\Xi_r H_B^{-1}(H_F')^{-1}V_r'$; comparing with the SVD $M=\Xi_r L_r V_r'$ gives $(H_F'H_B)^{-1} = (NT)^{-1/2}L_r$. This proves identity~\eqref{identBF}.

\section{General Theorems}\label{sec:genthm}

\subsection{Preliminaries: notation and probabilistic tools}\label{sec:prelim}
We follow the notation of Sections~\ref{secmodel} and~\ref{secfpca}: for $r\ge 1$, $M=BF'=\Xi_r L_r V_r'$ is the SVD of the signal, $\nu_{\min}=\lambda_r(N^{-1}B'B)$, and $\widehat\Xi_R=(\widehat\Xi_r,\widehat\Xi_{-r})$, $\widehat V_R=(\widehat V_r,\widehat V_{-r})$, and $\widehat B=\sqrt{N}\,\widehat\Xi_R$ collect the empirical singular vectors of $X$. In addition, $\widehat\xi_l\in\R^N$ and $\widehat v_l\in\R^T$ denote the individual left and right singular vectors of $X$ associated with its $l$-th largest singular value $\widehat\lambda_l$; $\widehat B_r=\sqrt{N}\,\widehat\Xi_r$ and $\widehat B_{-r}=\sqrt{N}\,\widehat\Xi_{-r}$ are the corresponding blocks of $\widehat B$; and, for the noise matrix $U=\Sigma_e^{1/2}E$, $u_{0,k}\in\R^N$ and $w_{0,k}\in\R^T$ denote its $k$-th left and right singular vectors. By Assumption~\ref{aspt:X}(iii), with probability approaching one, the singular values of $BF'$ are of order
$$
\nu_1(L_r)\asymp \nu_{r}( BF')\asymp \nu_{1}( BF')\asymp \sqrt{NT} \nu_{\min}^{1/2}\asymp \nu_M\sqrt{T},
$$
so that $\nu_{\min}^{-1} \asymp  N\nu_M^{-2}$.

Under Assumption~\ref{aspt:X}(i), the factor matrix $F$ --- and with it $M=BF'$, its singular values, and the singular vectors $\Xi_r$ and $V_r$ --- is random but independent of the noise $E$, so the conditional law of $E$ given $F$ coincides with its unconditional law. All random-matrix estimates for $E$ are therefore applied conditionally on $F$, with constants that do not depend on the realization of $F$; the proofs of Lemmas~\ref{lem:XiU} and~\ref{lem:Wkanis} below spell out this conditioning step once, and subsequent proofs use it without further comment. Throughout, we say an event holds \emph{with high probability} if its complement has probability at most $T^{-K}$ for every fixed $K>0$ and all large $T$, with constants depending on $K$. Arguments that invoke the bounds of Assumption~\ref{aspt:X}(i) and~(iii) are carried out on the event
\[
\mathcal E_T \;:=\; \bigl\{\text{the bounds in Assumption~\ref{aspt:X}(iii) hold}\bigr\}\;\cap\;\bigl\{\max_{t\le T}\|f_t\|\le C_\delta\,\log T\,\bigr\},
\]
where $\delta>0$ is fixed and the constant $C_\delta$ is chosen so that $\Pr(\mathcal E_T)\ge 1-\delta$ for all large $T$, as Assumption~\ref{aspt:X}(i) and~(iii) permit; on $\mathcal E_T$ we have $\lambda_r(M)\ge c\,\nu_M\sqrt T>0$, so $\rank(M)=r$ and $\Xi_r$, $L_r$, $V_r$ are well defined. Each stochastic-domination bound $\Phi\prec\phi$ asserted in the proofs for quantities depending on $F$ is established in the following strong conditional form: for every $\varepsilon>0$ and $K>0$ there exists $C=C(\varepsilon,K,\delta)$, not depending on the realization of $F$, such that $\Pr\bigl(\{\Phi> T^{\varepsilon} C\phi\}\cap\mathcal E_T\bigr)\le T^{-K}$ for all large $T$. This form composes under union bounds over polynomially many events, which is how it is used within the proofs; and since $\delta>0$ is arbitrary, it implies $\Phi=O_P(T^{\varepsilon}\phi)$ for every fixed $\varepsilon>0$ --- that is, $\Phi\prec\phi$ in the sense defined in Section~1 --- unconditionally, which is the form asserted in the statements of our results. The factor $\log T$ in the bound on $\max_{t\le T}\|f_t\|$ is absorbed by $T^{\varepsilon}$, since $\log T=o(T^{\varepsilon})$ for every $\varepsilon>0$, and appears explicitly only through $\|F^c\|_{\infty}$ in Step~4 of the proof of Theorem~\ref{thinf}.

We will use the following result, which is Lemma S.5 of \cite{hu2024network}.
\begin{lemma}
\label{lem:condprod}
Suppose $A\in \reals^{n\times K}, B\in \reals^{K\times p}$ and $\rank(A)=\rank(B)=K$, then
\[
\lambda_{K}(AB)\geq \lambda_K(A)\lambda_K(B),\quad \lambda_1(AB)\leq \lambda_1(A)\lambda_1(B).
\]
\end{lemma}
\begin{proof}
See \cite{hu2024network}, Lemma~S.5.
\end{proof}

The next lemma controls the alignment between the rank-$r$ signal subspace and the leading $k$ noise singular vectors. Together with Lemma~\ref{lem:Wkanis} below, it provides the ``orthogonality budget'' that drives the proof of Theorem~\ref{thm:main}(i).

\begin{lemma}\label{lem:XiU}
Under Assumptions~\ref{aspt:D0} and~\ref{aspt:X}(i), with $\Xi_r\in\R^{N\times r}$ the left singular vectors of $M$ and $U_k\in\R^{N\times k}$ the leading $k$ left singular vectors of $\Sigma_e^{1/2}E$, for every fixed $r,k=O(1)$ and every $\varepsilon>0$,
\[
\|U_k'\,\Xi_r\|\;\le\;T^{\varepsilon-1/2}\quad\text{with high probability}.
\]
The same bound holds for $\|W_k'\,V_r\|$ on the right side, where $V_r\in\R^{T\times r}$ are the right singular vectors of $M$ and $W_k\in\R^{T\times k}$ the leading $k$ right singular vectors of $\Sigma_e^{1/2}E$.
\end{lemma}
\begin{proof}
By Proposition~\ref{prop:loallaw}(3) applied entrywise. Each $\xi_i$ is a function of $F$, which is independent of $E$ by Assumption~\ref{aspt:X}(i); conditionally on $F$, the vector $\xi_i$ is a fixed unit vector while the conditional law of $E$ coincides with its unconditional law, so Proposition~\ref{prop:loallaw}(3) bounds each entry $u_{0,j}'\,\xi_i$ by $C T^{\varepsilon-1/2}$ with high probability, with constants uniform over unit vectors and hence over the realizations of $F$; averaging over $F$ yields the same unconditional bound. The operator-norm bound follows by a union bound over the $rk=O(1)$ entries and the equivalence between operator norm and maximum entry for matrices of fixed dimension. The right-side bound is symmetric.
\end{proof}

We also need an incoherence-type bound on noise-side singular vectors.

\begin{lemma}\label{lem:Wkanis}
Under Assumptions~\ref{aspt:D0} and~\ref{aspt:X}(i), for every fixed $k\le R-r$ and every $\varepsilon>0$, the $k$-th right singular vector $w_k\in\R^{T}$ of $\Sigma_e^{1/2}E$ satisfies, for any (possibly random) unit vector $v\in\R^{T}$ independent of $E$,
\[
|w_k'\,v|\;\le\;T^{\varepsilon-1/2}\quad\text{with high probability}.
\]
Consequently, with $W_k=[w_1,\ldots,w_k]\in\R^{T\times k}$ and $V_r\in\R^{T\times r}$ the right singular vectors of the rank-$r$ signal $M$,
\[
\|V_r'\,W_k\|\;\le\;T^{\varepsilon-1/2}\quad\text{with high probability},
\]
and, on the event $\mathcal E_T$ (on which $\|M\|\le C\nu_M\sqrt{T}$ by Assumption~\ref{aspt:X}(iii)), for the signal matrix $M\in\R^{N\times T}$,
\[
\|M\,W_k\|\;\le\;\|L_r\|\,\|V_r'\,W_k\|\;\le\;\nu_M T^{\varepsilon}\quad\text{with high probability}.
\]
\end{lemma}
\begin{proof}
The first claim is the anisotropic delocalization bound for $w_k$ given in Theorem~3.12 (and Remark~3.13) of \cite{knowles2017anisotropic}; see Proposition~\ref{prop:loallaw} (claim~3) below for the form used here. That bound is stated for deterministic unit vectors, with constants uniform over $v$; for a random $v$ independent of $E$, conditioning on $v$ leaves the law of $E$ unchanged, so the same bound holds conditionally with the same constants and hence unconditionally. The bound on $\|V_r'\,W_k\|$ follows by applying the first claim with $v$ ranging over the $r=O(1)$ columns of $V_r$ (which are functions of $F$ and hence independent of $E$ by Assumption~\ref{aspt:X}(i)) and assembling via a union bound. The bound on $\|MW_k\|$ then uses the SVD $M=\Xi_r L_r V_r'$ and $\|L_r\|\asymp \nu_M\sqrt{T}$ on $\mathcal E_T$.
\end{proof}

\subsection{Proof of Theorem \ref{thm:main}}

The proof uses local laws for the noise matrix under factor models, stated later in Proposition \ref{prop:loallaw}.

\noindent\textbf{Claim~(i): singular values.}
Throughout, $\xi_1,\ldots,\xi_r$ denote the columns of $\Xi_r\in\R^{N\times r}$ (the left singular vectors of $M$), $u_1,\ldots,u_k$ the columns of $U_k\in\R^{N\times k}$ (the leading $k$ left singular vectors of $\Sigma_e^{1/2}E$), $w_1,\ldots,w_k$ the columns of $W_k\in\R^{T\times k}$ (the corresponding right singular vectors), and $\Lambda_k\in\R^{k\times k}$ the diagonal matrix of the leading $k$ singular values of $\Sigma_e^{1/2}E$. The rank-$k$ truncated SVD identity $\Sigma_e^{1/2}E\,W_k=U_k\Lambda_k$, equivalently $(\Sigma_e^{1/2}E)'\,U_k=W_k\Lambda_k$, will be used repeatedly.

\emph{Upper bound.} The bound $\widehat\lambda_{r+k}\le \lambda_k(\Sigma_e^{1/2}E)$ is immediate from Weyl's inequality for singular values applied to $X=M+\Sigma_e^{1/2}E$:
\[
\widehat\lambda_{r+k}=\sigma_{r+k}(M+\Sigma_e^{1/2}E)\;\le\;\sigma_{r+1}(M)+\sigma_k(\Sigma_e^{1/2}E)\;=\;0+\lambda_k(\Sigma_e^{1/2}E),
\]
since $\rank(M)=r$ implies $\sigma_{r+1}(M)=0$.

\emph{Lower bound.} Consider the subspace $S=\mathrm{span}\{\xi_1,\ldots, \xi_r,u_1,\ldots,u_k\}\subset\R^{N}$, of dimension $r+k$ with high probability by Lemma~\ref{lem:XiU}. By Courant--Fischer,
\[
\lambdahat_{r+k}\geq \min_{z\in S, \|z\|=1} \|(M+\Sigma^{1/2}_e  E)'z\|.
\]
Every $z\in S$ admits the representation $z=\Xi_r x+U_ky$ with $x\in\R^r$, $y\in\R^k$. Meanwhile, consider 
$P_r=\Xi_r\Xi_r', P_c=I-P_r$, which are the projection onto the subspace of $\Xi_r$ and its complement respectively. We note that 
\[
z=\Xi_rx+U_ky=
\Xi_rx+(\Xi_r\Xi_r')U_ky+P_c U_k y=
\Xi_r(x+\Xi_r'U_ky)+P_c U_k y. 
\]
In other words, we can write all $z\in S$ as $z=\Xi_r x+P_c U_k y$ for some $x$ and $y$.

\emph{Step 1: orthogonality budget.} By Lemma~\ref{lem:XiU}, $\|U_k'\Xi_r\|\le T^{\varepsilon-1/2}$ with high probability, so
\begin{align*}
1=\|z\|^2=\|\Xi_r x+P_c U_ky\|^2=\|x\|^2+\|P_c U_k y\|^2.
\end{align*}
Meanwhile, we note that by Lemma~\ref{lem:XiU}, $\|P_r U_k\|=\|\Xi_r' U_k\|=O_p( T^{\varepsilon-1/2})$
\[
\|P_c U_k y\|^2=\|y\|^2-\|\Xi_r'U_ky\|^2
\]
which yields
\begin{equation}\label{eq:supp:budget}
 |\|x\|^2+\|y\|^2-1|\;\leq\; C\,T^{2\varepsilon-1}\,\|y\|^2\;\le\;C\,T^{2\varepsilon-1}.
\end{equation}

\emph{Step 2: lower bound on $\|X'z\|^2$.} Decompose
\begin{align*}
X'z&=(M+\Sigma_e^{1/2}E)'(\Xi_r x+ U_k y-P_r U_ky)\\
&=\underbrace{(M+\Sigma_e^{1/2}E)'\Xi_r x}_{A\,\in\,\R^T}+\underbrace{W_k\Lambda_k y}_{D\,\in\,\R^T}-\underbrace{(\Sigma_e^{1/2} E)'P_rU_k y}_{B\,\in\,\R^T},
\end{align*}
where 
we used that $M'(I-P_r)=0$ and 
the second piece uses $(\Sigma_e^{1/2}E)'U_k=W_k\Lambda_k$. We bound each piece.

\emph{(a) $\|A\|$.} Using $\Xi_r' M=L_r V_r'$, $M'\Xi_r=V_r L_r$, and $\|(\Sigma_e^{1/2}E)'\Xi_r\|\le\|\Sigma_e^{1/2}E\|\asymp\sqrt{T}$:
\[
\|A\|\;\ge\;\|V_r L_r x\|-\|(\Sigma_e^{1/2}E)'\Xi_r\,x\|\;\ge\;\nu_M\sqrt{T}\,\|x\|-C\sqrt{T}\,\|x\|\;\ge\;c\,\nu_M\sqrt{T}\,\|x\|
\]
for $\nu_M$ large, hence $\|A\|^2\ge c\,\nu_M^2 T\,\|x\|^2$.
Meanwhile, $\|A\|\le \|x\|\|M+\Sigma_e^{1/2}E\|=O(T \|x\|)$.

\emph{(b) $\|D\|$.} Since $W_k$ has orthonormal columns, $\|D\|^2=\|\Lambda_k y\|^2\ge \lambda_k(\Sigma_e^{1/2}E)^2\,\|y\|^2$.

\emph{(c) $\|B\|$.} By Lemma~\ref{lem:XiU},
$\|B\|\;\le\;
\|\Sigma_e^{1/2}E\|\|P_r U_k\|
\,\|y\|\;=O(T^{\varepsilon}\|y\|)$.

\emph{(d) Cross terms.} The cross terms involving $D$ are controlled using Lemma~\ref{lem:Wkanis} ($\|V_r'W_k\|\le T^{\varepsilon-1/2}$ with high probability) and Lemma~\ref{lem:XiU}:
\begin{align*}
|\langle A,D\rangle|&=|x'\Xi_r'(M+\Sigma_e^{1/2}E)W_k\Lambda_k y|\;\le\;\|x\|\,\|y\|\,\|(M+\Sigma_e^{1/2}E)W_k\|\,\|\Lambda_k\|,\\
\|(M+\Sigma_e^{1/2}E)W_k\|&\le \|MW_k\|+\|\Sigma_e^{1/2}E\,W_k\|\;\le\;\nu_MT^{\varepsilon}+\sqrt{T}\;\le\;C\sqrt{T}\,T^{\varepsilon},
\end{align*}
where the high-probability bound $\|MW_k\|\le\nu_M T^{\varepsilon}$ is from Lemma~\ref{lem:Wkanis}. Hence $|\langle A,D\rangle|\le C T^{1+\varepsilon}\|x\|\|y\|$.
Meanwhile, we have 
\[|\langle A,B\rangle|
\leq \|A\|\|B\|=O(T^{1+\varepsilon}\|x\|\|y\|).
\]
Moving on, using $W_k'(\Sigma_e^{1/2} E)'=(U_k\Lambda_k)'$ we find 
\[
\langle D, B\rangle=y'\Lambda_k W_k'(\Sigma_e^{1/2} E)' P_r U_k y=y'\Lambda_k^2 U_k'\Xi_r\Xi_r' U_k y.
\]
Therefore by Lemma~\ref{lem:XiU}, 
\[
|\langle D, B\rangle|\leq \|\Lambda_k\|^2 \|U_k'\Xi_r\|^2\|y\|^2\prec T\cdot T^{2\varepsilon-1}\|y\|^2=T^{2\varepsilon}\|y\|^2. 
\]
Combining (a)--(d), since $X'z=A+D-B$,
\begin{align*}
\|X'z\|^2\;&
=\|A\|^2+\|D\|^2+\|B\|^2-2\langle A, B\rangle+2\langle A, D\rangle-2\langle D, B\rangle\\
&\ge\;\nu_M^2 T\,\|x\|^2+\lambda_k(\Sigma_e^{1/2}E)^2\,\|y\|^2+0-C T^{1+\varepsilon}\,\|x\|\|y\|-C T^{2\varepsilon}\|y\|^2.
\end{align*}
\emph{Step 3: minimization.} Let $a=\|x\|,\;l_k=\lambda_k(\Sigma^{1/2}_eE)\asymp\sqrt{T}$; using \eqref{eq:supp:budget}, that is, lower-bound $\|y\|^2$ by $1-a^2$, upper-bound $\|y\|^2$ by $1-a^2+CT^{2\varepsilon-1}$ when it appears with negative sign,
\begin{align*}
\|X'z\|^2
&\geq \nu_M^2T a^2+l_k^2(1-a^2)-C T^{1+\varepsilon}\,a\sqrt{1+C T^{2\varepsilon-1}-a^2}-CT^{2\varepsilon}(1+C T^{2\varepsilon-1}-a^2).
\end{align*}
Define
\[
f(a)=(\nu_M^2 T-l_k^2+C T^{2\varepsilon})a^2+l_k^2-CT^{1+\varepsilon} a\sqrt{1+CT^{2\varepsilon-1}-a^2}-CT^{2\varepsilon}(1+CT^{2\varepsilon-1}).
\]
Note that 
$f$ has derivative
\[
f'(a)=2a(\nu_M^2 T-l_k^2+CT^{2\varepsilon})-CT^{1+\varepsilon}\frac{1+CT^{2\varepsilon-1}-2a^2}{\sqrt{1+CT^{2\varepsilon-1}-a^2}}.
\]
Since $\nu_M^2 T-l_k^2+CT^{2\varepsilon}\asymp \nu_M^2 T$ for $\nu_M\to \infty$, $f'(a)>0$ whenever $a>\nu_M^{-2}T^{\varepsilon}$; hence the minimizer $a^*\leq \nu_M^{-2}T^{\varepsilon}$, and
\begin{equation}\label{eq:svgap}
\lambdahat_{r+k}^2\geq f(a^*)\geq l_k^2-C\nu_M^{-2} T^{1+2\varepsilon}-2CT^{2\varepsilon}\geq \lambda_k(\Sigma_e^{1/2}E)^2-O(T^{1+2\varepsilon}\nu_M^{-2}),
\end{equation}
since $T^{2\varepsilon}=O(T^{2\varepsilon}\cdot\nu_M^{-2}T)$ throughout the range $\nu_M=O(\sqrt{N})$ (the inequality $1\leq C\nu_M^{-2}T$ holds at $\nu_M\leq C\sqrt{T}$, with equality up to constants at the strong-factor boundary); as $\varepsilon>0$ is arbitrary, this is the claimed bound $\lambdahat_{r+k}^2\ge\lambda_k(\Sigma_e^{1/2}E)^2-O_P(T^{\varepsilon}\,\nu_M^{-2}T)$ for every $\varepsilon>0$.
Combined with the Weyl upper bound, this proves \[
0\leq \lambda_k(\Sigma_e^{1/2}E)^2-\widehat\lambda_{r+k}^2\prec \nu_M^{-2}T,
\]
uniformly over $\nu_M=O(\sqrt{N})$, including the strong-factor boundary $\nu_M\asymp\sqrt{N}$. 

\noindent\textbf{Claim~(iii): near-orthogonality, left singular vectors.} We first prove claim (iii) and then claim (ii), since the proof of the latter uses the former.

\textbf{Step 1. Block transformation of the eigenvector equation.}
Recall the SVD $M=\Xi_r L_r V_r'$ from the Preliminaries; we write $B_r:=L_r V_r'\in\R^{r\times T}$ for compactness, so $M=\Xi_r B_r$ and $\|B_r\|=\|L_r\|\asymp\sqrt{T}\,\nu_M$. Since $B$ shares its left singular space with $\Xi_r$ and all $r$ eigenvalues of $B'B/N$ are of order $\nu_{\min}\asymp \nu_M^2/N$, showing $\|B\|_{\mathrm{F}}^{-1}\|B'\Xihat_{-r}\|=O(T^{\varepsilon}\nu_M^{-2})$ is equivalent to showing $\|\Xi_r'\xihat_k\|=O(T^{\varepsilon}\nu_M^{-2})$ for all $k\in [r+1,R]$.

As in the display following Theorem~\ref{thm:main}, let $\Xi_c\in\R^{N\times(N-r)}$ be the orthonormal completion of $\Xi_r$, so $\Xi_r\Xi_r'+\Xi_c\Xi_c'=I_N$, and decompose $\xihat_k=\Xi_r x_k+\Xi_c y_k$ with
\[
x_k=\Xi_r' \xihat_k\in\R^r,\quad
y_k=\Xi_c' \xihat_k\in\R^{N-r}.
\]
Then $\|x_k\|^2+\|y_k\|^2=1$ by orthonormality of $\xihat_k$ and the identity $\Xi_r\Xi_r'+\Xi_c\Xi_c'=I_N$. The goal is to bound $\|x_k\|$ for $r+1\leq k\leq R$.

By claim 1, with high probability $ \lambdahat_k\asymp\sqrt{T}$. We write $E_r=\Xi_r' \Sigma^{1/2}_e  E,E_c=\Xi_c' \Sigma^{1/2}_e  E$. Then we have 
\[
(M+\Sigma^{1/2}_e  E)(M+\Sigma^{1/2}_e  E)'(\Xi_rx_k+\Xi_cy_k)=\lambdahat_k^2\cdot (\Xi_rx_k+\Xi_cy_k)
\]
Because $\Xi_r\bot \Xi_c$, respectively left multiply by $\Xi_r$ and $\Xi_c$ to both sides leads to:
\begin{equation}\label{eqQ}
Q_r:=( B_r + E_r),\quad
\begin{bmatrix}
Q_rQ_r' &  Q_rE'_c\\
 E_cQ_r' &  E_cE'_c
\end{bmatrix}\begin{bmatrix}
x_k\\
y_k
\end{bmatrix}=
\begin{bmatrix}
\lambdahat_k^2x_k\\
\lambdahat_k^2 y_k
\end{bmatrix}
\end{equation}
First row gives us 
\begin{equation}\label{eq2.1x}
x_k=- (Q_rQ_r'-\lambdahat_k^2 I)^{-1}Q_rE_c'y_k
\end{equation}
The inversion is feasible because with high probability, the following holds with some constant
\[
\lambda_r(Q_r)\geq \lambda_r( B_r)-\|E_r\|\geq \frac{1}{C}\nu_M\sqrt{T}-C\sqrt{T}\geq \frac{1}{2C}\nu_M\sqrt{T}
\]
for some constant $c$, while $\lambdahat_k\leq \|\Sigma^{1/2}_e  E\|=C\sqrt{T}$ with high probability. So $Q_rQ_r'-\lambdahat_k^2 I\succ 0$ with high probability.
Equation~\eqref{eq2.1x} also delivers an a priori bound, valid for every $k\in\{r+1,\ldots,R\}$ and free of any induction: the two displays above give $\lambda_{\min}(Q_rQ_r'-\lambdahat_k^2 I)\ge c\,\nu_M^2T$ with high probability, while $\|Q_r\|\le\|B_r\|+\|E_r\|\le C\nu_M\sqrt{T}$ on $\mathcal E_T$ and $\|E_c'y_k\|\le\|E_c\|\le C\sqrt{T}$, so that
\begin{equation}\label{eq:xk-apriori}
\|x_k\|\;\le\;\bigl\|(Q_rQ_r'-\lambdahat_k^2 I)^{-1}\bigr\|\,\|Q_r\|\,\|E_c'y_k\|\;\le\;C\,\nu_M^{-1},
\qquad
\|y_k\|^2=1-\|x_k\|^2\;\ge\;1-C\nu_M^{-2}.
\end{equation}
Substituting \eqref{eq2.1x} into the second row gives
\begin{align*}
\lambdahat_k^2y_k= E_c Q_r'x_k+ E_cE_c'y_k=- E_cQ_r'(Q_rQ_r'-\lambdahat_k^2 I)^{-1}Q_rE_c'y_k+ E_cE_c'y_k. 
\end{align*}

\textbf{Step 2: the case $k=r+1$.}  
Left multiply $y'$ on both sides, and rearrange, 
\begin{equation}
\label{tmp:thm1}
 y_k'E_cQ_r'(Q_rQ_r'-\lambdahat_k^2 I)^{-1}Q_rE_c'y_k=
 y_k'E_cE_c'y_k-\lambdahat_k^2\|y_k\|^2
\end{equation}
 We observe that the left-hand side is non-negative because $(Q_rQ_r'-\lambdahat_k^2 I)^{-1}$ is positive definite with high probability.
Moreover, since $\lambda_1(Q_rQ_r'-\lambdahat_k^2 I)\leq CT\nu_M^2$ with high probability, we have the positive-semidefinite inequality $(Q_rQ_r'-\lambdahat_k^2 I)^{-1}\succeq \frac{1}{CT\nu_M^2}\,I$, whence
\begin{equation}
\label{tmp:ineq1}
\frac{1}{CT\nu_M^2}\|Q_rE_c'y_k\|^2\leq y_k'E_cE_c'y_k-\lambdahat_{k}^2\|y_k\|^2.
\end{equation}
We argue by induction, beginning with the case $k=r+1$.
Thus, by claim (i), we have
\[
\frac{1}{CT\nu_M^2}\|Q_r E_c' y_{r+1}\|^2\leq \lambda_1(\Sigma^{1/2}_e  E)^2-\lambdahat_{r+1}^2= O(T^{1+2\varepsilon}\nu_M^{-2}). 
\]
So $\|Q_r E_c' y_{r+1}\|\leq C T^{1+\varepsilon}$.
\[
x_{r+1}=-(Q_rQ_r'-\lambdahat_{r+1}^2 I)^{-1}Q_r E_c'y_{r+1}=O(T^{\varepsilon}\nu_M^{-2}),
\]
which proves the claim for $r+1$.

\textbf{Step 3: induction to other cases.} Throughout this step the index $i$ ranges over $1,\dots,R-r$. Since $R-r=O(1)$ by assumption, all accumulated Gram--Schmidt errors satisfy $C_R\cdot (\prec \nu_M^{-2})$ with $C_R$ a constant depending only on the bounded difference $R-r$; we absorb $C_R$ into the implicit constant in $\prec$ below.
Next, suppose the bound holds through $k=r+i-1$. For $k=r+i$, orthonormality implies that, for any $j\leq i-1$,
\[
0=\xihat_{r+j}'\xihat_{r+i}=y_{r+j}' y_{r+i}+x_{r+j}'x_{r+i}\Rightarrow y_{r+j}'y_{r+i}=-x_{r+j}'x_{r+i}=O(T^{\varepsilon}\nu_M^{-2}),
\]
where we used the induction $\|x_{r+j}\|=O(T^{\varepsilon}\nu_M^{-2})$  and  $\|x_{r+i}\|\leq 1$.
In other words, $y_{r+j}$ is close to orthogonal to $y_{r+i}$. 
Consider the exact Gram-Schmidt transform: set $\tilde y_{r+1}:=y_{r+1}/\|y_{r+1}\|$ and, recursively for $i\ge 2$,
\[
\tilde{y}_{r+i}:=\frac{y_{r+i}-\sum_{j\leq i-1}\langle y_{r+i} ,\tilde{y}_{r+j}\rangle \tilde{y}_{r+j}}{\|y_{r+i}-\sum_{j\leq i-1}\langle y_{r+i} ,\tilde{y}_{r+j}\rangle \tilde{y}_{r+j}\|},
\]
so that $\tilde y_{r+1},\ldots,\tilde y_{r+i}$ are exactly orthonormal.
Then $\|\tilde{y}_{r+i}-y_{r+i}\|= O(T^{\varepsilon}\nu_M^{-2})$: by induction on $i$, each coefficient obeys $\langle y_{r+i},\tilde y_{r+j}\rangle=\langle y_{r+i},y_{r+j}\rangle+O(T^{\varepsilon}\nu_M^{-2})=O(T^{\varepsilon}\nu_M^{-2})$, the denominator equals $\|y_{r+i}\|+O(T^{\varepsilon}\nu_M^{-2})=1+O(T^{\varepsilon}\nu_M^{-2}+\nu_M^{-2})$ by the a priori bound \eqref{eq:xk-apriori} (which requires no induction), and $R-r=O(1)$ bounds the accumulation. Consequently,
\[
\bigl|\|E_c'\tilde{y}_{r+i}\|^2-\|E_c' y_{r+i}\|^2\bigr|\leq 2\|E_c\|^2\,\|\tilde{y}_{r+i}-y_{r+i}\|\leq C\|\Sigma^{1/2}_e  E\|^2 T^{\varepsilon}\nu_M^{-2}=O(T^{1+\varepsilon}\nu_M^{-2}).
\]
Since $i\leq R-r=O(1)$ and $\tilde{y}_{r+1:r+i}$ consists of $i$ exactly orthonormal vectors,
\begin{align*}
 \|E'_cy_{r+1:r+i}\|_{\mathrm F}^2-O(T^{1+\varepsilon}\nu_M^{-2})&\leq \|E'_c \tilde{y}_{r+1:r+i}\|_{\mathrm F}^2\\
 &\leq \sum_{j\leq i}\lambda_{j}(E_c)^2\leq
 \sum_{j\leq i}\lambda_{j}(\Sigma^{1/2}_e  E)^2\leq
\sum_{j\leq i}\lambdahat^2_{j+r}+O(T^{1+2\varepsilon}\nu_M^{-2}),
\end{align*}
where the second line uses, in order, the Ky Fan extremal inequality for the orthonormal collection $\tilde y_{r+1:r+i}$, the bound $\lambda_j(E_c)\le\lambda_j(\Sigma_e^{1/2}E)$ (as $E_c=\Xi_c'\Sigma_e^{1/2}E$ with $\|\Xi_c\|=1$), and the rearranged form of \eqref{eq:svgap}.
So we have 
\[
\|E'_cy_{r+1:r+i}\|_{\mathrm F}^2-\sum_{j\leq i}\lambdahat^2_{j+r}=O(T^{1+2\varepsilon}\nu_M^{-2})
\]
Meanwhile,
\begin{align*}
\sum_{j\leq i}\|x_{r+j}\|^2&=\sum_{j\leq i}\|(Q_{r}Q_{r}'-\lambdahat_{r+j}^2 I)^{-1}Q_{r}E_c'y_{r+j}\|^2\\
&\lesssim
\frac{1}{T^2\nu_M^4} \sum_{j\leq i}\|Q_{r}E_c'y_{r+j}\|^2\\
&\lesssim
\frac{1}{T\nu_M^2}\sum_{j\leq i}(\|E_c'y_{r+j}\|^2-\lambdahat^2_{r+j}\|y_{r+j}\|^2)\\
&\leq
\frac{1}{T\nu_M^2}\sum_{j\leq i}(\|E_c'y_{r+j}\|^2-\lambdahat^2_{r+j})+O(\nu_M^{-4})\\
&=\frac{1}{T\nu_M^2}
\left(\|E_c'y_{r+1:r+i}\|_{\mathrm F}^2-\sum_{j\leq i}\lambdahat^2_{r+j}\right)+O(\nu_M^{-4})\lesssim T^{2\varepsilon}\nu_M^{-4}.
\end{align*}
Here the second step uses the resolvent bound below \eqref{eq2.1x}; the third is \eqref{tmp:ineq1}, applied at $k=r+j$; and the fourth uses $\lambdahat_{r+j}^2\|x_{r+j}\|^2\le CT\nu_M^{-2}$ from \eqref{eq:xk-apriori}.
From this we conclude that $\|x_{r+i}\|=O(T^{\varepsilon}\nu_M^{-2})$.

\noindent\textbf{Claim~(iii): near-orthogonality, right singular vectors.}
The corresponding bound for the right side, $\|F\|_{\mathrm{F}}^{-1}\,\|F'\widehat V_{-r}\|\prec \nu_M^{-2}$, follows by symmetry. Specifically, applying the argument above to the transposed matrix $X'=M'+(\Sigma_e^{1/2}E)'\in\R^{T\times N}$ exchanges the roles of left and right singular vectors, mapping $\Xi_r\mapsto V_r$, $V_r\mapsto \Xi_r$, $U_k\mapsto W_k$, and $W_k\mapsto U_k$. The orthogonality budget $\|W_k'V_r\|\prec T^{\varepsilon-1/2}$ is the right-side counterpart of Lemma~\ref{lem:XiU}, also supplied by Proposition~\ref{prop:loallaw}~(3). All local-law inputs from \cite{knowles2017anisotropic} are symmetric in the left/right convention under our setting $N\asymp T$, $\phi\ne 1$, so the entire argument transposes verbatim, yielding $\|V_r'\widehat v_k\|\prec T^{\varepsilon}\nu_M^{-2}$ for each $k\in[r+1,R]$ and consequently $\|F\|_{\mathrm{F}}^{-1}\|F'\widehat V_{-r}\|\prec \nu_M^{-2}$ via the analogue of the equivalence between $\|B\|_{\mathrm{F}}^{-1}\|B'\widehat\Xi_{-r}\|$ and $\|\Xi_r'\widehat\xi_k\|$ established in Step~1.

\noindent\textbf{Claim~(ii): incoherence of extra eigenvectors.}
Let $u_{1},u_2,\ldots\in\R^N$ and $w_{1},w_2,\ldots\in\R^T$ denote the left and right singular vectors of $\Sigma_e^{1/2}E$ ordered by decreasing singular value. We first prove the bound $\|\eta'\widehat\xi_k\|\prec\nu_M^{-1}$ for the left singular vectors of $X$, for each $k\in\{r+1,\ldots,R\}$; the symmetric bound for $\widehat v_k$ follows by transposition (cf.\ Step~3 below).

\textbf{Step 1: expansion of $\xihat_k$ in terms of $u_i$.}
The column space of $X=M+\Sigma_e^{1/2}E$ is contained in the sum of the column spaces of $M$ (which equals $\mathrm{span}(\Xi_r)$) and of $\Sigma_e^{1/2}E$ (which is spanned by $u_1,\ldots,u_{m}$, where $m:=\min(N,T)$). Hence each empirical left singular vector $\widehat\xi_k\in\R^N$ admits the decomposition $\widehat\xi_k=\Xi_r x_k+\Xi_c y_k$ from claim~(iii), and we expand $\Xi_c y_k$ in the directions $u_i$, keeping the residual explicit:
\begin{equation}\label{eq:xihat-exp}
\Xi_c y_k=\sum_{i\in [m]} a_{k,i}\, u_i+\rho_k,\qquad a_{k,i}:=\langle u_i,\Xi_c y_k\rangle,\qquad \rho_k:=(I-P_U)\,\Xi_c y_k,
\end{equation}
where $P_U$ denotes the orthogonal projection onto $\mathrm{col}(\Sigma_e^{1/2}E)$, and we set $a_{k,i}:=0$ and $\lambda_i(\Sigma_e^{1/2}E):=0$ for $m<i\le N$. Since $\phi\ne1$, the deformed MP law keeps $\lambda_m(\Sigma_e^{1/2}E)$ bounded away from zero (of order $\sqrt T$) with high probability, so $\rank(\Sigma_e^{1/2}E)=m$, the vectors $u_1,\ldots,u_m$ form an orthonormal basis of $\mathrm{col}(\Sigma_e^{1/2}E)$, and the expansion \eqref{eq:xihat-exp} holds exactly on this event. When $\phi<1$, so that $m=N$, $\mathrm{col}(\Sigma_e^{1/2}E)=\R^N$ and $\rho_k=0$. When $\phi>1$, the vectors $\{u_i\}_{i\le T}$ span only the $T$-dimensional $\mathrm{col}(\Sigma_e^{1/2}E)$, and $\rho_k$ need not vanish; we now bound it for $k\in\{r+1,\ldots,R\}$, the range relevant for claim~(ii). Writing $\widehat\xi_k=\widehat\lambda_k^{-1}X\widehat v_k$ and using $(I-P_U)\Sigma_e^{1/2}E=0$, we have $(I-P_U)\widehat\xi_k=\widehat\lambda_k^{-1}(I-P_U)\,\Xi_r L_r V_r'\widehat v_k$, and hence, since $\rho_k=(I-P_U)\widehat\xi_k-(I-P_U)\,\Xi_r x_k$, the bounds $\widehat\lambda_k\asymp\sqrt{T}$ from claim~(i) together with Proposition~\ref{prop:loallaw}(1), $\|L_r\|\asymp\nu_M\sqrt{T}$ on $\mathcal E_T$, $\|V_r'\widehat v_k\|\prec\nu_M^{-2}$ from claim~(iii) (proved above), and $\|x_k\|=O(T^{\varepsilon}\nu_M^{-2})$ give
\begin{equation}\label{eq:rho-bound}
\|\rho_k\|\;\le\;\widehat\lambda_k^{-1}\,\|L_r\|\,\|V_r'\widehat v_k\|+\|x_k\|\;\prec\;\nu_M^{-1}.
\end{equation}
The residual is carried explicitly through the two places where the expansion \eqref{eq:xihat-exp} enters below: the Parseval identity in \eqref{tmp:Ai2}, and the delocalization decomposition of Step~2.

Recall from the induction used in the claim (iii),  where we have that
\begin{gather*}
x_i=O(T^{\varepsilon}\nu_M^{-2}),\qquad
\|y_{i}-\tilde{y}_i\|=O(T^{\varepsilon}\nu_M^{-2}),\qquad i,j\in \{r+1,\ldots,R\},\\
y'_iy_j=(\Xi_c y_i)'(\Xi_c y_j)=O(T^{\varepsilon}\nu_M^{-2})\quad (i\neq j).
\end{gather*}
In other words, $\{\Xi_c y_k\}_{k\in \{r+1,\ldots,R\}}$ is within distance $O(T^{\varepsilon}\nu_M^{-2})$ of a collection of orthonormal vectors. As a consequence, if we write $A_i=[a_{r+1,i},\ldots,a_{R,i}]'$, then the following holds for $i\le m$ with high probability:
\begin{align}
\notag
\|A_i\|^2=\sum_{j=r+1}^R \langle u_i, \Xi_c y_j\rangle^2&\leq
\sum_{j=r+1}^R \langle u_i, \Xi_c\tilde{y}_j\rangle^2+ O(T^{\varepsilon}\nu_M^{-2})\\
\label{tmp:Ai1}
&\leq \|u_i\|^2+O(T^{\varepsilon}\nu_M^{-2})=
1+O(T^{\varepsilon}\nu_M^{-2}).
\end{align}
For $m<i\le N$, the same bound is trivial because $A_i=0$ by convention.
Moreover,
\begin{equation}
\label{tmp:Ai2}
\sum_{i=1}^N\|A_i\|^2=\sum_{j=r+1}^R \sum_{i=1}^m a_{j,i}^2=
\sum_{j=r+1}^R\bigl(\|\Xi_c y_j\|^2-\|\rho_j\|^2\bigr)=(R-r)\pm O(T^{\varepsilon}\nu_M^{-2}),
\end{equation}
where the second equality is Parseval's identity on $\mathrm{col}(\Sigma_e^{1/2}E)$, and the last step uses $\|\Xi_c y_j\|^2=1-\|x_j\|^2=1-O(T^{2\varepsilon}\nu_M^{-4})$ together with $\|\rho_j\|^2\prec\nu_M^{-2}$ from \eqref{eq:rho-bound}, absorbed into the error term.
To continue, recall from \eqref{tmp:ineq1}
,\[
y_j'\Xi_c' \Sigma^{1/2}_e  EE'\Sigma^{1/2}_e  \Xi_c y_j=y_j' E_c E_c' y_j\geq \lambdahat_j^2 \|y_j\|^2,\qquad j\in\{r+1,\ldots,R\}.
\]
Next, recall claim (i) and $\|y_j\|^2=1-\|x_j\|^2\geq 1-O(T^{2\varepsilon}\nu_M^{-4})$. Since $(\Sigma_e^{1/2}E)'\rho_j=0$, we have $E'\Sigma_e^{1/2}\,\Xi_c y_j=\sum_{i\in[m]}a_{j,i}\,\lambda_i(\Sigma^{1/2}_e E)\,w_i$ exactly, so the residual does not enter the following quadratic form. We obtain, for $j\in\{r+1,\ldots,R\}$,
\begin{align*}
\sum_{i\in [N]} a_{j,i}^2\lambda_i^2(\Sigma^{1/2}_e  E)=y_j'\Xi_c' \Sigma^{1/2}_e  EE'\Sigma^{1/2}_e  \Xi_c y_j
&\geq \lambdahat_j^2 \|y_j\|^2\geq \lambda_{j-r}^2(\Sigma^{1/2}_e  E)\|y_j\|^2-O(T^{1+2\varepsilon}\nu_M^{-2})\\
&\geq \lambda_{j-r}^2(\Sigma^{1/2}_e  E)-O(T^{1+2\varepsilon}\nu_M^{-2}),
\end{align*}
where the middle inequality is \eqref{eq:svgap}, applied with $k=j-r$.

Summing these equations over $j\in\{r+1,\ldots,R\}$ gives
\[
\sum_{i\in [N]} \|A_{i}\|^2\lambda_i^2(\Sigma^{1/2}_e  E)\geq \sum_{k=1}^{R-r}\lambda_k^2(\Sigma^{1/2}_e  E)-O(T^{1+2\varepsilon}\nu^{-2}_M).
\]
We rewrite the left-hand side as
$\sum_{i\leq R-r} \|A_{i}\|^2\lambda_i^2+\sum_{i>R-r} \|A_{i}\|^2\lambda_i^2$, abbreviating $\lambda_i:=\lambda_i(\Sigma^{1/2}_e  E)$ throughout this display, and move these terms to the right-hand side:
\begin{align*}
O(T^{1+2\varepsilon}\nu_M^{-2})&\geq
\sum_{k=1}^{R-r}\lambda_k^2-
\sum_{i=1}^{R-r}\|A_i\|^2\lambda_i^2-
\sum_{i=R-r+1}^N\|A_i\|^2\lambda_i^2\\
&=\sum_{i=1}^{R-r}(1-\|A_i\|^2)\lambda_i^2
-
\sum_{i=R-r+1}^N\|A_i\|^2\lambda_i^2\\
&\geq
\lambda_{R-r}^2\sum_{i=1}^{R-r}(1-\|A_i\|^2)-
\sum_{i=R-r+1}^N\|A_i\|^2\lambda_i^2-O(T^{1+2\varepsilon}\nu_M^{-2})\\
&\geq \lambda_{R-r}^2 \sum_{i=R-r+1}^{N} \|A_i\|^2
-
\sum_{i=R-r+1}^N\|A_i\|^2\lambda_i^2-O(T^{1+2\varepsilon}\nu_M^{-2})\\
&=-O(T^{1+2\varepsilon}\nu_M^{-2})+\sum_{i=R-r+1}^N (\lambda_{R-r}^2-\lambda_i^2)\|A_i\|^2.
\end{align*}
Here the inequality introducing $\lambda_{R-r}^2$ uses \eqref{tmp:Ai1} together with $\lambda_i\ge\lambda_{R-r}$ for $i\le R-r$, and the next inequality uses \eqref{tmp:Ai2}; the $O(T^{1+2\varepsilon}\nu_M^{-2})$ errors incurred in these two steps are the ones displayed.
By Proposition \ref{prop:loallaw} regularity  condition (2), we know that with high probability,
\[
\lambda_{R-r}(\Sigma^{1/2}_e  E)^2-\lambda_k(\Sigma^{1/2}_e  E)^2\geq c k,\quad k\ge T^{\varepsilon}.
\]
 {Since $R$ is fixed and $T^\varepsilon\to\infty$, for $T$ large enough $T^\varepsilon>R-r$, so this gap inequality covers the entire range $k\ge T^\varepsilon$ used below; for $\phi>1$ and $k>m=T$, where $\lambda_k(\Sigma_e^{1/2}E)=0$ by our convention, the inequality holds trivially for $k\le N\asymp\phi T$ since $\lambda_{R-r}(\Sigma_e^{1/2}E)^2\ge cT$ by Proposition~\ref{prop:loallaw}(1).} Therefore
\[
O(T^{1+2\varepsilon}\nu_M^{-2})\geq \sum_{i\geq T^{\varepsilon}} \|A_{i}\|^2 (\lambda_{R-r}(\Sigma^{1/2}_e  E)^2-\lambda_i^2(\Sigma^{1/2}_e  E))
\geq c\sum_{k\geq T^\varepsilon} \|A_{k}\|^2 k,
\]
which by Cauchy--Schwarz further leads to
\begin{equation}
\label{tmp:CS}
\sum_{k=T^{\varepsilon}+1}^{N} \|A_{k}\|\leq \sqrt{(\sum_{k=T^{\varepsilon}+1}^{N}\|A_{k}\|^2 k)(\sum_{k=T^{\varepsilon}+1}^{N}\frac{1}k)}
= O(T^{2\varepsilon+1/2}\nu^{-1}_M). 
\end{equation}

\textbf{Step 2: delocalization for $\xihat_k$.}
Recall from \eqref{eq:xihat-exp} that
\[
\xihat_k=\Xi_r x_{k}+\sum_{i\in [N]}a_{k,i}u_i+\rho_k.
\]
Therefore,
\[
\eta'\xihat_k=\eta'\Xi_r x_{k}+\sum_{i\leq T^\varepsilon} a_{k,i} \eta'u_i
+\sum_{i\geq T^\varepsilon+1} a_{k,i} \eta'u_i+\eta'\rho_k.
\]
Recall from claim~(iii), proved above, that with high probability $\|x_k\|=O(T^{\varepsilon}\nu_M^{-2})$; hence $|\eta'\Xi_r x_{k}|= O(T^{\varepsilon}\nu_M^{-2})$.
Moreover, by \eqref{eq:rho-bound}, $|\eta'\rho_k|\le\|\rho_k\|\prec\nu_M^{-1}$, precisely the order of the bound being proved.
Meanwhile, by Proposition \ref{prop:loallaw} and $a_{k,i}\leq 1$,
\begin{align*}
\sum_{i\leq T^\varepsilon} a_{k,i} \eta'u_i&
\lesssim  T^{\varepsilon-1/2}\sum_{i\leq T^\varepsilon} |a_{k,i}|\lesssim  T^{2\varepsilon-1/2}.
\end{align*}
Moreover, by \eqref{tmp:CS}
\begin{align*}
\sum_{i\geq T^\varepsilon+1} a_{k,i} \eta'u_i&
\lesssim  T^{\varepsilon-1/2}\sum_{i\geq T^\varepsilon+1} |a_{k,i}|
\lesssim  T^{\varepsilon-1/2}\sum_{i\geq T^\varepsilon+1} \|A_{i}\|=O(T^{3\varepsilon}\nu_M^{-1}).
\end{align*}
Combining the preceding four bounds proves the claim for $\xihat_k$.

\textbf{Step 3: delocalization for $\vhat_k$.} The bound $\|\zeta'\widehat v_k\|\prec\nu_M^{-1}$ on the right side follows by the same transposition argument used for claim~(iii). Replacing the model $X=M+\Sigma_e^{1/2}E$ by its transpose $X'$ exchanges left and right singular vectors, mapping the decomposition $\widehat\xi_k=\Xi_r x_k+\sum_i a_{k,i}u_i+\rho_k$ to $\widehat v_k=V_r x_k+\sum_i a_{k,i}w_i+\rho_k'$, where $\rho_k'$ is the analogue of $\rho_k$ for the transposed model --- it vanishes with high probability when $\phi>1$ and is bounded exactly as in \eqref{eq:rho-bound} when $\phi<1$, the two aspect-ratio regimes exchanging roles under transposition --- and Steps~1--2 carry over verbatim with $\eta\in\R^N$ replaced by $\zeta\in\R^T$ and Lemma~\ref{lem:XiU}'s right-side bound $\|W_k'V_r\|\le T^{\varepsilon-1/2}$ (with high probability) supplying the orthogonality budget.

\subsection{Local laws of noise matrix}
Recall that $X=M+U$ with $U=\Sigma_e^{1/2}E$, where $E$ and $\Sigma_e$ satisfy Assumption~\ref{aspt:D0}, and that $\pi$ denotes the empirical spectral distribution of the eigenvalues of $\Sigma_e$ (Section~\ref{secextra}).

In Proposition~\ref{prop:loallaw} below, claim~3 (anisotropic delocalization of $u_{0,k}$ and $w_{0,k}$) is a direct consequence of Theorem~3.12 and Remark~3.13 of \cite{knowles2017anisotropic} and \cite{alex2014isotropic}. Claims~1 and~2 combine eigenvalue rigidity with the square-root behavior at the regular upper edge of the deformed MP law. The three claims provide, respectively, the spectral lower bound used in Step~2 of the proof of Theorem~\ref{thm:main}(i), the squared-singular-value spacing used to control $\sum_{k\geq T^\varepsilon}\|A_k\|^2$ in claim~(ii), and the entry-level delocalization used throughout (Lemmas~\ref{lem:Wkanis} and~\ref{lem:XiU}).

\begin{proposition}
\label{prop:loallaw} Suppose $N/T\to \phi\in(0,\infty)\backslash\{1\}$.
Under assumption \ref{aspt:D0}, for every fixed $\varepsilon>0$ and when $T$ is sufficiently large, the following statements hold with high probability. In this proposition set $m:=\min\{N,T\}$ and adopt the convention $\lambda_k(\Sigma_e^{1/2}E)=0$ for $m<k\le N$:
\begin{enumerate}
    \item $\lambda_j(\Sigma^{1/2}_e  E)>c_0\sqrt{T}$ for all $j\leq R$.
    \item  $|\lambda_j(\Sigma^{1/2}_e  E)^2-\lambda_k(\Sigma^{1/2}_e  E)^2|\geq c_0 k$ for all $j\leq R$ and $T^\varepsilon<k\le N$.
    \item Let $u_{0,k}$ and $w_{0,k}$ be the $k$-th left and right singular vectors of $U$. 
For any two groups of norm-$1$ vectors $\mathcal{A}\subset \mathbb{R}^{N},\mathcal{B}\subset \mathbb{R}^T$ that are independent of both $E$ and $\Sigma_e^{1/2}$ (but may depend on $M$), suppose their cardinality $|\mathcal{A}|+|\mathcal{B}|\leq N^D$ for some fixed power $D$, then 
\[
\max_{\eta\in \mathcal{A}}|\eta' u_{0,j}|\leq C T^{\varepsilon-1/2},\quad
\max_{\zeta\in \mathcal{B}}|\zeta' w_{0,j}|\leq C T^{\varepsilon-1/2},\quad \forall j\le \min\{T, N\}.
\]  

\end{enumerate}
\end{proposition}
\begin{proof}

To obtain the first two claims, let
\[
q_k:=\frac{\lambda_k(\Sigma_e^{1/2}E)^2}{T}
\]
and let $\gamma_{T,k}$ be the corresponding classical location under the deformed MP law $\rho_T\equiv\varrho$ of Appendix~\ref{sec:mp}:
\[
\int_{\gamma_{T,k}}^\infty \rho_T(x)\,dx=\frac{k}{T}-\frac{1}{2T}.
\]
Thus both $q_k$ and $\gamma_{T,k}$ are on the \emph{squared}-singular-value scale. By the regular-upper-edge part of Assumption~\ref{aspt:D0}(ii) and Definition~\ref{def:regularity}, the upper edge $a_1$ is bounded above and away from zero uniformly in $T$, and the density has the regular square-root behavior
\begin{equation}\label{eq:regular-edge-density}
c\sqrt{a_1-x}\ \le\ \rho_T(x)\ \le\ C\sqrt{a_1-x}
\end{equation}
for $x$ in a fixed left neighborhood of $a_1$ intersected with the support. Consequently,
\begin{equation}\label{eq:classical-edge-locations}
c\Bigl(\frac{k}{T}\Bigr)^{2/3}
\ \le\ a_1-\gamma_{T,k}\ \le\
C\Bigl(\frac{k}{T}\Bigr)^{2/3}
\end{equation}
as long as $\gamma_{T,k}$ remains in that neighborhood. If $k$ lies below this upper-edge neighborhood (or in a lower bulk component), regularity and the fixed separation of distinct support components instead give $a_1-\gamma_{T,k}\ge c$.

The eigenvalue-rigidity estimate of Theorem~3.12 of \cite{knowles2017anisotropic}, applied on the squared-singular-value scale, yields for every fixed $\delta>0$
\begin{equation}\label{eq:squared-rigidity}
|q_k-\gamma_{T,k}|
\ \le\
C\,T^\delta\{\min(N_k,N_k^-)\}^{-1/3}T^{-2/3}
\end{equation}
with high probability, where $N_k$ and $N_k^-$ are the forward and reverse ranks of $\gamma_{T,k}$ in its support component. For fixed $R$, \eqref{eq:classical-edge-locations} and \eqref{eq:squared-rigidity} give $q_R\ge a_1-o(1)\ge c$, proving claim~1 after taking square roots.

For claim~2, first consider $T^\varepsilon<k\le m$. Fix the $\varepsilon>0$ in the proposition and choose $\delta\in(0,\varepsilon)$ in \eqref{eq:squared-rigidity}. For fixed $j\leq R$ and $k$ whose classical location remains in the upper-edge neighborhood, \eqref{eq:classical-edge-locations} gives
\[
T(\gamma_{T,j}-\gamma_{T,k})\ge c\,T^{1/3}k^{2/3}.
\]
The rigidity error on the squared-singular-value scale is at most
$C T^{1/3+\delta}k^{-1/3}=o(T^{1/3}k^{2/3})$, because $k>T^\varepsilon$ and $\delta<\varepsilon$. Hence
\[
\lambda_j(\Sigma_e^{1/2}E)^2-\lambda_k(\Sigma_e^{1/2}E)^2
=T(q_j-q_k)\ge c\,T^{1/3}k^{2/3}\ge c\,k.
\]
If $\gamma_{T,k}$ lies below the fixed upper-edge neighborhood or in a lower component, the deterministic location gap is bounded below by a positive constant, and rigidity instead gives $T(q_j-q_k)\ge cT\ge ck$. Finally, if $m<k\le N$ (a range that occurs only when $N>T$), our convention gives $\lambda_k(U)=0$, while claim~1 gives $\lambda_j(U)^2\ge cT\ge ck$ because $k\le N\asymp T$. This proves claim~2.

As for claim~3, (anisotropic delocalization of $u_{0,k}$ and $w_{0,k}$): Remark~3.13 of \cite{knowles2017anisotropic} along with Theorem 3.13 in \cite{alex2014isotropic} indicate that for any unit norm  $\eta_i\in \mathbb{R}^N$ and $\zeta_i\in \mathbb{R}^T$ that are independent of $E$, any fixed $\varepsilon>0$ and $D>0$, the following take place 
\[
\mathbb{P}(|\langle \eta_i, u_{0,k}\rangle|^2+|\langle \zeta_i, w_{0,k}\rangle|^2\geq N^{\varepsilon-1})\leq N^{-D-1}.
\]
Our claim applies a union bound over all $\eta_i\in \mathcal{A}$ and $\zeta_i\in \mathcal{B}$.
\end{proof}

\section{Proofs for Section \ref{secpca}}

 \subsection{Proof of Corollary \ref{coro2.1}}

\begin{lemma}[if $r\geq 1$]\label{lem2H}
Under the assumptions of Theorem~\ref{thm:main} with $r\ge1$, suppose additionally that $T^{\tau}=o(\nu_M^{2})$ for some constant $\tau>0$. Then

(i)  There is  $r\times r$ matrix  $H_r$ so that $
\frac{1}{\sqrt{N}} \|\widehat B_r   - B H_r\|= O_P(
\frac{1}{\nu_M}  )
$  

(ii) $\lambda_{\min}(H_r)\asymp \lambda_{\max}(H_r)=O_P(\nu_{\min}^{-1/2})$

(iii)  $\frac{1}{TN} \|G_T'U'\widehat B_r\|=O_P(T^{-1} \nu_{\min}^{-1/2})$
for any $T\times 1$ vector $G_T$ that is independent of $U$ and that $\|G_T\|= O_P(\sqrt{T})$. 

(iv) Recall $H:=\frac{1}{N}\widehat B'   B $, $R\times r$.  Then $\lambda_{\min}(H)\asymp \lambda_{\max}(H)=O_P(\nu_{\min}^{1/2})$.

\end{lemma}

\begin{proof}
Because 
$
  \nu_{\min}^{-1} \asymp  N\nu_M^{-2}.
$ and $T\asymp N$. 
Hence the hypothesis $T^{\tau}=o(\nu_M^{2})$ is equivalent to $T^{\tau-1}=o(\nu_{\min})$.

(i) Let that $\widetilde L_r$ be the $r\times r$ diagonal matrix of top $r$ eigenvalues of $XX'/T$.
By definition, 
$$
  XX' \widehat B_r=T\widehat B_r \widetilde L_r.
$$
With $X= BF'+U$, we can define $H_r:=\frac{1}{T} F'X'\widehat B_r \widetilde L_r^{-1}$. Then 
$$
 \widehat B_r -  B H_r
 =\frac{1}{T}UX'\widehat B_r\widetilde L_r^{-1}.
$$
Classical PCA analysis shows $\widetilde L_r^{-1}= O_P(\nu_{M}^{-2})$, also $\|X\|= O_P(\sqrt{T}\nu_{M})$, $\|U\|= O_P(T^{1/2})$. Thus 
$$
\frac{1}{\sqrt{N}} \|\widehat B_r   - B H_r\|= O_P(\nu_M^{-1} ).
$$

(ii) The desired result is standard analysis   the first $r$ eignvectors,  whose proof appears in many places when factors are strong. We provide a proof below for completeness, which also allows weaker factors.
Recall from (i) that $\|\widetilde L_r^{-1}\|=O_P(\nu_M^{-2})$ and $\|X\|=O_P(\sqrt{T}\nu_M)$. Since $\|F\|=O_P(\sqrt T)$ and $\|\widehat B_r\|=\sqrt N$, the definition of $H_r$ and $\nu_{\min}\asymp \nu_M^2/N$ give
\[
\|H_r\|\leq \frac1T\|F\|\,\|X\|\,\|\widehat B_r\|\,\|\widetilde L_r^{-1}\|
=O_P(\sqrt N\,\nu_M^{-1})=O_P(\nu_{\min}^{-1/2}).
\]
To bound the minimum singular value of $H_r$. Note $\frac{1}{N}\widehat B_r'\widehat B_r=I_r= H_r'\frac{1}{N}B'B H_r+ o_{\mathrm{P}}(1).$ 
Let $x$ denote the eigenvector of $H_r'H_r$ corresponding to its minimum eigenvalue. Then
\begin{eqnarray*}
    \lambda_{\min}(H_r)^2&=&\lambda_{\min}(H_r'H_r)=x'H_r'H_rx\geq \frac{x'H_r'\frac{1}{N}B'BH_rx}{\lambda_{\max}(\frac{1}{N}B'B)}\geq \frac{\lambda_{\min}(H_r'\frac{1}{N}B'BH_r)}{\lambda_{\max}(\frac{1}{N}B'B)}\cr 
    &\geq& \frac{N}{ \nu_M^2} (1-o_P(1))\geq \nu_{\min}^{-1}(1-o_P(1)).
\end{eqnarray*}

(iii)  {We bound the two pieces of the decomposition $\widehat B_r=BH_r+(\widehat B_r-BH_r)$:
\[
\tfrac{1}{NT}\|\widehat B_r' UG_T\|\;\le\;\tfrac{\|H_r\|}{NT}\,\|B'UG_T\|\;+\;\tfrac{1}{NT}\,\|\widehat B_r-BH_r\|\,\|UG_T\|.
\]
For the first piece, we condition on $G_T$, which is independent of $U$. Since $\E[Ugg'U'\mid G_T=g]=\|g\|^2\,\Sigma_e$ for any fixed $g\in\R^T$ (using the independent, mean-zero, unit-variance entries of $E$),
$$
\E[\|B'UG_T\|^2\mid G_T]=\|G_T\|^2\,\tr(B'\Sigma_e B)
\leq C\|G_T\|^2\|B\|_{\mathrm F}^2=O_P(\nu_M^2T),
$$
where the last step uses $\|G_T\|=O_P(\sqrt{T})$ and $\|B\|_{\mathrm F}^2\asymp\nu_M^2$; conditional Chebyshev then gives $\|B'UG_T\|=O_P(\sqrt{T}\,\nu_M)$. The conditional argument requires no moment conditions on the entries of $G_T$, which matters when the lemma is applied with $G_T=F$, whose rows need not be bounded under Assumption~\ref{aspt:X}(i).
So with $\|H_r\|=O_P(\nu_{\min}^{-1/2})$ and   $\nu_M\asymp\sqrt{N\nu_{\min}}$,
$$\tfrac{\|H_r\|}{NT}\,\|B'UG_T\|=O_P(\frac{1}{\sqrt{ NT}}).$$
For the second piece, $\|UG_T\|\le\|U\|\,\|G_T\|= O_P(T)$ and $\|\widehat B_r-BH_r\|=O_P(\sqrt{N}\,\nu_M^{-1})$ from~(i), so
\[
\tfrac{1}{NT}\,\|\widehat B_r-BH_r\|\,\|UG_T\|\;=\;O_P\Bigl(\tfrac{\sqrt{N}\,\nu_M^{-1}\cdot T}{NT}\Bigr)\;= \;O_P(T^{-1}\nu_{\min}^{-1/2}).
\]
Combining the two pieces gives the claim.}

 (iv)   By (iii), because $T^{-1}=o(\nu_{\min})$,
$$
\frac{1}{N} B'\widehat B_r   =\frac{1}{N}S_f^{-1}H_r\widetilde L_r -\frac{1}{TN} S_f^{-1}F'U'\widehat B_r= \frac{1}{N}S_f^{-1}H_r\widetilde L_r  + O_P(T^{-1} \nu_{\min}^{-1/2}).
$$

 This shows   $\lambda_{\min}(\frac{1}{N} B'   \widehat B_r)\asymp  \nu_{\min}^{1/2}$. Moreover, $\lambda_{\max}(H)\le \lambda_1(S_B)^{1/2}\asymp\nu_{\min}^{1/2}$, since $H'=N^{-1}B'\widehat B$ with $\|N^{-1/2}\widehat B\|=1$, while $\lambda_{\min}(H)\ge c_0\,\nu_{\min}^{1/2}$ with probability tending to one by Proposition~\ref{prop:H-invert}, whose proof (Appendix~\ref{secHinvert}) relies only on Theorem~\ref{thm:main} and identity~\eqref{identBF} and is independent of this lemma. Hence $\lambda_{\min}(H)\asymp \lambda_{\max}(H)\asymp \nu_{\min}^{1/2}$ with probability tending to one.

\end{proof}

 \textbf{Proof of Corollary \ref{coro2.1}}

\begin{proof}
  
 First, Lemma \ref{lem2H}, applied columnwise to the $K=O(1)$ columns of $G_T$, shows $\frac{1}{TN} \|G_T'U'\widehat B_r\|=O_P(T^{-1} \nu_{\min}^{-1/2})$.
For the extra block, Theorem~\ref{thm:main}(i)--(ii), applied columnwise to $G_T$, gives
\[
\|\widehat L_{-r}\widehat V_{-r}'G_T\|
\le \|\widehat L_{-r}\|\,\|\widehat V_{-r}'G_T\|
\prec T\nu_M^{-1}.
\]
Moreover, Theorem~\ref{thm:main}(iii), $\|B\|_{\mathrm F}\asymp\nu_M$, and $\|F'G_T\|\le\|F\|\,\|G_T\|_{\mathrm F}=O_P(T)$ give
\[
\|\widehat\Xi_{-r}'MG_T\|
\le\|\widehat\Xi_{-r}'B\|\,\|F'G_T\|
\prec T\nu_M^{-1}.
\]
Since $\widehat B_{-r}=\sqrt N\,\widehat\Xi_{-r}$ and
$\widehat\Xi_{-r}'UG_T=\widehat L_{-r}\widehat V_{-r}'G_T-\widehat\Xi_{-r}'MG_T$, it follows that
\[
\frac{1}{NT}\|\widehat B_{-r}'UG_T\|
\prec \frac{\sqrt N}{NT}\,\frac{T}{\nu_M}
\asymp T^{-1/2}\nu_M^{-1}.
\]
Combining the leading and extra blocks and using $\nu_{\min}\asymp\nu_M^2/N$ yields $\frac{1}{NT}\|\widehat B'UG_T\|\prec T^{-1/2}\nu_M^{-1}$. We prove the factor-side bound directly, without transposing the model. First consider the leading block. Conditionally on $(F,G_N)$, the columns of $U$ are independent, $V_r$ is fixed and orthonormal, and
\[
\E\!\left[\|V_r'U'G_N\|_{\mathrm F}^2\mid F,G_N\right]
=r\,\tr(G_N'\Sigma_eG_N)\le C\|G_N\|_{\mathrm F}^2=O_P(T).
\]
Hence $\|V_r'U'G_N\|=O_P(\sqrt T)$. By Davis--Kahan, for some orthogonal $O_r$,
\[
\|\widehat V_r-V_rO_r\|_{\mathrm F}=O_P(\nu_M^{-1}),
\]
and therefore
\[
\|\widehat V_r'U'G_N\|
\le O_P(\sqrt T)+O_P(\nu_M^{-1})\|U\|\,\|G_N\|
=O_P(\sqrt T+T\nu_M^{-1}).
\]
Since $\|\widehat L_r\|/\sqrt N=O_P(\nu_M)$, $N\asymp T$, and $\nu_M\le C\sqrt T$, the leading block satisfies
\[
\frac{1}{NT\sqrt N}\|\widehat L_r\widehat V_r'U'G_N\|
=O_P(\nu_MT^{-3/2}+T^{-1})
=O_P(T^{-1/2}\nu_M^{-1}).
\]
For the extra block, Theorem~\ref{thm:main}(ii), applied columnwise to $G_N$, and Theorem~\ref{thm:main}(iii) give
\[
\|\widehat\Xi_{-r}'G_N\|\prec\sqrt T\,\nu_M^{-1},
\qquad
\|\widehat V_{-r}'F\|\prec\sqrt T\,\nu_M^{-2}.
\]
Using $\widehat V_{-r}'U'G_N=\widehat L_{-r}\widehat\Xi_{-r}'G_N-\widehat V_{-r}'FB'G_N$, together with $\|\widehat L_{-r}\|=O_P(\sqrt T)$ and $\|B'G_N\|\le\|B\|_{\mathrm F}\|G_N\|=O_P(\nu_M\sqrt T)$, yields
\[
\|\widehat V_{-r}'U'G_N\|\prec T\nu_M^{-1}.
\]
Because $\|\widehat L_{-r}\|/\sqrt N=O_P(1)$,
\[
\frac{1}{NT\sqrt N}\|\widehat L_{-r}\widehat V_{-r}'U'G_N\|
\prec T^{-1}\nu_M^{-1}
\le C T^{-1/2}\nu_M^{-1}.
\]
Combining the leading and extra blocks proves $\frac{1}{NT}\|\widehat F'U'G_N\|\prec T^{-1/2}\nu_M^{-1}$.

 \end{proof}

\subsection{Proof of Theorem \ref{thm:common}}

\begin{proof}
Recall the SVD $M= \Xi_{r}L_{r}V'_{r}$ from Section~\ref{secmodel}. By the Davis--Kahan $\sin\Theta$ theorem, for the SVD of $X$ there exists a rotation $O_{r}$ such that $\|\Vhat_{r}-V_{r}O_{r}\|_{\mathrm{F}}=O_P(\nu^{-1}_M)$. Note that $\widehat M= X \Vhat_{R}\Vhat_{R}'$ and
$M V_rV_r'= M$.
We decompose
\begin{align*}
X \Vhat_{R}\Vhat_{R}'-M&=
X \Vhat_{r}\Vhat_{r}'-M+X \Vhat_{-r}\Vhat_{-r}'\\
&=M(\Vhat_r\Vhat_r'-V_rV_r')
+U \Vhat_r\Vhat_r'+ M\Vhat_{-r}\Vhat_{-r}'+U\Vhat_{-r}\Vhat_{-r}'.
\end{align*}
We bound:  $\|U\Vhat_{r}\Vhat'_{r}\|_{\mathrm{F}}\le\sqrt{r}\,\|U\|= O_P(\sqrt{T})$,
\[
\|M(\Vhat_r\Vhat_r'-V_rV_r')\|_{\mathrm{F}}=O_P(\sqrt{T}),\quad
\|U\Vhat_{-r}\Vhat'_{-r}\|_{\mathrm{F}}\leq  \|\widehat V_{-r}\|_{\mathrm{F}}^2\|U\|=O_P(\sqrt{T})
 \]
 and, writing $M\Vhat_{-r}=X\Vhat_{-r}-U\Vhat_{-r}=\widehat\Xi_{-r}\widehat L_{-r}-U\Vhat_{-r}$, where $\|\widehat L_{-r}\|=\widehat\lambda_{r+1}\le\lambda_1(\Sigma_e^{1/2}E)$ by Theorem~\ref{thm:main}(i) and $\lambda_1(\Sigma_e^{1/2}E)=O_P(\sqrt T)$ under Assumption~\ref{aspt:D0},
$$
\|M\Vhat_{-r}\Vhat_{-r}'\|_{\mathrm{F}}\;\le\;\sqrt{R-r}\,\bigl(\|\widehat L_{-r}\|+\|U\|\bigr)\;=\;O_P(\sqrt{T})
$$
(Theorem~\ref{thm:main}(iii) refines this to $\prec\nu_M^{-1}\sqrt T$, but the crude bound suffices here). Together,
$$
\frac{1}{\sqrt{NT}}\|\widehat M- M\|_{\mathrm{F}}=O_P\Bigl(\frac{\sqrt{T}}{\sqrt{NT}}\Bigr)= O_P(T^{-1/2})
$$
given $T\asymp N$. {The same conclusion holds in the  case $r=0$ ($M=0$), where the first three terms in  $X \Vhat_{R}\Vhat_{R}'-M$ vanish and the remaining term $\|U\widehat V_{-r}\widehat V_{-r}'\|_{\mathrm{F}}= O_P(\sqrt{T})$. }

\end{proof}

\subsection{Proof of Theorem \ref{thm:factor}}

 \begin{proof}
In this theorem we assume $T^{\tau}=o(\nu_M)$ for some constant $\tau>0$; since $\nu_M^2\asymp N\nu_{\min}$ and $N\asymp T$, this implies $T^{\tau-1}=o(\nu_{\min})$, and in particular Lemma~\ref{lem2H} and Corollary~\ref{coro2.1} apply.
   Note   $H'H^+= I_r$.

(i)(ii)  By Lemma \ref{lem2H}, $
\|H\|=    O_P(\nu_{\min}^{1/2}).
$ 
In addition, recall from \eqref{eq2.1} and \eqref{eq2.2} --- the latter valid on the event $\{\lambda_r(H)>0\}$, which has probability tending to one by Proposition~\ref{prop:H-invert} --- that $\widehat F = F H'+ \frac{1}{N}U'\widehat B$ and $\widehat F H^+ = F +\frac{1}{N}U'\widehat B H^+$.

Recalling $T\asymp N$, we have 
\begin{eqnarray*}
\frac{1}{\sqrt{T}}\|\widehat F H^+ - F\|
&=&\frac{1}{\sqrt{T}}\|\frac{1}{N}U'\widehat B H^+\| = \frac{1}{N } \| U\|  O_P(\nu^{-1/2}_{\min})= O_P(T^{-1/2} \nu^{-1/2}_{\min})\cr
\frac{1}{\sqrt{T}}\|\widehat F- FH'\|
&\leq&  \frac{1}{N } \| U\|  =  O_P(T^{-1/2} ).
\end{eqnarray*}

Now set $G_T=F$  in Corollary \ref{coro2.1} to reach
$$
\|\frac{1}{TN} F'U'\widehat B   \|=O_p(  T^{\varepsilon-1}     \nu_{\min}^{-1/2}) .$$
Hence 
\begin{eqnarray*}
\frac{1}{T}\|F'(\widehat F H^+ - F)\| &\leq& \|\frac{1}{TN}F'U'\widehat B  \|O_P(\nu^{-1/2}_{\min})=O_P(T^{\varepsilon-1}     \nu_{\min}^{-1} ) \cr
\frac{1}{T}\|F'(\widehat F- FH')\| &\leq& \|\frac{1}{TN}F'U'\widehat B  \| = O_P( T^{\varepsilon-1}     \nu_{\min}^{-1/2}) .
\end{eqnarray*}
Since $\varepsilon>0$ is arbitrary, and $T^{-1}\nu_{\min}^{-1}\asymp\nu_M^{-2}$, $T^{-1}\nu_{\min}^{-1/2}\asymp T^{-1/2}\nu_M^{-1}$ (using $\nu_{\min}\asymp\nu_M^{2}/N$ and $N\asymp T$), this proves the stochastic-domination bounds in (i) and (ii).

(iii) From (\ref{identBF}), 
  $  \frac{1}{\sqrt{N}}B= \Xi_rH_B^{-1} .
    $ This implies 
    $
  \Xi_r L_r V_r'=  M= BF' =\sqrt{N}\Xi_rH_B^{-1}F', 
    $ yielding
    $$
F'=N^{-1/2}H_B  L_r V_r',\quad F'F=\frac{1}{N}H_B  L_r ^2H_B',\quad (F'F)^{-1}= N H_B^{'-1}  L_r ^{-2}H_B^{-1}.
    $$
Meanwhile, $(\widehat F'\widehat F)^{-1}=   N \widehat L_R^{-2}$, and $H=\frac{1}{N}\widehat B'   B
= \widehat\Xi_R'\Xi_rH_B^{-1}, 
$  yielding
$$
H'\left(\widehat  F'\widehat F\right)^{-1}H= N H_B^{'-1}\Xi_r'\widehat\Xi_R\widehat L_R^{-2} 
\widehat\Xi_R'\Xi_rH_B^{-1}.$$ Also, $\|H_B^{-1}\|\leq \|S_B\|^{1/2}\lesssim\nu_{\min}^{1/2}$.
This implies, for $\nu_M^2\asymp N\nu_{\min}$,  
\begin{eqnarray*}
    \|H'\left(\widehat  F'\widehat F\right)^{-1}H- \left( F' F\right)^{-1}\|
 &\leq& N\| H_B^{-1}\|^2\| \Xi_r'\widehat\Xi_R\widehat L_R^{-2} 
\widehat\Xi_R'\Xi_r-L_r^{-2}\|\cr 
&\lesssim&   \nu_M^2 \|\Xi_r'\Xihat_r \Lhat^{-2}_r\Xihat_r'\Xi_r-L_r^{-2}\|+\nu_M^2 \|\Xi_r'\Xihat_{-r}\Lhat^{-2}_{-r}\Xihat_{-r}'\Xi_r\|\\
&\lesssim& \nu_M^2 \|\Xihat_r \Lhat^{-2}_r\Xihat_r'-\Xi_rL_r^{-2}\Xi'_r\|+\nu_M^2 \|\Xi_r'\Xihat_{-r}\|^2\|\Lhat^{-2}_{-r}\|\\
&=& O_P\bigl(T^{\varepsilon}(T^{-1}\nu_M^{-1}+T^{-1}\nu_M^{-2})\bigr),
\end{eqnarray*}
 where the second last line uses Lemma \ref{lem:condprod}, and the last line first use Davis--Kahan theorem for the $\Xi_r,L_r$ perturbation, and the second item  follows from Theorem \ref{thm:main}. Hence
  $$ H'\left(\frac{1}{T}\widehat  F'\widehat F\right)^{-1} H-\left(\frac{1}{T}F'  F\right)^{-1} = O_P\bigl(T^{\varepsilon}(\nu_M^{-1}+\nu_M^{-2})\bigr)=o_P(1),
  $$
  upon choosing $\varepsilon\le\tau$, so that $T^{\varepsilon}=o(\nu_M)$.

 \end{proof}

\subsection{Proof of Proposition \ref{prop:H-invert}}
\label{secHinvert}

\begin{proof}
Recall $H'=\frac{1}{N} B'\widehat B = \frac{1}{\sqrt{N}}B' \widehat\Xi_R$ since $\widehat B = \sqrt{N}\widehat\Xi_R$. Using identity~(\ref{identBF}), $\frac{1}{\sqrt{N}}B = \Xi_r H_B^{-1}$, hence
\[
H' \;=\; H_B^{-1'}\,\Xi_r'\widehat\Xi_R \;=\; H_B^{-1'}\,[\,\Xi_r'\widehat\Xi_r,\ \Xi_r'\widehat\Xi_{-r}\,].
\]
 {The two blocks contribute on different scales: the leading $r\times r$ block $\Xi_r'\widehat\Xi_r$ is bounded below by Davis--Kahan, while the $r\times(R-r)$ overestimated block $\Xi_r'\widehat\Xi_{-r}$ is asymptotically negligible by Theorem~\ref{thm:main}(iii). } 
 Also $\frac{1}{\sqrt{N}}\|H_B'\|\|B\|= O(1)$ by (\ref{eb.1}). 
  By the Davis--Kahan $\sin\Theta$ theorem combined with the singular-value separation $\lambda_r(X)\ge c\,\nu_M\sqrt T-O_P(\sqrt T)$ versus $\lambda_{r+1}(X)\le\lambda_1(\Sigma_e^{1/2}E)=O_P(\sqrt T)$ (Assumption~\ref{aspt:X}(iii) with Weyl's inequality, and Theorem~\ref{thm:main}(i)), $\lambda_r(\Xi_r'\widehat\Xi_r)\geq 1-O_P(\nu_M^{-1})\geq \tfrac12$ with probability tending to one. Moreover, since the columns of $\widehat\Xi_R=[\widehat\Xi_r,\widehat\Xi_{-r}]$ are orthonormal, for every unit vector $x\in\R^r$,
 $$
 \|\widehat\Xi_{-r}'\,\Xi_r x\|^2\;=\;\|\widehat\Xi_R'\,\Xi_r x\|^2-\|\widehat\Xi_r'\,\Xi_r x\|^2\;\le\; 1-\lambda_r(\Xi_r'\widehat\Xi_r)^2\;=\;O_P(\nu_M^{-1}),
 $$
 so $\|\Xi_r'\widehat\Xi_{-r}\|=O_P(\nu_M^{-1/2})=o_P(1)$; Theorem~\ref{thm:main}(iii) refines this to $\prec\nu_M^{-2}$, but the crude bound suffices here. Together,
\[
\lambda_r(\Xi_r'\widehat\Xi_R)\;\geq\;\lambda_r(\Xi_r'\widehat\Xi_r)-\|\Xi_r'\widehat\Xi_{-r}\|\;\geq\;\tfrac14
\]
with probability tending to one. Since Assumption~\ref{aspt:X}(iii) implies that the singular values of $H_B^{-1}$ are of order $\nu_{\min}^{1/2}$, it follows that
\[
\lambda_r(H')\;\geq\; \lambda_r(H_B^{-1'})\lambda_r(\Xi_r'\widehat\Xi_R)\;\geq\; c_0\nu_{\min}^{1/2}
\]
for a constant $c_0>0$ depending only on $(c,C,\phi)$ in Assumptions~\ref{aspt:D0}--\ref{aspt:X}, with probability $1-o(1)$. Finally, $\|H^+\|=1/\lambda_r(H)\leq c_0^{-1}\nu_{\min}^{-1/2}=O_P(\nu_{\min}^{-1/2})$, which proves the claim.
\end{proof}

\section{Proofs for Section \ref{secapp}}\label{sec:proofsecapp}

 \textbf{Proof of Theorem \ref{thinf}}

\begin{proof}
The estimator is the just-identified IV slope $\widehat\beta=(\widehat\varepsilon_z'\widehat\varepsilon_g)^{-1}\widehat\varepsilon_z'\widehat\varepsilon_y$ on the residualization of $(Y,G,Z)$ on $[1_T,\widehat F]$. By the Frisch--Waugh--Lovell identity,
\[
\widehat\varepsilon_a \;=\; (I-P_{[1_T,\widehat F]})A \;=\; (I-P_{\widehat F^c})\widetilde A,\qquad A\in\{Y,G,Z\},
\]
where $\widetilde A:=A-\bar a\,1_T$ is the sample-demeaned regressor and $\widehat F^c:=(I-P_{1_T})\widehat F$ is the column-demeaned PCA factor estimator (with the orthonormal basis taken via QR). Substituting the model~\eqref{eq4.1} kills the intercepts $\mu_y,\mu_g,\mu_z$ inside the demeaning,
\[
\widetilde Y \;=\; F^c\alpha_y + \widetilde\varepsilon_y,\qquad
\widetilde G \;=\; F^c\alpha_g + \widetilde\varepsilon_g,\qquad
\widetilde Z \;=\; F^c\alpha_z + \widetilde\varepsilon_z,
\]
where $F^c:=(I-P_{1_T})F$, $\widetilde\varepsilon_a:=\varepsilon_a-\bar\varepsilon_a\,1_T$ for $a\in\{y,g,z\}$, $\alpha_y=\beta\alpha_g+\rho$, and $\varepsilon_y=\beta\varepsilon_g+\eta$. 

Let $M_1:=I-P_{1_T}$. Assumption~\ref{ass4.1}(iv) gives $T^{-1}1_T'F=O_P(T^{-1/2})$, and Theorem~\ref{thm:factor}(i) gives $\|\widehat F-FH'\|=O_P(1)$; hence $T^{-1}1_T'\widehat F=O_P(T^{-1/2})$. It follows that both demeaning perturbations $F^c-F$ and $\widehat F^c-\widehat F$ have operator norm $O_P(1)$. Moreover,
\[
\frac1T F^{c\,\prime}F^c=S_f-\bar f\bar f'=S_f+o_P(1),
\]
so $F^c$ retains rank $r$ and a well-conditioned empirical covariance with probability tending to one.

For precision, let $V_r^c$ be the orthonormal basis obtained by applying QR to $M_1V_r$, let $\widehat V_r^c$ be the QR basis of $M_1\widehat V_r$, and let $\widehat V_{-r}^c$ be an orthonormal basis of $(I-P_{\widehat V_r^c})M_1\widehat V_{-r}$. Then
\[
P_{F^c}=V_r^cV_r^{c\,\prime},\qquad
P_{\widehat F^c}
=\widehat V_r^c\widehat V_r^{c\,\prime}
+\widehat V_{-r}^c\widehat V_{-r}^{c\,\prime}.
\]
These are QR bases of the demeaned fitted spaces, not singular vectors of a separately re-estimated demeaned panel. To transfer the previously established bounds, write $q:=T^{-1/2}1_T$. Assumption~\ref{ass4.1}(iv), Davis--Kahan, and Theorem~\ref{thm:main}(ii) give
\[
\|q'V_r\|=O_P(T^{-1/2}),\qquad
\|q'\widehat V_r\|=O_P(T^{-1/2}+\nu_M^{-1}),\qquad
\|q'\widehat V_{-r}\|\prec\nu_M^{-1}.
\]
Elementary fixed-dimensional QR perturbation then shows that the leading-subspace Davis--Kahan bound, the extra-vector incoherence bound, and the near-orthogonality bound of Theorem~\ref{thm:main}(iii) continue to hold for $(V_r^c,\widehat V_r^c,\widehat V_{-r}^c)$. For the last assertion, note in particular that $F^{c\,\prime}M_1=F^{c\,\prime}$ and that the coefficient introduced by orthogonalizing $M_1\widehat V_{-r}$ against $M_1\widehat V_r$ is
\[
\widehat V_r'M_1\widehat V_{-r}
=-(q'\widehat V_r)'(q'\widehat V_{-r})\prec\nu_M^{-2},
\]
up to the harmless $T^{-1/2}\nu_M^{-1}$ term. Below we use these centered QR bases throughout.


\smallskip\textbf{Step 1: residual approximation.} For each $a\in\{y,g,z\}$,
\begin{equation}\label{eqd.0}
    \widetilde\varepsilon_a-\widehat\varepsilon_a
=(P_{\widehat F^c}-P_{F^c})\widetilde A+P_{F^c}\widetilde\varepsilon_a
=(\widehat V_r^c\widehat V_r^{c\,\prime}-V_r^cV_r^{c\,\prime})\widetilde A
+\widehat V_{-r}^c\widehat V_{-r}^{c\,\prime}\widetilde A+P_{F^c}\widetilde\varepsilon_a.
\end{equation}

Since the left-hand side of above depends on $\widehat V_r^c,V_r^c$ only through the projectors, and
the sin-theta theorem gives $\|\widehat V_r^c-V_r^c\|= O_P(\nu_M^{-1})$, so
\begin{equation}\label{eqmainrA}
\frac{1}{\sqrt T}\|(\widehat V_r^c\widehat V_r^{c\,\prime}-V_r^cV_r^{c\,\prime})\widetilde A\|_{\mathrm F}= O_P(\nu_M^{-1}).
\end{equation}
Write $\widetilde A=F^c\alpha_a+\widetilde\varepsilon_a$. 
{Meanwhile, by Theorem~\ref{thm:main}(iii) for the first term, and using the i.i.d.\ structure of Assumption~\ref{ass4.1}(i) and that $\widetilde\varepsilon_a$ is independent of $X$ to obtain $\| V_{r}^cV_{r}^{c\,\prime}\widetilde\varepsilon_a\| +\| \widehat V_{-r}^c\widehat V_{-r}^{c\,\prime}\widetilde\varepsilon_a\|= O_P(1)$}.
\begin{eqnarray}\label{equvvA}
\frac{1}{\sqrt T}\|\widehat V_{-r}^c\widehat V_{-r}^{c\,\prime}\widetilde A\|
&\leq& \frac{1}{\sqrt T}\|\widehat V_{-r}^c\widehat V_{-r}^{c\,\prime}F^c\alpha_a\|+\frac{1}{\sqrt T}\|\widehat V_{-r}^c\widehat V_{-r}^{c\,\prime}\widetilde\varepsilon_a\|\cr
&\leq& O_P(\frac{1}{\sqrt T})\|\widehat V_{-r}^{c\,\prime}F^c\| +O_P(T^{-1/2})\cr
&=& O_P\bigl(T^{\varepsilon}(\nu_M^{-2}+T^{-1/2})\bigr)\quad\text{for every }\varepsilon>0,
\end{eqnarray}
where $\|\widehat V_{-r}^{c\,\prime}F^c\|=O_P(T^{\varepsilon}\sqrt{T}\,\nu_M^{-2})$ by Theorem~\ref{thm:main}(iii).
Also $\frac{1}{\sqrt T}\|P_{F^c}\widetilde\varepsilon_a\|= O_P(T^{-1/2})$. Hence, for every $\varepsilon>0$ and $a\in\{y,g,z\}$,
$\frac{1}{\sqrt T}\|\widetilde\varepsilon_a-\widehat\varepsilon_a\|= O_P(\nu_M^{-1})+O_P\bigl(T^{\varepsilon}(\nu_M^{-2}+T^{-1/2})\bigr)$. Then, since $\sqrt T=o(\nu_M^{2})$ and $\nu_M\gg T^{1/4}$ --- both implied by Assumption~\ref{ass4.1}(ii) ---
$$\frac{1}{\sqrt T}\|\widetilde\varepsilon_a-\widehat\varepsilon_a\|^{2}= O_P(\sqrt T\,\nu_M^{-2})+O_P\bigl(T^{2\varepsilon}(\sqrt T\,\nu_M^{-4}+T^{-1/2})\bigr)=o_P(1),$$
upon choosing $\varepsilon<1/4$.

\smallskip\textbf{Step 2: cross-product approximation.} We show $\frac{1}{\sqrt T}\widetilde\varepsilon_a'(\widetilde\varepsilon_b-\widehat\varepsilon_b)=o_P(1)$ for every $a,b\in\{y,g,z\}$. Decompose
\begin{eqnarray*}
\frac{1}{\sqrt T}\widetilde\varepsilon_a'(\widetilde\varepsilon_b-\widehat\varepsilon_b)&=&(I)+(II)\cr
(I)&=&\frac{1}{\sqrt T}\widetilde\varepsilon_a'(\widehat V_r^c\widehat V_r^{c\,\prime}-V_r^cV_r^{c\,\prime})F^c\alpha_b\cr
(II)&=&\frac{1}{\sqrt T}\widetilde\varepsilon_a'\widehat V_r^c\widehat V_r^{c\,\prime}\widetilde\varepsilon_b-\frac{1}{\sqrt T}\widetilde\varepsilon_a'V_r^cV_r^{c\,\prime}\widetilde\varepsilon_b+\frac{1}{\sqrt T}\widetilde\varepsilon_a'\widehat V_{-r}^c\widehat V_{-r}^{c\,\prime}\widetilde B+\frac{1}{\sqrt T}\widetilde\varepsilon_a'P_{F^c}\widetilde\varepsilon_b
\end{eqnarray*}
where $\widetilde B\in \{\widetilde Y, \widetilde G, \widetilde Z\}$.  
Let $J:=\frac{1}{\sqrt T}(\widehat V_r^c\widehat V_r^{c\,\prime}-V_r^cV_r^{c\,\prime})F^c\alpha_b$. By~\eqref{eqmainrA}, $\|J\|=O_P(\nu_M^{-1})$, so
$
\E[(I)^2\mid F,U]=J'\E[\widetilde\varepsilon_a\widetilde\varepsilon_a'\mid F,U]J\le C\|J\|^{2}=O_P(\nu_M^{-2}),
$
hence $(I)=o_P(1)$. For $(II)$, using $(\widehat V_r^c\widehat V_r^{c\,\prime})^{2}=\widehat V_r^c\widehat V_r^{c\,\prime}$ and~\eqref{equvvA},
\[
\Bigl|\frac{1}{\sqrt T}\widetilde\varepsilon_a'\widehat V_{-r}^c\widehat V_{-r}^{c\,\prime}\widetilde B\Bigr|\le\Bigl\|\frac{1}{\sqrt T}\widetilde\varepsilon_a'\widehat V_{-r}^c\widehat V_{-r}^{c\,\prime}\Bigr\|\,\|\widehat V_{-r}^c\widehat V_{-r}^{c\,\prime}\widetilde B\|\prec T^{-1/2}(1+T^{1/2}\nu_M^{-2})=o_P(1).
\]
For the remaining three projection terms, conditional on $(F,U)$ each $P_A$, $A\in\{\widehat V_r^c,V_r^c,F^c\}$, is a projection of fixed rank. Independence across $t$, the bounded fourth moments implied by Assumption~\ref{ass4.1}(iii), and the standard second-moment formula for a fixed-rank bilinear form give
\[
\E\!\left[\left.
\left|\frac{1}{\sqrt T}\widetilde\varepsilon_a'P_A\widetilde\varepsilon_b\right|^2
\right|F,U\right]\le \frac{C}{T}\{\operatorname{rank}(A)+\operatorname{rank}(A)^2\}
=O(T^{-1}).
\]
Thus each projection term is $O_P(T^{-1/2})$, so $(II)=o_P(1)$. Combining,
\begin{equation}\label{eq:crossiv}
\frac{1}{\sqrt T}\widehat\varepsilon_a'\widehat\varepsilon_b=\frac{1}{\sqrt T}\widetilde\varepsilon_a'\widetilde\varepsilon_b+o_P(1)=\frac{1}{\sqrt T}\varepsilon_a'\varepsilon_b+o_P(1),\qquad a,b\in\{y,g,z\},
\end{equation}
where the last equality uses the $O_P(T^{-1/2})$ rank-one correction noted in the preamble.

\smallskip\textbf{Step 3: asymptotic distribution.} Since $\varepsilon_y=\beta\varepsilon_g+\eta$, applying~\eqref{eq:crossiv} with $(a,b)=(z,y)$ and $(a,b)=(z,g)$,
\[
\frac{1}{\sqrt T}\widehat\varepsilon_z'\widehat\varepsilon_y-\beta\,\frac{1}{\sqrt T}\widehat\varepsilon_z'\widehat\varepsilon_g
=\frac{1}{\sqrt T}\sum_{t=1}^T\varepsilon_{z,t}\eta_t+o_P(1),
\]
and $T^{-1}\widehat\varepsilon_z'\widehat\varepsilon_g=T^{-1}\varepsilon_z'\varepsilon_g+o_P(1)\to_p\gamma$ by the law of large numbers and the relevance condition $\gamma=\E[\varepsilon_{g,t}\varepsilon_{z,t}]\ne 0$ in Assumption~\ref{ass4.1}(i). Hence
\begin{equation}\label{eq:lin-iv}
\sqrt T(\widehat\beta-\beta)=\gamma^{-1}\,\frac{1}{\sqrt T}\sum_{t=1}^T\varepsilon_{z,t}\eta_t+o_P(1).
\end{equation}
By the exclusion restriction $\E[\eta_t\varepsilon_{z,t}]=0$ in Assumption~\ref{ass4.1}(i), the i.i.d.\ score $\varepsilon_{z,t}\eta_t$ has mean zero, and by Assumption~\ref{ass4.1}(iii) and H\"older $\E[\varepsilon_{z,t}^{2}\eta_t^{2}]\le(\E\varepsilon_{z,t}^{4})^{1/2}(\E\eta_t^{4})^{1/2}<\infty$. Lyapunov's central limit theorem (using the eighth moment in (iii) for the Lyapunov ratio) yields
\[
T^{-1/2}\sum_{t=1}^T\varepsilon_{z,t}\eta_t\to^d\mathcal N\bigl(0,\E[\varepsilon_{z,t}^{2}\eta_t^{2}]\bigr),
\]
so $\sqrt T\sigma^{-1}(\widehat\beta-\beta)\to^d\mathcal N(0,1)$, with $\sigma^{2}=\gamma^{-2}\,\E[\varepsilon_{z,t}^{2}\eta_t^{2}]>0$ as defined in Theorem~\ref{thinf}.

\smallskip\textbf{Step 4: Entrywise bound.} For the consistency of the HC$_0$ variance estimator, it will be useful to prove the following result. 
For a ``tall" matrix $v$ whose dimension is $T\times l$ with $l=O(1)$ with $v_t'$ as the $t$ th row,   define $\|v\|_{\infty}=\max_{t\leq T}\|v_t\|$. 
We   prove $\|\widehat\varepsilon_{a}-\widetilde \varepsilon_{a}\|_{\infty} = o_P(1)$.   From (\ref{eqd.0}), 
$$
\|\widetilde\varepsilon_a-\widehat\varepsilon_a\|_{\infty}
\leq  \|(\widehat V_r^c\widehat V_r^{c\,\prime}-V_r^cV_r^{c\,\prime})\widetilde A\|_{\infty}
+\|\widehat V_{-r}^c\widehat V_{-r}^{c\,\prime}\widetilde A\|_{\infty}+\|P_{F^c}\widetilde\varepsilon_a\|_{\infty}.
$$
By Assumption~\ref{aspt:X}, $\|F^c\|_{\infty}\leq \max_{t\le T}\|f_t\|+\|\bar f\|=O_P(\log T)$. Since $T^{-1}F^{c\,\prime}F^c$ is well conditioned by the argument above, an orthonormal basis can be written as $V_r^c=T^{-1/2}F^cQ_c$ with $\|Q_c\|=O_P(1)$, and therefore
$\|V_r^c\|_{\infty}=O_P(T^{-1/2}\log T)$.  Also, the within-signal eigengap in Assumption~\ref{aspt:X}(iii) permits the entrywise control of the first $r$ eigenvectors (e.g., Theorem 1 of \cite{fan2021recent}; in their notation $\eta_N=O_P(\frac{N}{\sqrt{T}}),$ $c_N=\sqrt{\frac{\log N}{T}}$, $g_N=\nu_M\sqrt{T}$, $\|S\|=O(1)$, $\|S\xi_d\|_\infty\leq \|S\|_1\|\xi_d\|_{\infty}\leq O_P(\nu_M^{-1})\|B\|_{\infty}=O_P(\nu_M^{-1})$). After the same Procrustes alignment used above,
$$
\|\widehat V_r^c-V_r^c\|_{\infty}
= O_P\left(\Bigl(\nu_M^{-2}+\sqrt{\frac{\log N}{T}}\nu_M^{-1}\Bigr)\log T\right).
$$ The factor $\log T$ enters through the incoherence level $\|V_r^c\|_{\infty}=O_P(T^{-1/2}\log T)$ of the true right singular vectors. Then
\begin{eqnarray*}
  &&  \|(\widehat V_r^c\widehat V_r^{c\,\prime}-V_r^cV_r^{c\,\prime})\widetilde A\|_{\infty}\leq \|\widehat V_r^c-V_r^c\|_{\infty}\|\widehat V_r^{c'}\widetilde A\| +\|\widehat V_r^c\|_{\infty} \|(\widehat V_r^c-V_r^c)'\widetilde A\|\cr 
    &\leq&  \|\widehat V_r^c-V_r^c\|_{\infty}O_P(\sqrt{T})
    + (\|\widehat V_r^c-V_r^c\|_{\infty}+\|V_r^c\|_{\infty})\|\widehat V_r^c-V_r^c\| O_P(\sqrt{T})\cr 
    &=& O_P\bigl((\nu_M^{-2}\sqrt{T}+ \sqrt{\log N}\, \nu_M^{-1})\log T\bigr)=o_P(1).
\end{eqnarray*}
Here the second line uses $\|V_r^c\|_{\infty}\|\widehat V_r^c-V_r^c\|O_P(\sqrt{T})=O_P(T^{-1/2}\log T\cdot\nu_M^{-1}\cdot\sqrt{T})=O_P(\nu_M^{-1}\log T)$, which is absorbed by the second term, and the final $o_P(1)$ follows from the rate condition $\sqrt{T}\log T=o(\nu_M^{2})$ in Assumption~\ref{ass4.1}(ii) together with $\nu_M\gg T^{1/4}\gg\sqrt{\log N}\,\log T$, the first relation as already noted in Step~1.

Next, by the incoherency of overestimated eigenvectors (Theorem~\ref{thm:main}(ii)),
$$
\|\widehat V_{-r}^c\widehat V_{-r}^{c\,\prime}\widetilde A\|_{\infty}
\leq \|\widehat V_{-r}^c\|_{\infty}\|\widehat V_{-r}^{c\,\prime}\widetilde A\|\prec \nu_M^{-1}(1+\sqrt{T}\nu_M^{-2})=o_P(1).
$$
Finally, $\|P_{F^c}\widetilde\varepsilon_a\|_{\infty}\leq O_P(T^{-1})\|F^c\|_{\infty}\|F^{c'}\widetilde\varepsilon_a\|=O_P(T^{-1/2}\log T)=o_P(1)$, using again $\|F^c\|_{\infty}=O_P(\log T)$.

This proves $\|\widetilde\varepsilon_a-\widehat\varepsilon_a\|_{\infty}=o_P(1)$.

\smallskip\textbf{Step 5: HC$_0$ sandwich consistency.} It remains to show that the Eicker--White (HC$_0$) estimator $\widehat\sigma^{2}=(\widehat\varepsilon_z'\widehat\varepsilon_g/T)^{-2}\,T^{-1}\sum_t\widehat\varepsilon_{z,t}^{2}\widehat\eta_t^{2}$ defined in Theorem~\ref{thinf} satisfies $\widehat\sigma^{2}=\sigma^{2}+o_P(1)$.  The outer factor satisfies $\widehat\varepsilon_z'\widehat\varepsilon_g/T\to_p\gamma$ by Step~3. For the meat, write $\widetilde\eta_t:=\widetilde\varepsilon_{y,t}-\beta\widetilde\varepsilon_{g,t}$ so that $\widehat\eta_t=\widehat\varepsilon_{y,t}-\widehat\beta\widehat\varepsilon_{g,t}=\widetilde\eta_t+(\widehat\varepsilon_{y,t}-\widetilde\varepsilon_{y,t})-\widehat\beta(\widehat\varepsilon_{g,t}-\widetilde\varepsilon_{g,t})+(\beta-\widehat\beta)\widetilde\varepsilon_{g,t}$.
Note that by Assumption \ref{ass4.1} (iii), $\max_{t\leq T}|\varepsilon_{g,t}|=O_p(T^{1/8})$ while $|\beta-\widehat{\beta}|=O_p(T^{-1/2})$.
The $\|\cdot\|_\infty$ consistency in 
Step 4 implies $\|\widehat\eta -\widetilde\eta\|_\infty=o_P(1)$. Also, 
\[
T^{-1}\sum_t\widehat\varepsilon_{z,t}^{2}\widehat\eta_t^{2}=T^{-1}\sum_t\widetilde\varepsilon_{z,t}^{2}\widetilde\eta_t^{2}+R_T,
\]
where $R_T$ collects cross-product terms of the form $T^{-1}\sum_t(\widehat\varepsilon_{z,t}-\widetilde\varepsilon_{z,t})\widehat\varepsilon_{z,t}\widehat\eta_t^{2}$, $T^{-1}\sum_t\widetilde\varepsilon_{z,t}^{2}(\widehat\eta_t-\widetilde\eta_t)\widehat\eta_t$, etc., together with the $(\widehat\beta-\beta)$ contribution which is $O_P(T^{-1/2})$ by Step~3 and is multiplied by $T^{-1}\sum_t\widetilde\varepsilon_{z,t}^{2}|\widetilde\varepsilon_{g,t}|=O_P(1)$. By Cauchy--Schwarz and Step~1,
\[
|R_T|\le\Bigl(T^{-1}\sum_t(\widehat\varepsilon_{z,t}-\widetilde\varepsilon_{z,t})^{2}\Bigr)^{1/2}\Bigl(T^{-1}\sum_t\widehat\varepsilon_{z,t}^{2}\widehat\eta_t^{4}\Bigr)^{1/2}+(\text{symmetric terms})+o_P(1).
\]
The first factor  $T^{-1}\sum_t(\widehat\varepsilon_{z,t}-\widetilde\varepsilon_{z,t})^{2}=o_P(1)$ by Step~1. For the second factor, by the elementary inequality $\widehat\varepsilon_{z,t}^{2}\le 2\varepsilon_{z,t}^{2}+2(\widehat\varepsilon_{z,t}-\varepsilon_{z,t})^{2}$ and $\widehat\eta_t^{4}\le 8\eta_t^{4}+8(\widehat\eta_t-\eta_t)^{4}$,
\begin{align*}
T^{-1}\sum_t\widehat\varepsilon_{z,t}^{2}\widehat\eta_t^{4}\;&\le\;C\,T^{-1}\sum_t\varepsilon_{z,t}^{2}\eta_t^{4}+C\,T^{-1}\sum_t\varepsilon_{z,t}^{2}(\widehat\eta_t-\eta_t)^{4}\\
&\quad+C\,T^{-1}\sum_t(\widehat\varepsilon_{z,t}-\varepsilon_{z,t})^{2}\eta_t^{4}+C\,T^{-1}\sum_t(\widehat\varepsilon_{z,t}-\varepsilon_{z,t})^{2}(\widehat\eta_t-\eta_t)^{4}.
\end{align*}
The first term is $O_P(1)$ by the law of large numbers and Assumption~\ref{ass4.1}(iii) ($\E[\varepsilon_{z,t}^{2}\eta_t^{4}]\le(\E\varepsilon_{z,t}^{4})^{1/2}(\E\eta_t^{8})^{1/2}<\infty$ by Cauchy--Schwarz). All the other terms are $o_P(1)$ by the $\|\cdot\|_\infty$-consistency of $\widehat\epsilon_a,\widehat\eta$.  Hence $T^{-1}\sum_t\widehat\varepsilon_{z,t}^{2}\widehat\eta_t^{4}=O_P(1)$. Thus $R_T=o_P(1)$, and using $T^{-1}\sum_t\widetilde\varepsilon_{z,t}^{2}\widetilde\eta_t^{2}=T^{-1}\sum_t\varepsilon_{z,t}^{2}\eta_t^{2}+O_P(T^{-1/2})$,
\[
T^{-1}\sum_t\widehat\varepsilon_{z,t}^{2}\widehat\eta_t^{2}=\E[\varepsilon_{z,t}^{2}\eta_t^{2}]+o_P(1).
\]
Combining, $\widehat\sigma^{2}=\sigma^{2}+o_P(1)$, and hence $\sqrt T\widehat\sigma^{-1}(\widehat\beta-\beta)\to^d\mathcal N(0,1)$, as claimed.

\smallskip\textbf{OLS as the special case $z_t=g_t$.} Setting $z_t=g_t$ throughout collapses $\varepsilon_z\equiv\varepsilon_g$ and $\widehat\varepsilon_z\equiv\widehat\varepsilon_g$, so $\widehat\beta=(\widehat\varepsilon_g'\widehat\varepsilon_g)^{-1}\widehat\varepsilon_g'\widehat\varepsilon_y$ is the partialled-out OLS estimator; the corresponding specializations of $\gamma$, the exclusion restriction, and $\sigma^2$, recorded after Assumption~\ref{ass4.1} and in Theorem~\ref{thinf}, follow by direct substitution.

\smallskip\textbf{$r=0$ as a special case.} When $r=0$,
$
\widetilde\varepsilon_a-\widehat\varepsilon_a
= P_{\widehat F^c} \widetilde A 
=\widehat V_{-r}^c\widehat V_{-r}^{c\,\prime}\widetilde \varepsilon_a. 
$
\begin{eqnarray} 
 \frac{1}{\sqrt T}\|\widehat V_{-r}^c\widehat V_{-r}^{c\,\prime}\widetilde \varepsilon_a\|
&=& O_P(T^{-1/2}).
\end{eqnarray}
So $\frac{1}{\sqrt T}\|\widetilde\varepsilon_a-\widehat\varepsilon_a\|^{2} =o_P(1).$ Meanwhile
$
\frac{1}{\sqrt T}\widetilde\varepsilon_a'(\widetilde\varepsilon_b-\widehat\varepsilon_b)=\frac{1}{\sqrt T}\widetilde\varepsilon_a' \widehat V_{-r}^c\widehat V_{-r}^{c\,\prime}\widetilde \varepsilon_b= O_P(T^{-1/2}).
$ This implies 
$$
\frac{1}{\sqrt T}\widehat\varepsilon_a'\widehat\varepsilon_b= \frac{1}{\sqrt T}\varepsilon_a'\varepsilon_b+o_P(1),\qquad a,b\in\{y,g,z\}.
$$
Moreover, Proposition~\ref{prop:loallaw}(3), applied jointly to the $T$ standard basis vectors and $\zeta=T^{-1/2}1_T$, followed by the same fixed-dimensional demeaning and QR argument used in the preamble, gives $\|\widehat V_{-r}^c\|_\infty\prec T^{-1/2}$. Hence
\[
\|\widetilde\varepsilon_a-\widehat\varepsilon_a\|_\infty
\le \|\widehat V_{-r}^c\|_\infty
\|\widehat V_{-r}^{c\,\prime}\widetilde\varepsilon_a\|=o_P(1).
\]
The asymptotic-normality and sandwich-consistency arguments in Steps~3--5 therefore apply unchanged.

\end{proof}

\newpage

\bibliographystyle{chicago}
\bibliography{liao_bib}

\end{document}